\documentclass[12pt]{amsart}
\usepackage{amsmath, amssymb, amsthm}
\usepackage{comment, empheq}
\usepackage{mathdots, breqn}
\usepackage{graphicx, microtype}
\usepackage{hyperref, mathtools}
\numberwithin{equation}{section}
\theoremstyle{plain}
\newtheorem{theorem}{Theorem}
\newtheorem{lemma}{Lemma}[section]

\newtheorem*{toythm}{Toy Theorem}
\newtheorem{prop}{Proposition}[section]

\theoremstyle{definition}
\newtheorem{defn}{Definition}[section]
\theoremstyle{remark}
\newtheorem{remark}{Remark}
\newtheorem*{ackno}{Acknowledgements}
\def\Ad{\operatorname{Ad}}
\usepackage[usenames,dvipsnames]{color}

\newcommand{\pn}[1]{{\color{ForestGreen} \sf PN: [#1]}}

\newcommand{\R}{\mathbb{R}}

\newcommand{\C}{\mathbb{C}}
\def\O{\operatorname{O}}
\newcommand{\Q}{\mathbb{Q}}
\newcommand{\Z}{\mathbb{Z}}
\newcommand{\N}{\mathbb{N}}

\newcommand{\B}{\mathcal{B}}
\newcommand{\D}{\mathcal{D}}
\def\vol{\operatorname{vol}}
\newcommand{\W}{\mathcal{W}}
\newcommand{\GL}{\mathrm{GL}}
\newcommand{\g}{\mathfrak{g}}

\newcommand{\We}{V}
\newcommand{\V}{W}
\newcommand{\PGL}{\mathrm{PGL}}

\newcommand{\Tr}{\mathrm{Tr}}

\newcommand{\sgn}{\mathrm{sgn}} 

\newcommand{\diag}{\mathrm{diag}}
\newcommand{\pop}{\mathbf{pop}}
\newcommand{\res}{\mathrm{res}}

\title{Analytic newvectors for $\mathrm{GL}_n(\mathbb{R})$}
\author{Subhajit Jana and Paul D. Nelson}
\address{Max Planck Institute for Mathematics, Vivatgasse 7, 53111 Bonn, Germany.}
\email{subhajit@mpim-bonn.mpg.de}

\address{ETH Z\"urich, R\"amistrasse 101, 8092 Z\"urich, Switzerland.}
\email{paul.nelson@math.ethz.ch}
\date{}

\begin{document}

\begin{abstract}
We relate
the analytic conductor of a generic irreducible
representation of $\mathrm{GL}_n(\mathbb{R})$
to the invariance properties of vectors
in that representation.
The relationship is an analytic archimedean analogue
of some aspects
of
the classical non-archimedean newvector theory of
Casselman and Jacquet--Piatetski-Shapiro--Shalika.
We illustrate
how this relationship
may be applied
in trace formulas
to majorize sums
over
automorphic forms
on $\mathrm{PGL}_n(\mathbb{Z}) \backslash \mathrm{PGL}_n(\mathbb{R})$
ordered by analytic conductor.




\end{abstract}

\maketitle

\section{Introduction}
Let $F$ be a local field, and let $\pi$ be a generic irreducible
representation of $\GL_n(F)$.  

Suppose
first that $F$ is non-archimedean,
with ring of integers $\mathfrak{o}$,
maximal ideal $\mathfrak{p}$,
and residue field cardinality $q = \# \mathfrak{o}/\mathfrak{p}$.
The newvector theory developed by Casselman and
Jacquet--Piatetski-Shapiro--Shalika
(see \cite{C,JPSS1}
and \cite{M})
then
gives a precise
relationship between the conductor
$C(\pi)$ of $\pi$ and the existence of vectors 
$v$ in $\pi$
invariant under certain compact open subgroups.
Here
we define
$C(\pi)$ in terms of the local $\epsilon$-factor attached to $\pi$
with respect to an unramified additive
character $\psi$ of $F$
by the relation
\begin{equation}\label{eqn:abs-gamma-is-C}
\epsilon(1/2-s,\pi,\psi) = \epsilon(1/2,\pi,\psi) C(\pi)^s,
\end{equation}
so that if $c(\pi) \in \mathbb{Z}_{\geq 0}$ denotes
the conductor exponent,
then $C(\pi) = q^{c(\pi)}$.
The main theorem of newvector theory
says that
$C(\pi)$ is the smallest element 
$X = q^n \geq 1$ of the value group
of $F$ for which $\pi$ contains a nonzero vector $v$ invariant
by the group
\begin{equation}\label{eqn:non-arch-K1-defn}
K_1(X) := \left\{
g \in \GL_n(\mathfrak{o}) :
\begin{gathered}
\text{$|g_{n j}| \leq 1/X$ for $1 \leq j < n$}, \\
\,
|g_{n n} - 1| \leq 1/X
\end{gathered}
\right\}
=
\begin{pmatrix}
\mathfrak{o} & \dotsb & \mathfrak{o} & \mathfrak{o} \\
\dotsb & \dotsb & \dotsb & \dotsb \\
\mathfrak{o} & \dotsb & \mathfrak{o} & \mathfrak{o} \\
\mathfrak{p}^n & \dotsb & \mathfrak{p}^n & 1 + \mathfrak{p}^n \\
\end{pmatrix}
\cap \GL_n(\mathfrak{o})
\end{equation}
or equivalently,
transforming under the group
\begin{equation}\label{eqn:non-arch-K0-defn}
K_0(X) := \left\{
g \in \GL_n(\mathfrak{o}) :
\text{$|g_{n j}| \leq 1/X$ for $1 \leq j < n$}
\right\}
=
\begin{pmatrix}
\mathfrak{o} & \dotsb & \mathfrak{o} & \mathfrak{o} \\
\dotsb & \dotsb & \dotsb & \dotsb \\
\mathfrak{o} & \dotsb & \mathfrak{o} & \mathfrak{o} \\
\mathfrak{p}^n & \dotsb & \mathfrak{p}^n &
\mathfrak{o} \\
\end{pmatrix}
\cap \GL_n(\mathfrak{o})
\end{equation}
by the character
$g \mapsto \omega_{\pi}(g_{nn})$,
with $\omega_\pi$ the central character
of $\pi$.
The vector $v$,
called the \emph{newvector}
(or,
in the language of \cite{JPSS1}, the \textit{essential vector}), is 
then uniquely determined up to a
scalar.

Suppose now that
$F = \mathbb{R}$.
(A similar discussion
should apply when $F = \mathbb{C}$,
but we focus
in this paper on the real case.)
The analytic conductor \cite{IS}
of $\pi$
is then typically defined,
for the sake of concreteness,
by
the the formula
$C(\pi) := \prod_j (1 + |\mu_j|)$,
where $\{\mu_1,\dotsc,\mu_n\}$
is the set of the parameters of $\pi$
characterized by the relation
$L(s,\pi)
= \prod_j \Gamma_{\mathbb{R}}(s+\mu_j)$,
$\Gamma_{\mathbb{R}}(s) := \pi^{-s/2} \Gamma(s/2)$.
In practice, minor variants
of this definition
serve the same purpose;
for instance, one occasionally
sees
the factors 
$(1 + |\mu_j|)$
replaced by
$(3 + |\mu_j|)$,
so that $\log C(\pi)$ is bounded uniformly
away from zero.
By Stirling's formula, one has
$$\gamma(1/2-s,\pi)=\gamma(1/2,\pi)C(\pi)^s+O_\pi(s),\quad s\to 0,$$
where $\gamma$ denotes the local $\gamma$-factor attached to $\pi$.
(cf. \eqref{eqn:abs-gamma-is-C})
The above equation may be understood as
an asymptotic characterization
of $C(\pi)$ and its mild variants.

The classical non-archimedean newvector theory has been applied towards many problems in the analytic theory of automorphic forms.
For instance, to apply trace
formulas (or relative trace formulas, integral representations,
...)  to study averages over a family of automorphic forms of
various quantities of interest (Fourier coefficients, Hecke
eigenvalues, $L$-values, ...), 
one needs a test function that
approximately projects to that family, so that
the spectral side of the trace formula localizes to that family.
For a level aspect family
of forms having spectral parameter
bounded and finite conductor dividing
a given large natural number,
one can use the 
(twisted) characteristic functions of
the congruence subgroups
\eqref{eqn:non-arch-K1-defn}
and
\eqref{eqn:non-arch-K0-defn} to construct 
suitable projectors.
We mention the works
\cite{B2, BBM, V, MV} which make (often implicit) use of such ideas.

By comparison, we are not aware
of an existing comparably simple way
to project (approximately)
onto (say) the family of cusp forms
on $\PGL_n(\mathbb{Z}) \backslash \PGL_n(\mathbb{R})$
whose analytic conductor at the real place
is bounded from above
by a certain quantity.
One approach 
would be
to describe
that family in terms
of local parameters,
partition it 
into subfamilies according
to the approximate values
of those parameters,
and then sum the projectors
associated to the subfamilies;
see
\cite[\S8]{BrM}.
However, this approach 
is
significantly less direct
than
in the non-archimedean
case, and does not
immediately clarify the shape of the
test function defining
the projection.


This discrepancy 
motivates looking
for an archimedean analogue
to classical newvector theory,
that is to say,
an interpretation
of the analytic conductor $C(\pi)$
in terms of the invariance properties
of vectors in a generic representation
$\pi$ of
$\GL_n(\mathbb{R})$.
One
might hope for such an interpretation
to be useful in analytic problems
in which $\pi$ varies in archimedean aspects
in much the same way that
newvector theory
has been useful in non-archimedean aspects.

An exact analogue is clearly
too ambitious:
the group $\GL_n(\mathbb{R})$
has \emph{no} compact open subgroups,
and very few vectors 
\emph{exactly} invariant
by \emph{any}
open subset of $\GL_n(\R)$.
One can ask instead for \emph{approximate}
analogues.

The main purpose of this paper
is to introduce,
for generic irreducible representations
$\pi$ of $\GL_n(\mathbb{R})$,
a
relationship
between the
analytic conductor $C(\pi)$ and the existence of vectors
satisfying a form of approximate invariance under certain
subsets closely related to the subgroups
$K_1(X)$ and $K_0(X)$
arising in the non-archimedean case.


\subsection{The case \texorpdfstring{$n=1$}{}}
As warmup for the 
more complicated statements
given below,
let us consider in detail
what ``analytic newvector theory''
looks like in the simplest case $n=1$.

We recall first the non-archimedean story,
which we elect
to present classically in terms
of Dirichlet characters
(rather than the
closely-related characters of $\GL_1(\mathbb{Q}_p)
= \mathbb{Q}_p^\times$).
Let $\chi$ be a Dirichlet character.
There are then 
\emph{two}
equivalent
ways to define the \emph{conductor} $q
= q(\chi)$:
\begin{itemize}
\item 
(in terms of the invariance)
$\chi$ has conductor
$q$ if it
is well-defined on
$\mathbb{Z}/q \mathbb{Z}$
and if the following conditions
on a divisor $d$ of $q$ are equivalent:
\begin{itemize}
    \item $d = q$
    \item
    For all integers $n$ with
    $n - 1 \equiv 0 (d)$,
    we have $\chi(n) = 1$.
\end{itemize}
\item
(analytically)
$q = C(\chi)$,
defined using local $\gamma$-factors
(or Gauss sums) as above.
\end{itemize}

Consider now
a character $\chi$ of $\GL_1(\mathbb{R})$,
say
$\chi = |.|^{i t}$ for some real number $t$.
As noted above,
the analytic conductor
is defined by $C(\chi) = (1 + |t|)$,
and admits an asymptotic analytic
characterization via the local $\gamma$-factors.
How should one interpret $C(\chi)$
in terms of invariance, by analogy
to the second of the two characterizations
of $q(\chi)$ above?
A natural interpretation
is that $\chi$
is
\emph{approximately invariant}
under group elements
$y \in \GL_1(\mathbb{R})$
of the form $y = 1 + o(1/C(\chi))$,
but not in general
under those of the form
$y = 1 + \O(1/C(\chi))$.
More precisely,
we have the following
asymptotic characterization
of the analytic conductor
$(1 + |t|)$ of $|.|^{i t}$
up to bounded multiplicative error:
\begin{toythm}\label{toy-theorem}
For
a sequence of real
numbers $t_j$
tending off to $\infty$
and a corresponding
sequence of positive real
scaling parameters
$X_j$,
the following are equivalent:
\begin{itemize}
\item 
    $X_j / (1 + |t_j|) \rightarrow \infty$.
\item
    For all sequences
    $y_j \in \GL_1(\mathbb{R})$
    with $|y_j - 1| < 1/X_j$,
    we have
    $|y_j|^{i t_j} \rightarrow 1$.
\end{itemize}
\end{toythm}
The straightforward
proof is left to the reader.
Note the similarity
between these conditions and those
appearing above in the characterizations
of the conductor of a Dirichlet character,
with $d$ playing the role of $X_j$
and $n$ that of $y_j$.



\subsection{Notation and preliminaries}
We now prepare to describe our main results.
We denote the identity element of $\GL_r(\mathbb{R})$
by $1_r$.
When the dimension $r$ is clear from context,
we abbreviate $1 := 1_r$.


\subsubsection{Standard congruence subsets}
Let $X \geq 1$
(thought of as tending off to $\infty$)
and let $\tau \in (0,1)$
(thought of as small but fixed,
or perhaps very slowly
tending to $0$).
We define the
following
archimedean
analogues
of
the standard $p$-adic congruence subgroup
\eqref{eqn:non-arch-K0-defn}:
\begin{equation}\label{defn-congruence-subgroup}
    K_0(X,\tau):=
    \left\{\begin{pmatrix}
    a&b\\c&d
    \end{pmatrix}\in\mathrm{GL}_{n}(\R)\middle|
    a \in \GL_{n-1}(\R),
    d \in \GL_1(\R),
    \begin{aligned}
    &|a - 1_{n-1}| < \tau,
    \quad
    |b|<\tau, \\
    &|c|<\frac{\tau}{X},
    \quad
    |d-1|<\tau
    \end{aligned}
    \right\}.
\end{equation}
Here the various $|.|$ denote arbitrary fixed
norms on the various spaces of matrices.
We define $K_1(X,\tau)$
similarly,
but with the stronger constraint
$|d-1| < \tau/X$.

While the sets $K_*(X,\tau)$
are not groups,
they have some group-like
properties
in the $\tau \rightarrow 0$ limit.
For instance, it is
easy to see that if
$\tau'$ is small enough with respect to $\tau$,
then $g_1 g_2 \in K_0(X, \tau)$
for all $g_1, g_2 \in K_0(X,\tau')$.
Moreover,
these subsets enjoy the following
F{\o}lner-type property:

\begin{lemma}\label{folner-lemma}
Let $* \in \{0,1\}$.
For all $\tau_0, \delta \in (0,1)$
there exists $\tau_1 > 0$
so that for all $X \geq 1$
and all $g \in K_*(X,\tau_1)$,
the set $A := K_*(X, \tau_0)$
enjoys the following approximate
invariance property under translation by $g$:
\[
\frac{
\vol ( g A \triangle A )
}{
\vol ( A )
}
< \delta.
\]
Here
$A \triangle B
:= (A \setminus B) \cup (B \setminus A)$
denotes the symmetric difference
and $\vol$ is taken with respect
to any fixed Haar measure.
\end{lemma}

\begin{proof}
By direct matrix multiplication,
we see
that for
any $\nu \in (\tau_0, 1)$
there exists
$\tau_1 \in (0,1)$
so that
$K_*(X,\tau_1) \cdot A \subseteq K_*(X, \nu)$.
Thus
$$\frac{\vol(gA\Delta A)}{\vol(A)}\le \frac{\vol(K_0(X,\nu))-\vol(K_0(X,\tau_0))}{\vol(K_0(X,\tau_0))}.$$
On the other hand,
by
a straightforward
calculation
using the formula
$d g = |\det g|^{-n}
\prod_{i,j=1}^n d g_{i j}$
for the Haar measure
on $\GL_n(\mathbb{R})$
in terms
of the Lebesgue measure
on the coordinates,
we see that 
if $\nu$
is chosen close enough to $\tau_0$,
then
$$
 \frac{\vol(K_0(X,\nu))-\vol(K_0(X,\tau_0))}{\vol(K_0(X,\tau_0))}
 < \delta
$$
for all $X \geq 1$.
Combining these
two inequalities, we conclude.
\end{proof}

\subsubsection{$\theta$-temperedness}
Let $\theta \geq 0$.
By the Langlands classification, we 
know that any unitary irreducible representation $\pi$ of $\GL_{n}(\R)$ is a Langlands quotient of an isobaric sum of the form
$$\sigma_1\otimes|\det|^{s_1}\boxplus\dots\boxplus\sigma_r\otimes|\det|^{s_r}.$$
where the underlying Levi of the above induction is attached to a partition of $n$ by $2$'s and $1$'s.
Here each $\sigma_i$ is either a discrete series of $\GL_2(\R)$ or a character of $\GL_1(\R)$ of the form $\sgn^{\delta}|.|^{\mu_i}$ for some $\delta\in \{0,1\}$ and $\mu_i\in i\R$. We say that  $\pi$ is \textit{$\theta$-tempered} if all such $s_i$ have real parts in $[-\theta,\theta]$. 
By
\cite{MS}
the local component
at any real place of any cuspidal automorphic representation
of $\GL(n)$ over a number field
is $\theta$-tempered 
with $\theta = 1/2 - 1/(1+n^2) < 1/2$.

\subsection {Main results}
Recall that, for a generic irreducible
representation $\pi$ of $\GL_n(F)$
over a non-archimedean local field $F$,
the conductor $C(\pi)$ is the smallest
element $X$ of the value group
for which $\pi$ admits
a nonzero $K_1(X)$-invariant vector.
We may split this statement
into three parts:
\begin{itemize}
\item
(Existence) If $X \geq C(\pi)$,
then $\pi$ contains
a nonzero $K_1(X)$-invariant vector.
\item
(Non-existence)
If $X < C(\pi)$,
then $\pi$ does not contain
any nonzero $K_1(X)$-invariant vectors.
\item
(Uniqueness)
If $X = C(\pi)$, 
then the space of $K_1(X)$-invariant
vector in $\pi$ is one-dimensional.
\end{itemize}
Each of these assertions
admits an equivalent formulation
in terms of $K_0(X)$ and the central
character of $\pi$.

We will formulate and prove
analogues
in the archimedean setting
of each of these assertions.
We start with an ``existence''
result.
\begin{theorem}[Existence of analytic newvectors]\label{existence-weak}
Fix $n \in \mathbb{Z}_{\geq 1}$ and
$\theta \in [0,1/2)$.  For each $\delta>0$ there exists $\tau>0$ with
the following property:
For each generic irreducible $\theta$-tempered unitary
representation $\pi$ of $\GL_{n}(\mathbb{R})$, there exists
a unit vector $v \in \pi$
such that
for all $g\in K_0(C(\pi),\tau)$,
$$\|\pi(g) v- \omega_{\pi}(d_g) v\|<\delta.$$
Here $\omega_{\pi}$ denotes the central
character of $\pi$ and $d_g$ denotes the lower-right
entry of $g$.
The vector $v$ may be taken
independent of $\delta$.
\end{theorem}

Our next result
gives a precise sense in which the conclusion
of Theorem \ref{existence-weak}
would fail if
$C(\pi)$ were replaced
with a substantially smaller quantity.
\begin{theorem}[Non-existence
of analytic newvectors]\label{nonexistence-weak}
There exists $\delta > 0$ so that for each $\tau > 0$, there
exists $\mu > 0$ so that for every $\pi$ with
$\mu C(\pi) \geq 1$, there does not exist a unit vector
$v \in \pi$ satisfying $\|\pi(g) v - v\| < \delta$ for all
$g \in K_1(\mu C(\pi), \tau)$.
\end{theorem}

It is instructive to reformulate the above
results in terms of sequences.
\begin{defn}[Level of a sequence of vectors]\label{defn:level}
Let $\pi_j$
be a sequence
of generic $\theta$-tempered
irreducible unitary representations
of $\GL_{n}(\mathbb{R})$.
Let $v_j \in \pi_j$ be unit vectors,
and let $X_j > 0$.
We say that the sequence
$(v_j)$ has \emph{level} $(X_j)$
if for any
sequence $\tau_j$ of positive numbers
tending to zero
and any sequence of group elements
$g_j \in K_0(C(X_j), \tau_j)$,
we have
\[
\lim_{j \rightarrow \infty}
\| \pi(g_j) V_j - \omega_{\pi_j}(d_{g_j})
V_j \| = 0.
\]
\end{defn}

\begin{theorem}[Existence and uniqueness, sequential form]\label{thm:existi-unique-sequences}
Let $\pi_j$ be as in 
Definition \ref{defn:level}.
\begin{itemize}
\item (Existence)
There is a sequence $(v_j)$ of unit
vectors with level $(C(\pi_j))$.
\item (Non-existence)
Assume that $C(\pi_j) \rightarrow \infty$.
For any sequence $\mu_j$ positive numbers tending to zero,
there does \emph{not} exist
a sequence $(v_j)$ of unit vectors
with level $(\mu_j C(\pi_j))$.
\end{itemize}
\end{theorem}
\begin{proof}
This is a direct translation
of Theorems \ref{existence-weak}
and \ref{nonexistence-weak}.
Indeed, let us first verify
the existence assertion.
Let $v_j \in \pi_j$ be as furnished
by Theorem \ref{existence-weak},
and let $\tau_j, g_j$ be as in
Definition \ref{defn:level}.
We must check that
for each fixed $\delta > 0$, we have
the inequality
$\| \pi(g_j) V_j - \omega_{\pi_j}(d_{g_j})
V_j \| < \delta$
for large enough $j$.
Indeed, taking $\tau > 0$
as in
Theorem \ref{existence-weak},
the required inequality holds
whenever $\tau_j < \tau$.

Turning to non-existence,
suppose otherwise
that there are
$\mu_j > 0$ and $v_j \in \pi_j$
with
$\mu_j \rightarrow 0$
and $(v_j)$ of level $(\mu_j C(\pi_j))$.
Fix $\delta > 0$ as in Theorem \ref{nonexistence-weak}.
For each $\tau > 0$,
we choose $\mu(\tau) > 0$
so that the conclusion
of
Theorem \ref{nonexistence-weak}
holds.
Since $\mu_j \rightarrow 0$
and $C(\pi_j) \rightarrow \infty$,
we may construct a sequence
$\tau_j$
tending to zero sufficiently
slowly
that for large $j$,
we have
$\mu(\tau_j) > \mu_j$
and $\mu (\tau_j) C(\pi_j) \geq 1$.
Our hypotheses
imply that $(v_j)$ is of level
$(\mu_j C(\pi_j))$, hence also
of level $(\mu(\tau_j) C(\pi_j))$.
The hypothesized
existence of the vector $v_j \in \pi_j$
thus contradicts the conclusion
of Theorem \ref{nonexistence-weak}
for large $j$.
\end{proof}

We will informally
refer to any vector $v$ satisfying
the conclusion
of Theorem \ref{existence-weak}
as an analytic newvector.
This notion depends upon the parameters
$\delta$ and $\tau$, which we assume in practice
to be sufficiently small.
More formally,
we might refer to a sequence
$(v_j)$ of level $(C(\pi_j))$
as an ``analytic newvector sequence.''

Theorem \ref{thm:existi-unique-sequences}
may be understood
as characterizing the analytic conductor $C(\pi)$,
up to  bounded multiplicative error,
in terms of the invariance properties of vectors in $\pi$:
\begin{theorem}
Let $\mathcal{P}$
denote the set of isomorphism
classes of generic $\theta$-tempered
irreducible unitary representations
of $\GL_{n}(\mathbb{R})$.
Let $Q : \mathcal{P} \rightarrow \mathbb{R}_{\geq 1}$ be a function
with the following properties:
\begin{enumerate}
    \item For each
    sequence $\pi_j \in \mathcal{P}$,
    there is a sequence of unit vectors
    $v_j \in \pi_j$
    so that $(v_j)$
    has level $(Q(\pi_j))$.
    \item
    For each sequence $\pi_j \in \mathcal{P}$ with $C(\pi_j) \rightarrow \infty$
    and $\mu_j > 0$ with $\mu_j \rightarrow 0$,
    there does not exist
    a sequence
    of unit vectors
    $v_j \in \pi_j$
    for which $(v_j)$ has level
    $(\mu_j Q(\pi_j))$.
\end{enumerate}
Then there are constants
$c_2 > c_1 > 0$
and $c_0 \geq 1$
so that
for all $\pi \in \mathcal{P}$
with $C(\pi) \geq c_0$,
we have
\[
c_1 C(\pi) \leq Q(\pi) \leq c_2 C(\pi).
\]
\end{theorem}
\begin{proof}
  If the conclusion
  fails,
  then we may find a sequence
  $\pi_j \in \mathcal{P}$
  with $C(\pi_j) \rightarrow \infty$
  so that the ratio
  $\mu_j := Q(\pi_j)/C(\pi_j)$
  tends either to zero or to
  infinity.
  If it tends to infinity,
  so that $\mu_j^{-1} \rightarrow 0$,
  then there is no
  sequence $(v_j)$
  of level $(\mu_j^{-1} Q(\pi_j)) = C(\pi_j)$,
  contrary
  to the existence
  assertion of Theorem
  \ref{thm:existi-unique-sequences}.
  If it tends to zero,
  then there is a sequence
  $(v_j)$ of level $(Q(\pi_j)) = (\mu_j C(\pi_j))$,
  contrary to the non-existence
  assertion of Theorem
  \ref{thm:existi-unique-sequences}.
  In either case, we obtain
  the required contradiction.
\end{proof}

Theorem \ref{nonexistence-weak}
admits the following quantitative
refinement.
\begin{theorem}\label{nonexistence}
Let $0 \leq \theta < 1/2$,
$\tau > 0$,
$\delta > 0$.
Let $X > 1$ be sufficiently large.
Let $\pi$ 
be a generic $\theta$-tempered 
irreducible unitary representation
of $\GL_{n}(\mathbb{R})$.
Suppose that
there is a unit vector $v$ satisfying $\|\pi(g) v - v \| < \delta$ for all $g \in K_1(X,\tau)$.
Then
\[
C(\pi) \ll_{\tau,\sigma} X ( 1 - \delta)^{1/n \sigma},
\]
for any $0<\sigma<1/2-\theta$.
\end{theorem}

We finally
record a sense
in which analytic newvectors
are unique.
``Multiplicity one''
seems too much to hope 
for in the analytic setting,
so we settle for a form of
``bounded multiplicity'':
\begin{theorem}[Uniqueness]\label{uniqueness}
Fix $0\le \theta<1/2$ and $\tau>0$, and let $0<\delta<1$. Then for every generic irreducible unitary $\theta$-tempered representation $\pi$ of $\GL_{n}(\R)$ there can exist at most $O_\tau((1-\delta)^{-1})$ mutually orthogonal unit vectors $v\in\pi$ satisfying $\|\pi(g)v-v\|<\delta$ for all $g\in K_1(C(\pi),\tau)$.
\end{theorem}

In particular, if $0<\delta<1$ and $\tau>0$ are fixed,
then there are at most
a bounded number of mutually orthogonal unit vectors that are $\delta$-invariant under $K_1(C(\pi),\tau)$.


\subsection{Analytic newvectors
as Whittaker functions}
We fix a generic additive character $\tilde{\psi}$ of the standard maximal unipotent subgroup of $\GL_{n}(\R)$, consisting of upper-triangular unipotent matrices, as defined in \eqref{add-char-of-n}.
We denote by $\W(\pi,{\psi})$
the Whittaker model of a generic irreducible representation $\pi$ of $\GL_{n}(\R)$ 
(see \S\ref{sec:whittaker-kirillov} for details).

Theorem \ref{existence-weak}
gives a simple sense
in which the analytic conductor
controls the invariance properties of vectors.
Theorem \ref{existence-strong} 
is a more powerful yet more technical result
with additional features
that we expect to be useful in applications.
\begin{theorem}\label{existence-strong}
Fix $n \in \mathbb{Z}_{\geq 1}$ and
$\theta \in [0,1/2)$,
let $\Omega$ be a bounded open subset
of $\GL_{n-1}(\mathbb{R})$,
and let $\iota > 0$ be small enough
in terms of $n$ and $\Omega$.
For each $\delta>0$ there exists $\tau>0$ with the following property:

For each
generic irreducible $\theta$-tempered 
unitary representation 
$\pi$
of 
$\GL_{n}(\mathbb{R})$,
there exists
an element $\We\in\W(\pi,{\psi})$ 
of its Whittaker model
satisfying
\begin{itemize}
    \item the normalization $\|\We\|=1$,
    with the norm taken in the Kirillov model
    (\S\ref{sec:whittaker-kirillov}),
    \item the lower bound
    $\We\left[\begin{pmatrix}h&\\&1\end{pmatrix}\right] \geq \iota$ for all $h \in \Omega$,
    and
    \item 
    the invariance properties:
    \begin{enumerate}
        \item     for all $g\in K_0(C(\pi),\tau)$, 
    \begin{equation*}
        \|\pi(g)V- \omega_{\pi}(d_g) V \|_{\W(\pi,\tilde{\psi})}<\delta,
    \end{equation*}
    \item and for $h \in \Omega$,
\begin{equation*}
\left|\We\left[\begin{pmatrix}h&\\&1\end{pmatrix}g\right]-
\omega_{\pi}(d_g)
\We\left[\begin{pmatrix}h&\\&1\end{pmatrix}\right]\right|<\delta.
\end{equation*}
    \end{enumerate}
\end{itemize}
Here $\omega_\pi$ and $d_g$ are as in Theorem \ref{existence-weak}.
The Whittaker function
$V$ may be taken independent
of $\delta$.
\end{theorem}

Informally, Theorem 
\ref{existence-strong}
asserts that
if $\tau$ is small enough,
then there are nonzero vectors in $\pi$
satisfying a form of approximate invariance
under $K_0(C(\pi),\tau)$,
both in the sense of norm
and as quantified by the Whittaker functional.

\begin{remark}
The proof is constructive
(see \S\ref{sec:sketch-pf-thm-2}),
and shows that we may take
$V$ to be a fixed bump function
in the Kirillov model.
\end{remark}

\begin{remark}
The assumption of $\theta$-temperedness 
(or even unitarity) may seem
artificial, since
it is not required in the non-archimedean
setting.
It is used
in the proof
to ensure that $\gamma(1/2-s,{\pi})$ remains holomorphic for $\Re(s)\ge 0$
during a contour shift argument (see \S\ref{sec:sketch-pf-thm-2}). 
This assumption is satisfied in
our intended applications.
\end{remark}




\begin{remark}
The notion of newvector
introduced by
Jacquet--Piatetski-Shapiro--Shalika
\cite{JPSS1}
differs
fundamentally from ours. Those authors
define a Whittaker function $\We$ on $\GL_{n}(F)$ to be
an \emph{essential vector}
if for any spherical representation $\pi'$ of $\GL_{n-1}(F)$,
with $W_0 \in \pi'$
the spherical vector
normalized by $W_0(1)=1$, the local zeta integral (see \S\ref{sec:local-funct-equat}) of
$\We$ and $W_0$ equals the $L$-factor of the Rankin-Selberg convolution $\pi\otimes \pi'$. 
One can also define newvectors at an archimedean place as
test vectors of the Rankin-Selberg zeta integral;
Popa \cite{P} has introduced such a theory for $\GL_2(\R)$,
while the case of $\GL_n(\R)$ is 
addressed by recent work of Humphries \cite{H2}. Such test vectors can be thought as \textit{algebraic} analogues at the archimedean place of the classical newvectors in \cite{JPSS1}.
The analytic newvectors considered here
may be regarded
as ``analytic test vectors''
(i.e., the zeta integral
enjoys a quantitative lower bound rather than merely
nonvanishing)
for ``analytically unramified'' representations
(i.e., those whose analytic conductor
is sufficiently small).
The source of this dichotomy
between algebraic and analytic
is related to the question:
what is the analogue inside $\GL_n(\mathbb{R})$
of $\GL_n(\Z_p) \subseteq \GL_n(\mathbb{Q}_p)$?
An algebraic analogue is $\O(n)$ (a maximal compact subgroup),
while an analytic analogue is a small balanced
neighborhood of the identity.
Algebraic newvectors transform nicely under $\O(n)$,
while analytic newvector transform nicely
under suitable neighborhoods of the identity.
\end{remark}

\subsection{Estimates for Bessel
and trace distributions}
Fix $\tau \in (0,1)$
and
$*\in\{0,1\}$,
so that the congruence
subset $K_*(X,\tau)$
is defined for
all $X \geq 1$.
\begin{defn}\label{defn:normalized-majorant}
We say that
a family $(F_X)_{X \geq 1}$
of smooth functions
$F_X$ on $\GL_{n}(\mathbb{R})$
is a \emph{normalized
majorant
for $(K_*(X,\tau))_{X \geq 1})$}
if the following conditions
hold.
\begin{enumerate}
    \item 
    $F_X$ is nonnegative,
    supported on $K_*(X,\tau)$,
    and constant
    on $K_*(X,\tau ')$
    for some $\tau' \in (0,\tau)$ (independent of $X$).
    \item
    $\int F_X \asymp 1$.
    \item
    Let $A$
    be $n-1$ or $n$
    according as $\ast$ is $0$ or $1$, so that
    $\vol(K_*(X,\tau)) \asymp X^{-A}$.
    Then
    $\partial_a^\alpha\partial_b^\beta\partial_c^\gamma\partial_d^\delta F_X\left[d\begin{pmatrix}a&b\\&1\end{pmatrix}\begin{pmatrix}I_{n}&\\c&1\end{pmatrix}\right]\ll X^{A+|\gamma|+*\delta}$ for any fixed multi-indices $\alpha,\beta,\gamma,\delta$.
\end{enumerate}
Here $\ast \delta$
is $\delta$ or $0$
according as $\ast$ is $1$ or $0$,
while the asymptotic notation
$U \ll V$ signifies
that $|U| \leq C |V|$
for some $C > 0$ not depending
upon $X$
and $U \asymp V$
means that $U \ll V \ll U$.
Since $\tau$ is fixed,
any implied constants
may depend upon $\tau$.
\end{defn}
Such families $(F_X)_{X \geq 1}$ exist.
For instance,
if $F_1$
is any non-negative smooth
function supported in a small enough
neighborhood
of the identity element
of $\GL_{n}(\mathbb{R})$
and constant in some smaller neighborhood,
then the family $(F_X)_{X \geq 1}$
defined by
$F_X\left[\begin{pmatrix}a&b\\c&d\end{pmatrix}\right]:=X^{n-1} F_1\left[\begin{pmatrix}a&b\\cX&d\end{pmatrix}\right]$ 
is a normalized majorant
for $(K_0(X,\tau))_{X \geq 1}$. 

To simplify
terminology,
we will often
abbreviate the above
definition to
``$F_X$ is a normalized
majorant
for $K_0(X,\tau)$,''
but it will always
be understood implicitly
that $F_X$ comes
in a family satisfying
the indicated estimates,
with the indicated uniformity
in $X$.

Theorems \ref{nonexistence}
and \ref{uniqueness}
will be derived from
estimates,
given below
in Theorem \ref{trace-estimates},
for the (relative) trace 
of the operator $\pi(F_X)$,
with $F_X$
a normalized majorant
of $K_1(X,\tau)$. 
To prepare
for stating those estimates,
we first recall the definition of the (relative) trace.

Let $f\in C_c^\infty(\GL_{n}(\R))$. We recall that $\pi(f)$ is trace class and that
\begin{equation}\label{trace}
    \Tr(\pi(f))=\sum_{v\in\B(\pi)}\langle \pi(f)v,v\rangle_\pi.
\end{equation}
Here $\B(\pi)$ is any orthonormal basis of $\pi$.

We  define the \emph{relative trace} (also known as \textit{Bessel distribution})
by
\begin{equation}\label{relative-trace}
    J_\pi(f):=\sum_{\We\in\B(\W(\pi,\tilde{\psi}))}\pi(f)\We(1)\overline{\We(1)}.
\end{equation}
The sums in \eqref{relative-trace} and \eqref{trace} 
converge absolutely and do not depend on a choice of an orthonormal basis $\B(\pi)$, see \cite[Appendix 4]{BM}.



\begin{theorem}\label{trace-estimates}
Fix $0 \leq \theta < 1/2$ and
$\tau>0$.
Let $(F_X)_{X \geq 1}$ be a
normalized majorant for $(K_1(X,\tau))_{X \geq 1}$.
Let $\pi$ be a
$\theta$-tempered
generic irreducible unitary 
representation
of $\GL_{n}(\R)$.
Then
$$\Tr(\pi(F_X)), J_\pi(F_X)\ll_{\tau,\sigma} (C(\pi)/X)^{-n\sigma}$$
for any $0<\sigma<1/2-\theta$.
\end{theorem}

We pause
to describe the informal
content
of the above estimates.
Suppose that $F_X$ is a self-convolution,
so that the  operator $\pi(F_X)$ is positive-definite.
The trace 
$\Tr(\pi(F_X))$ of that operator is then an analytic proxy
for ``the dimension of the space
of 
vectors $\We \in \pi$
that are approximately
invariant by $K_1   (X,\tau)$''. In some applications,
the relative trace
$J_\pi(F_X)$
is the more relevant proxy
(see \S\ref{sec:proof-application}
for details).
Indeed, in the analogous
non-archimedean setting,
we could take for $F_X$
the normalized characteristic
function of $K_1(X)$.
Then
$\Tr(\pi(F_X))$
would be
the average of the character
of $\pi$ over the subgroup
$K_1(X)$.
That average
is
the dimension
of the space of $K_1(X)$-invariant vectors.


Theorem \ref{existence-weak}
implies that
$\Tr(\pi(F_X))\gg 1$
(similarly, Theorem \ref{existence-strong} implies that $J_\pi(F_X)\gg 1$) for $X \geq C(\pi)$
and $\tau > 0$ small but fixed.
Conversely,
smallness of
$\Tr(\pi(F_X))$
when $X$ is substantially smaller than $C(\pi)$, as described in Theorem \ref{trace-estimates}, supports the ``non-existence of the invariant vectors'' statement of Theorem \ref{nonexistence}.
In the transition regime $X \asymp C(\pi)$,
Theorem \ref{trace-estimates} states that $\Tr(\pi(F_X))$ is bounded, which supports the ``multiplicity one''
statement described in Theorem \ref{uniqueness}.

\begin{remark}
We expect that the the estimates in Theorem \ref{trace-estimates} can be improved to $\ll_N(C(\pi)/X)^{-N}$ for 
any fixed $N$. 
To obtain
such an improvement,
it seems
necessary to 
understand asymptotics of the relative character $j_\pi$ (defined in \S\ref{sphericality-unnecessary}) with uniformity
in $\pi$.
Such asymptotics
seem difficult to achieve when $\pi$ experiences
conductor-dropping,
i.e., when one of the Langlands parameters is extremely small.
\end{remark}


\subsection{Application}\label{sec:application}
As noted above,
we hope that 
the local technology
developed in this paper
may be usefully applied to global problems
involving averages
(of Fourier coefficients, $L$-values, ...)
over automorphic
forms on (say) $\PGL_{n}(\Z)\backslash\PGL_{n}(\R)$
ordered by analytic conductor.
We refer to a sequel of this paper \cite{J20} where the first named author provides a few such
applications: an analytic conductor variant of the orthogonality conjecture as in \cite{GSWH}, a Lindel\"of on average upper bound for the second moment of the central $L$-values, vertical Sato--Tate equidistribution of Satake parameters, density estimates for automorphic representations violating the Ramanujan conjecture at a finite place (Sarnak's density hypothesis), and a result on the distribution of low lying zeros of automorphic $L$-functions. 
We have attempted
to keep the present article short and focused on 
basic properties
of analytic newvectors,
but record
the following
simple application
for the sake of illustration.

\begin{theorem}\label{application}
For $X > 1$, we have
\begin{equation}\label{eqn:weak-weyl-cuspidal}
    \sum_{\pi : 
    C(\pi) \leq X }
    \frac{1}{L(1, \pi, \Ad)}
    \ll X^{n-1},
\end{equation}
where the sum is over
cuspidal automorphic representations
$\pi \subseteq L^2(\PGL_{n}(\Z)\backslash\PGL_{n}(\R))$
and $L(s,\pi,\Ad)$ denotes the adjoint $L$-function.
\end{theorem}
The proof of
Theorem \ref{application}
gives an asymptotic formula for a
smoothly truncated and weighted
variant of
the LHS of
\eqref{eqn:weak-weyl-cuspidal},
together with the corresponding
contribution from the
continuous spectrum.
The analogous estimate
for forms on $\GL_{n}(\mathbb{Z}) \backslash \GL_{n}(\mathbb{R})$ having a given central character
may be proved in the same way.
Such estimates are standard when $n=2$
\cite[\S14.10, \S16.5]{IK}. Some variants 
have been established for $n=3$ in \cite{B1, GK2}, for $n=4$ in \cite{GSWH}. 
We note that
Brumley \cite[Theorem 3]{Br} established lower
bounds for $L(1, \pi, \Ad)$
much sharper than those
that follow from 
\eqref{eqn:weak-weyl-cuspidal},
and Brumley--Mili\'cevi\'c \cite{BrM} have recently
obtained a Weyl law 
on adelic quotients,
such as $\PGL_{n}(\mathbb{Q})
\backslash \PGL_{n}(\mathbb{A})$,
for cuspidal automorphic representations
ordered by the 
product of the
analytic conductors
at each place.

While the estimate \eqref{eqn:weak-weyl-cuspidal}
appears to be new,
it should not be regarded
as analytically difficult.
Indeed, by applying
the Kuznetsov formula
to a test function
that projects
onto representations
having parameters
in a ball of essentially bounded size,
it should be possible 
to
establish the analogue of
\eqref{eqn:weak-weyl-cuspidal}
over substantially smaller families.
Our point is to record how a soft proof
follows naturally from Theorem \ref{existence-strong}
as illustration of a technique
that we hope will be useful
also in more analytically challenging
problems.
We are motivated in particular
by the
problem of studying
higher moments of $L$-functions
over such families.

\subsection{Organization of the paper}
In \S\ref{sec:sketch-pf-thm-2} we give a sketch of the proof of Theorem \ref{existence-strong} which is the most technical theorem in this paper. We recommend the reader to go through it first to understand various technicalities in the actual proof. In \S\ref{sec:auxiliary-lemmas} we record most of the definitions and prove auxiliary lemmas which we need at the various stages of the proofs. We recommend that readers to skim over this section for first reading and come back when some Lemma is recalled. In \S\ref{sec:reduction} and \S\ref{sec:proof-mainprop} we reduce the proof of Theorem \ref{existence-strong} to Proposition \ref{mainprop} and Proposition \ref{mainprop} to Proposition \ref{heart}, respectively. Finally, in \S\ref{sec:proof-heart} we prove Proposition \ref{heart}, which is the most technical part of the paper. In \S\ref{sec:proof-multiplicity-one} We prove Theorem \ref{trace-estimates}, Theorem \ref{nonexistence}, and Theorem \ref{uniqueness}. Lastly, in \S\ref{sec:proof-application} we prove Theorem \ref{application}.

\begin{ackno}
We would like to thank Jack Buttcane for 
encouragement and several helpful discussions regarding
the decomposition of the spherical Whittaker function in \S\ref{sec:proof-heart}, and Peter Humphries for careful reading of an earlier draft and making some useful comments. We also thank Djordje Mili\'cevi\'c, Farrell Brumley, and Valentin Blomer for feedback on various aspects of this paper.
\end{ackno}

\section{Sketch for the proof of Theorem \ref{existence-strong}}
\label{sec:sketch-pf-thm-2}
For notational ease, in the rest of the paper, We will work on $\GL_{n+1}(\R)$ (or $\PGL_{n+1}(\R)$) and $\Pi$ will denote an irreducible generic unitary $\theta$-tempered representation of $\GL_{n+1}(\R)$ (or $\PGL_{n+1}(\R)$). Also, in the definition of $K_*(X,\tau)$ in \eqref{defn-congruence-subgroup} the matrix will be written in $n$ plus $1$ block decomposition form.

In this section we assume that $\Pi$ is tempered (instead of $\theta$-tempered) and has trivial central character.
We construct the vector $\We \in \Pi$ by specifying that it be
given by a fixed bump function in the Kirillov model (see
\eqref{newvec}).
The key step in the proof of Theorem
\ref{existence-strong}
is to verify that $\We(g)$
approximates $\We(1)$ for all $g \in K_0(C(\Pi),\tau)$ with $\tau$ small;
the remaining assertions
are deduced from this one
fairly easily.
The main difficulties in the proof are present in the special
case that $g$ is lower-triangular unipotent, so for the purposes
of this sketch we restrict to that case.
It will suffice then to prove the following quantitative
refinement
of the conclusion of Theorem \ref{existence-strong}:
for all small $1 \times n$ row vectors $c$,
\begin{equation}\label{eq:sketch-initial-reduction-quantitative}
  \We \left[ \begin{pmatrix}
      1 &  \\
      c/C(\Pi) & 1
    \end{pmatrix} \right]
  - \We(1)
  \ll
  |c|
\end{equation}
To that end, we first expand the LHS of
\eqref{eq:sketch-initial-reduction-quantitative} using the
Whittaker--Plancherel formula \eqref{genwhitplan}.  We then apply
the 
$\GL(n+1)\times\GL(n)$ local functional equation
\eqref{primelfe} and attempt to analyze the resulting integral.
We will indicate how this goes first
in the simplest case $n = 1$, and then
describe the modifications necessary for general $n$, along with technical difficulties in those cases.

\subsection{Proof for \texorpdfstring{$n=1$}{}}
In this case $\Pi$ is a representation of $\GL_2(\mathbb{R})$.
We define the Whittaker model $\mathcal{W}(\Pi)$ using the
additive character of the unipotent radical in
$\GL_2(\mathbb{R})$ defined as in \eqref{add-char-of-n}.  We
recall that, by the theory of the Kirillov model, there is for
each $f \in C_c^\infty(\mathbb{R}^\times)$ a unique element
$\We \in \mathcal{W}(\pi)$ of the Whittaker model of $\Pi$ for
which
\begin{equation}\label{eq:kirillov-model-sketch-n-1}
    \We \left[ \begin{pmatrix}
      y &  \\
      & 1
    \end{pmatrix} \right]
  = f(y)
\end{equation}
for all $y \in \mathbb{R}^\times$.
We recall the local functional equation
\eqref{lfe}
for $\GL_2 \times \GL_1$:
for $s \in \mathbb{C}$,
the zeta integral
\[
  \int_{\mathbb{R}^\times}
  \We \left[ \begin{pmatrix}
      t &  \\
      & 1
    \end{pmatrix} \right] |t|^s  \, d^\times t
\]
converges absolutely for $\Re(s) > -1/2$
and
extends to a meromorphic function on the complex plane,
where it satisfies the relation
\[
  \int_{\mathbb{R}^\times}
  \We\left[ \begin{pmatrix}
      t &  \\
      & 1
    \end{pmatrix} \right] |t|^s \, d^\times t
  =
  \frac{1}{\gamma(\Pi, 1/2 + s)}
  \int_{\mathbb{R}^\times}
  \We\left[ \begin{pmatrix}
      1 &  \\
      & t
    \end{pmatrix} w \right] |t|^{s} \, d^\times t.
\]
Here $w := \begin{pmatrix}
  & 1 \\
  -1 & 
\end{pmatrix}$
denotes the Weyl element
and $\gamma$ the local $\gamma$-factor,
whose properties we recall in greater detail in
\S\ref{sec:local-funct-equat}.
The meromorphic function
\[
  \Theta(s,\Pi)
  :=
  \frac{
    C(\Pi)^{-s}
  }
  {
    \gamma(1/2+s,\Pi)
  }
\]
is holomorphic for $\Re(s) > -1/2$
and non-vanishing for $\Re(s) < 1/2$.
The normalization is such that $\Theta(s,\Pi)$
is approximately of size $1$
for bounded $s$, uniformly in $\Pi$;
more precisely,
we have
\begin{equation}\label{eq:}
  (1+|\Im(s)|)^{-2|\Re(s)|}
  \ll
  \Theta(s,\Pi)
  \ll
  (1+|\Im(s)|)^{2|\Re(s)|}
\end{equation}
for $s$ of bounded real part and a fixed
positive distance away from the poles of $\Theta(s,\Pi)$,
and with the
implied constant  uniform in $\Pi$
(see Lemma \ref{boundgammafactor}).

Let us assume in this sketch that $\Pi$ belongs to the discrete
series, so that in the Kirillov model of $\Pi$, the subspace of
functions that vanish off the group $\mathbb{R}^\times_+$ of
positive reals is
invariant by the group of positive-determinant
elements of $\GL_2(\mathbb{R})$; this allows us
to simplify the exposition slightly because the character group
of $\mathbb{R}^\times_+$ is a bit simpler than that of
$\mathbb{R}^\times$.

We fix a test function $f\in C^\infty_c(\R^\times_+)$ satisfying
the normalization $\|f\|_2 = 1$.  We extend $f$ by zero to an
element of $C_c^\infty(\mathbb{R}^\times)$.
We construct $\We$ using the theory of the Kirillov model
by requiring that \eqref{eq:kirillov-model-sketch-n-1} hold
for this choice of $f$.
We aim then to verify the estimate
\eqref{eq:sketch-initial-reduction-quantitative}.
To achieve this, we first apply Mellin inversion,
giving
for any $c \in \mathbb{R}$ the identity
\[
  \We \left[ \begin{pmatrix}
      1 &  \\
      c & 1
    \end{pmatrix} \right]
  =
  \int_{(0)}
  \left(
    \int_{t \in \mathbb{R}^\times_+}
    \We \left[ \begin{pmatrix}
        t &  \\
        & 1
      \end{pmatrix} \begin{pmatrix}
        1 &  \\
        c & 1
      \end{pmatrix} \right]
    |t|^s \, d^\times t
  \right)
  \, \frac{d s}{2 \pi i }.
\]
We then apply the local functional
equation to the inner integral;
after some matrix multiplication
and appeal to the left-$N$-equivariance
of $\We$,
we obtain
\[
  \We \left[ \begin{pmatrix}
      1 &  \\
      c & 1
    \end{pmatrix} \right]
  =
  \int_{(0)}
  \frac{1}{\gamma(\Pi,1/2+s)}
  \left(
    \int_{t \in \mathbb{R}^\times_+}
    e(-c/t) \We \left[ \begin{pmatrix}
        1 &  \\
        & t
      \end{pmatrix}w  \right]
    |t|^s \, d^\times t
  \right)
  \, \frac{d s}{2 \pi i }.
\]
We then substitute $c \mapsto c/C(\Pi)$,
apply the change of variables
$t \mapsto t/C(\Pi)$,
and subtract the corresponding identity
for $c=0$,
giving
\[
  \We \left[ \begin{pmatrix}
      1 &  \\
      \frac{c}{C(\Pi)} & 1
    \end{pmatrix} \right]
  -
  \We(1)
  =
  \int_{(0)}
  \Theta(s,\Pi)
  \left(
    \int_{t \in \mathbb{R}^\times_+}
    (e(-c/t) - 1)
    \We\left[ \begin{pmatrix}
        C(\Pi) &  \\
        & t
      \end{pmatrix} w \right]
    |t|^s \, d^\times t
  \right)
  \, \frac{d s}{2 \pi i }.
\]

We claim now that if $t$ is small, then $\We\left[ \begin{pmatrix}
    C(\Pi) &  \\
    & t
  \end{pmatrix} w \right]$ is
negligible;
more precisely,
we claim that
for any fixed integers $M,N \geq 0$,
\begin{equation}\label{eq:sketch-n-1-key-claim}
  (t \partial_t)^N
  \We \left[ \begin{pmatrix}
      C(\Pi) &  \\
      & t
    \end{pmatrix}w \right]
  \ll \min(1,t^M).
\end{equation}
From the claim it follows that $e(-c/t) \approx 1$ on the
``essential support'' of the inner integral, leading eventually
to the required estimate
\eqref{eq:sketch-initial-reduction-quantitative}.

We focus on the case
$N = 0$ of the claim,
and suppose that $t$ is small;
we must show then that
\begin{equation}\label{eq:sketch-n-1-key-claim-specialized}
  \We\left[ \begin{pmatrix}
      C(\Pi) &  \\
      & t
    \end{pmatrix} w \right]
  \ll t^M
\end{equation}
for any fixed $M$.
To that end, as before
we once again apply a Mellin inversion and appeal to the local functional
equation,
which gives
\begin{equation}\label{eq:sketch-n-equals-1-second-LFE}
  \We \left[ \begin{pmatrix}
      C(\Pi) &  \\
       & t
    \end{pmatrix}w \right]
  =
  \int_{(0)}
  t^s
  \Theta(s,\Pi)
  \tilde{f}(s)\, \frac{d s}{2 \pi i },
\end{equation}
with
$\tilde{f}(s) := \int_{t \in \mathbb{R}^\times_+} f(t) t^{-s} \,
d^\times t$.  By the construction of $f$, the its Mellin
transform $\tilde{f}$ is entire and of rapid decay in vertical
strips.  The crux of the argument is to now shift the
integration in \eqref{eq:sketch-n-equals-1-second-LFE} to the
line $\Re(s) = M$ for some fixed, large and positive $M$; the
properties of $\Theta$ summarized above imply that
\begin{equation}\label{eq:sketch-n-1-estimate-Theta-tilde-f}
  \Theta(s,\Pi) \tilde{f}(s) \ll |s|^{-2}
\end{equation}
(say) for such $s$,
leading to the required estimate
\eqref{eq:sketch-n-1-key-claim-specialized}.

In summary, the proof in the case $n=1$
follows readily from two applications
of the local functional equation
and a straightforward Paley--Wiener type analysis
of the Mellin integral representation
\eqref{eq:sketch-n-equals-1-second-LFE}.

\subsection{Difficulties in generalizing to
  \texorpdfstring{$n\ge 2$}{}}\label{sketch-n-ge-2}
We choose $\We$ in a similar
manner to the $n=1$ case; that is $\We\left[\begin{pmatrix}g&\\&1\end{pmatrix}\right]$ is given by a unipotent equivariant bump function on $\GL_n(\R)$ (see \eqref{newvec} for details)
and aim to show as before that
for small $1 \times n$ row vectors $c$,
\[
  \We
  \left[
    \begin{pmatrix}
      1 &  \\
      c/C(\Pi) & 1
    \end{pmatrix}
  \right]
  -
  \We(1)
  \ll
  |c|.
\]
The matrix entries here are written in the evident block form.
We first appeal to the local functional equation
much like in  $n=1$ case,
reducing in this way to proving
estimates slightly
more general
than the following
generalization of \eqref{eq:sketch-n-1-key-claim-specialized}
(see Proposition \ref{heart} for
details):
if
$a = \diag(a_1,\dotsc,a_n)$ is a diagonal matrix
with positive entries
and $a_1$ small,
then
\begin{equation}\label{eq:sketch-n-geq-2-key-estimate}
  \We \left[
    \begin{pmatrix}
      C(\Pi) &  \\
      & a
    \end{pmatrix}
    w
  \right]
  \ll_N
  \delta^{1/2}(a)
  a_1^N.
\end{equation}
Here $\delta(a) := \prod_{j<k} |a_j/a_k|$ is the modular
character of the upper-triangular Borel in $\GL_n$.
There are some similarities
between the present task
and what is accomplished
in the non-archimedean analogue by \cite[\S5 Lemme]{JPSS1}.

For the proof of \eqref{eq:sketch-n-geq-2-key-estimate},
we begin as in the $n=1$ case by expanding
the function $\GL_n(\mathbb{R})
\ni h \mapsto \We \left[ \begin{pmatrix}
    C(\Pi) &  \\
    & h
  \end{pmatrix} w \right]$ using the Whittaker--Plancherel
formula for $\GL_n(\mathbb{R})$
and applying the local functional
equation for $\GL_{n+1} \times\GL_n$ (see
\eqref{finalreducespherical}).
We arrive in this way
at the following generalization of
\eqref{eq:sketch-n-equals-1-second-LFE}:
\begin{equation}\label{sketch-main-integral}
  \We \left[
    \begin{pmatrix}
      C(\Pi) &  \\
      & a
    \end{pmatrix}
    w
  \right]
  =
  \int_{(0)^n}
  W_\mu(a)
  \Theta(\mu,\Pi)
  \langle f, W_\mu\rangle\frac{d\mu}{|c(\mu)|^2}.
\end{equation}
Here
\begin{itemize}
\item  $\int_{(0)^n}$
  denotes an integral over
  $\mu \in \mathbb{C}^n$
  with  $\Re(\mu_1) = \dotsb = \Re(\mu_n) = 0$. Here (and elsewhere in this paper) $d\mu$ denotes $\prod_i d\mu_i$ where $d\mu_i$ is the Lebesgue measure on the vertical line $\Re(s)=0$ normalized by $2\pi i$.
\item 
  $W_\mu$ is the spherical Whittaker function
  normalized so that $W_\mu(1) \asymp 1$
  (see \ref{L2whittaker} for details).
\item $\Theta(\mu,\Pi) := C(\Pi)^{-\mu_1 - \dotsb - \mu_n}
  \gamma(\Pi \otimes \tilde{\pi}_\mu,1/2)$
  is holomorphic for $\Re(\mu_i)> -1/2$
  and has properties analogous to those of $\Theta(s,\mu)$
  mentioned above
  in the $n=1$ case, and
\item $c(\mu)$ is a product of $\Gamma$-functions,
  related to the Plancherel density. 
\item $\langle f, W_\mu\rangle:=\int_{N\backslash \GL_n(\R)}f(g)\overline{W_\mu(g)}dg$, where $N$ is the unipotent subgroup of upper triangular matrices in $\GL_n(\R)$. 
 
\end{itemize}
The product
$\Theta(\mu,\Pi) \langle f, W_\mu \rangle$
enjoys strong estimates
analogous
to \eqref{eq:sketch-n-1-estimate-Theta-tilde-f}:
it extends to a meromorphic function in $\mu$,
holomorphic in $\Re(\mu_i) > -1/2$
and of rapid decay in vertical strips, uniformly in $\Pi$.

Similar to the standard proof of any Paley-Wiener type statement, we use the rapid decay of $\langle f, W_\mu\rangle$ in $\mu$ to ensure the convergence of the integral \eqref{sketch-main-integral}, and our source of the decay in the $a_1$ direction (i.e. as $a_1\to 0$, as required) is contour shifts.
However, the Paley-Wiener type argument as simple as in $n=1$ case
can not be directly applied in this case. In fact, it turns out that
any shift of the $\mu$ contour
that avoids polar hyperplanes of $\Theta(\mu,\Pi)$
is insufficient to achieve
the required bound $\ll\min(1,a_1^N)$.

The reason for this obstruction
is that $W_\mu$ has asymptotic expansion of the form
$$\delta^{1/2}(a)\sum_{w\in S_n}a^{w\mu}M_{w\mu}(a),\quad a^\mu=\prod_{i=1}^na_i^{\mu_i}.$$
Here $S_n$ is the Weyl group of $\GL(n)$ and $M(\mu,a)$, which we informally call as $M$-Whittaker function, is an infinite series of the form
$$\sum_{k\in Z^{n-1}_{\ge 0}}c_{\mu}(k)\prod_{i=1}^{n-1}(a_i/a_{i+1})^{2k_i}.$$
For example, on $\GL(2)$ the spherical Whittaker function $W_\mu(a)$ is given by the $K$-Bessel function as
\begin{align*}
&(a_1/a_2)^{1/2}(a_1a_2)^{(\mu_1+\mu_2)/2}K_{(\mu_1-\mu_2)/2}(2\pi a_1/a_2)\\&=a_1^{\mu_1+1/2}a_2^{\mu_2-1/2}\sum_{k=0}^\infty c_\mu(k)(a_1/a_2)^{2k}+a_1^{\mu_2+1/2}a_2^{\mu_1-1/2}\sum_{k=0}^\infty d_\mu(k)(a_1/a_2)^{2k}.
\end{align*}
for some complex coefficients $c_\mu(k)$ and $d_\mu(k)$. The infinite sums are essentially $I$-Bessel functions, and are exponentially increasing in $a_1/a_2$. If we shift the contour of $\mu_1$ to right side e.g. to $\Re(\mu_1)=N$ for some large positive $N$ we do not cross any pole of $\Theta(\mu,\Pi)$. If $a_2$ is very large compared to $a_1$ (e.g. $a_2=a_1^{-1}$ and $a_1\to 0$) this contour shift does not yield the required bound $\ll a_1^N$. That is why we first need to decompose the Whittaker functions into finitely many $M$-Whittaker functions, (see Lemma \ref{partialdecomposition}), and for each summand $M$-Whittaker function we shift contour to the relevant direction and get the required bound. For instance, in this case in the first summand we shift $\mu_1$ and in the second we shift $\mu_2$ to the right side.

However, there is still one technical issue. Each $M$-Whittaker function, like
the $I$-Bessel functions, is exponentially increasing in the positive roots (e.g. in $\GL(2)$ exponentially increasing in $a_1/a_2$). That is why we can effectively apply the technique of contour shifting only when the positive roots are bounded. For example, on $\GL(2)$ we decompose $W_\mu(\diag(a_1,a_2))$ 
into relevant $M$-Whittaker function and shift contour
only when $a_1<a_2$, as in the discussion in the previous paragraph. On the other hand, in $\GL(2)$ the Whittaker function $W_\mu(a)$ decays rapidly as $a_2/a_1\to 0$. Such decay is enough to treat the complimentary case $a_1>a_2$.

It is tempting to imagine that we should decompose only when $a$ is in the positive Weyl chamber. If $a$ does not lie in a positive Weyl chamber then at least one root is large, and we may expect that the rapid decay of the Whittaker function will save the day. Although this works when $n=2$ this fails for general $n$; for e.g., there are diagonal elements in $\GL(n)$ which barely fail to be in a positive Weyl chamber. For example, in $\GL(3)$ the element $Y:=\diag(y,-y/\log y, e^{1/y})$ as $y\to 0$ logarithmically fails to be in the positive chamber. For this element the rapid decay estimate of the Whittaker function yields only a logarithm decay
$$W(Y)\ll |\log y|^{-N},$$
not a polynomial decay $\ll y^N$, so does not meet our requirement.

To deal with this issue we need to treat the elements like $Y$ as if they are in the positive chamber. To do that we divide the set of diagonal matrices into two classes whether they satisfy a property $\pop$ or not (see Definition \ref{POP}). The $\pop$ refers to whether the tuple $(a_1,\dots, a_n)$ of a diagonal element $a=\diag(a_1,\dots, a_n)$ has a partial ordering of the form \emph{all of $a_1,\dots, a_s$ are smaller than all of $a_{s+1},\dots,a_n$}. For instance, the element $Y$ above has $\pop$ property for $s=2$, as $y\to 0$.

In Lemma \ref{notpop} we showed for the elements $a$ which do not satisfy $\pop$ the rapid decay of the Whittaker function implies the required bound. 
Rest of the section \S\ref{sec:proof-heart} is devoted to the case when $a$ satisfies $\pop(s)$ for some $s$. In this case we decompose the Whittaker function $W$ into the $M$-Whittaker functions. However, we can not do a \emph{full} decomposition as we described above in the $\GL(2)$ case,
because $a$ may not lie in the positive Weyl chamber and $M$ might exponentially blow up. To make sure we have control on the exponential increment we only \emph{partially} decompose, so that the $M$-Whittaker functions have exponential increment only in the roots of the form $a_i/a_j$ with $1\le i\le s$ and $s+1\le j\le n$; hence $M$ does not blow up as $a$ satisfies $\pop(s)$.
Loosely speaking, a partial decomposition is 
corresponding to the Levi in $\GL(n)$ attached to the partition of $n$
of the form
$n = s + 1 + \dotsb + 1$, and a full decomposition is the same with $s=1$.

Such full decompositions of the spherical Whittaker function have appeared
in the literature, in particular in the works of Goodman--Wallach \cite{GW} and Hashizume \cite{H}. Such results are also implicit in unpublished work of Casselman--Zuckerman and independent work of Kostant \cite{Ko}. In fact, \cite{H} also provides explicit asymptotic expansions of such rapidly increasing $M$-Whittaker functions. However, both \cite{GW,H} use techniques from differential equation, whereas we use spectral techniques to produce such decompositions (also see \cite{Bu}).

\subsection{Sketch of the proof for \texorpdfstring{$\mathrm{PGL}_{n+1}(F)$, where $F$ is non-archimedean}{}} 

We re-establish the invariance result \cite[\S5 Th\'eor\'eme]{JPSS1} along the exact same lines we prove in the archimedean case. In \cite{JPSS1} it is, instead, proved using the test vector property of the newvector. However, the proof in the $p$-adic case is, still, simpler due to existence of simpler explicit algebraic formulas. Also the corresponding $M$-Whittaker functions do not exponentially increase in the non-archimedean case.

We will only concentrate on proving the analogue of the crucial
part, Proposition \ref{heart}. Let $F$ be a non-archimedean
local field with ring of integers $\mathfrak{o}$ and uniformizer
$\varpi$. Let $\Pi$ be
a generic irreducible tempered unitary representation of $\GL_{n+1}(F)$. Let $\We$ be the vector in $\W(\Pi,\tilde{\psi})$ given in the Kirillov model by 
\begin{equation}
    \We\left[\begin{pmatrix}nak&\\&1\end{pmatrix}\right]=\psi(n)\mathrm{char}_{\mathfrak{o}^\times}(a),
\end{equation}
where $n,a:=\diag(a_1,\dots,a_n)$ are the unipotent and diagonal elements, respectively, and $k\in \GL_n(\mathfrak{o})$. 
\begin{prop}\label{heart-p-adic}
Let $w$ be the long Weyl element in $\GL(n+1)$. Then
$$\We\left[\begin{pmatrix}C(\Pi)&\\&a\end{pmatrix}w\right]\neq 0\implies |a_1|\gg 1.$$
\end{prop}
We emphasize the similarity of this proposition with \cite[\S5 Lemme]{JPSS1}. For simplicity, in the sketch of the proof we assume that $\Pi$ is supercuspidal.

\begin{proof}[Sketch of the proof]
Our point of departure is, as in the archimedean case, the $p$-adic Kontorovich-Lebedev-Whittaker transform (for a proof in the case of $F=\Q_p$ see \cite{Gu}, for general Whittaker--Plancherel formula we refer to \cite{D}) and $\GL(n+1)\times\GL(n)$ local functional equation. We obtain
$$\We\left[\begin{pmatrix}C(\Pi)&\\&a\end{pmatrix}w\right]=\int_{\pi}W_\pi(a)\gamma(1/2,\Pi\otimes\bar{\pi})\omega^{-1}_\pi(C(\Pi))\langle \mathrm{char}_{\mathfrak{o}^\times},W_\pi\rangle d\mu_p(\pi).$$
Here $\pi$
runs over the spherical tempered dual of $\GL_n(F)$ and $d\mu_p$ is the Plancherel measure on it. Let $m\in \Z^n$ and $a=\diag(\varpi^m)$, i.e. $a_i=\varpi^{m_i}$. Let $\alpha\in (S^1)^n$ be the Langlands parameters of $\pi$ and $W_\pi$ is the spherical Whittaker function of $\pi$ described by Shintani's formula \cite{Sh} below.
\begin{equation}
    W_\pi(a)=\begin{cases}&\delta^{1/2}(a)\frac{\det((\alpha_j^{m_i+n-i})_{i,j})}{\prod_{i<j}(\alpha_i-\alpha_j)},\text{ if }m_1\ge\dots\ge m_n,\\
    &0,\text{ if otherwise}.\end{cases}
\end{equation}
Thus we may restrict $a$ to be of the form $\diag(\varpi^m)$ with $m_1\ge\dots\ge m_n$. 
Inserting this formula
for $W_\pi$
and
explicating the Plancherel
density
(see \cite{Gu}),
we may rewrite $\We\left[\begin{pmatrix}C(\Pi)&\\&a\end{pmatrix}w\right]$ as
\begin{equation}
    \delta^{1/2}(a)\int_{(S^1)^n}\det((\alpha_j^{m_i+n-i-1})_{i,j})\gamma(1/2,\Pi\otimes\bar{\pi})\omega^{-1}_\pi(C(\Pi))\prod_{i>j}(\alpha_i-\alpha_j)d\alpha.
\end{equation}
We note that, as $\Pi$ is supercuspidal and $\pi$ is unitary unramified $\gamma(s,\Pi\otimes\bar{\pi})$ is same as $\epsilon(s,\Pi\otimes\bar{\pi})$ and
$$\omega_\pi^{-1}(C(\Pi))\epsilon(1/2,\Pi\otimes\bar{\pi})$$
is independent of $\pi$ and is bounded. To see this, recall that if $\pi$ has Langlands parameters $\{\alpha_i\}_{i=1}^n$ with $\alpha_i:=p^{\beta_i}$, then
$$\epsilon(1/2,\Pi\otimes\bar{\pi})=\prod_{i=1}^n\epsilon(1/2-\beta_i,\Pi).$$
From \eqref{eqn:abs-gamma-is-C} we obtain $$\epsilon(1/2-\beta_i,\Pi)={C(\Pi)}^{\beta_i}\epsilon(1/2,\Pi).$$
From the above formula and unitarity of $\epsilon(1/2,\Pi)$ the claim is immediate.

We will now proceed along the sketch we have provided in the previous subsection in the archimedean case. The analogue of ``decomposing the spherical Whittaker function into finitely many $M$-Whittaker function'' is expanding the determinant
$\det((\alpha_j^{m_i+n-i})_{i,j})$ into various monomials depending on $\alpha_i$; i.e., the analogue of $M$-Whittaker function is a monomial of the form $\prod_j\alpha_j^{n_j}$. We distribute the integral over this decomposition. A generic term in this decomposition will look like
$$\delta^{1/2}(a)\int_{(S^1)^n}\alpha_j^{m_1}H_j(\alpha)d\alpha,$$
where $H_j(\alpha)$ is a meromorphic function in $\{\alpha\mid |\alpha_j|\le 1\}$ with poles at most at $\alpha_i=0$, such that order of the pole at $\alpha_j=0$ is bounded (i.e. does not depend on $m$). Thus making $m_1$ sufficiently positive we compute the $\alpha_j$ integral to be zero. Hence we conclude.
\end{proof}


\section{Basic Notations and Auxiliary Lemmata}\label{sec:auxiliary-lemmas}
In the section we will recall some notations which we will use frequently. We will also need some some well-known tools from representation theory of $G:=\GL_n(\R)$, which we will describe in the next few subsections. 

We use the Iwasawa decomposition of $G=NAK$ with $N$ being the maximal unipotent subgroup of upper triangular matrices, $A$ being the subgroup of the positive diagonal matrices, and $K=\mathrm{O}(n)$. We fix Haar measure $dg, dn, dk$ on $G, N, K$, respectively such that the volume of $K$ is one with respect to $dk$, and 
$$dg=dn\frac{da}{\delta(a)}dk,\quad g=nak,$$
where $da$ on $A\ni \diag(a_1,\dots a_n)$ is given by $\prod_i d^\times a_i$ and $\delta$ is the modular character on $NA$. We will use similar Haar measures on $\GL(r)$ (and its subgroups) for any $r$ without mentioning explicitly.

We introduce a modified Vinogradov notation. Let $\epsilon>0$ be a fixed small quantity (say, $< n^{-10}$). In this article we abbreviate the inequality
$$\varphi_1(a,\dots)\ll_{\epsilon,\dots} \varphi_2(a,\dots)\prod_{i=1}^n(a_i+a_i^{-1})^\epsilon $$
by
$$\varphi_1(a,\dots)\prec_{\dots} \varphi_2(a,\dots),$$
where $\varphi_i$ are some functions on $a,\dots$.

\subsection{Additive character} Recall the maximal unipotent $N<G$. We fix an additive character $\psi$ of $N$, which is given by
\begin{equation}\label{add-char-of-n}
    \psi(n(x))=e\left(\sum_{i=1}^{n-1}x_{i,i+1}\right),\quad n(x):=(x_{i,j})\in N.
\end{equation} 
where $e(z):=\exp(2\pi i z)$. By $\tilde{\psi}$ we will denote the similarly defined character of $N_{n+1}$ which is the maximal unipotent subgroup of $\GL_{n+1}(\R)$.

\subsection{Whittaker and Kirillov models}\label{sec:whittaker-kirillov}
For details of this subsection we refer to \cite[Chapter $3$]{J2}. Let $\Pi$ be a generic irreducible unitary representation of $\GL_{n+1}(\R)$. Let $\tilde{\psi}$ be the character of $N_{n+1}<\GL_{n+1}(\R)$ as in \eqref{add-char-of-n}. Recall that $\Pi$ is generic if $$\mathrm{Hom}_{\GL_{n+1}(\R)}(\Pi,\mathrm{Ind}_{N_{n+1}}^{\GL_{n+1}(\R)}\tilde{\psi})\neq 0.$$
It is known that if $\Pi$ is generic then the above space is one dimensional. Let $\lambda$ be a nonzero element in this $\mathrm{Hom}$-space. Then the Whittaker model $\W(\Pi,\tilde{\psi})$ of $\Pi$ is the image of $\lambda$. One writes 
$$W_v(g)=\lambda(\Pi(g)v),\quad g\in \GL_{n+1}(\R), v\in \Pi.$$
If $\Pi$ is unitary then the corresponding unitary structure on $\W(\Pi,\tilde{\psi})$ is given by
$$\langle W_1,W_2\rangle_{\W(\Pi,\tilde{\psi})} =\int_{N\backslash G}W_1\left[\begin{pmatrix}g&\\&1\end{pmatrix}\right]\overline{W_2\left[\begin{pmatrix}g&\\&1\end{pmatrix}\right]}dg,$$
for $W_1,W_2\in \W(\Pi,\tilde{\psi})$.

Let $C_c^\infty(N\backslash G,\psi)\ni \phi$ be the set of smooth functions on $G$ compactly supported mod $N$ such that 
$$\phi(ng)=\psi(n)\phi(g), \quad n\in N, g\in G.$$ The theory of Kirillov model states that \cite[Proposition $5$]{J2} there exists a unique $W_\phi\in \W(\Pi,\tilde{\psi})$ such that
$$W_\phi\left[\begin{pmatrix}g&\\&1\end{pmatrix}\right]=\phi(g),$$
and the map $\phi\mapsto W_\phi$ is continuous.

\subsection{Langlands parameters}\label{langlands-parameters}
For a representation $\pi$ of $\GL_m(\R)$ let us write its $L$-factor $$L(s,\pi)=\prod_{i=1}^m\Gamma_\R(s+\mu_i(\pi)),$$
where $\mu_i(\pi)\in \C$ are the {Langlands Parameters} attached to $\pi$. In this case we define the \textit{analytic conductor} of $\pi$ to be
$$C(\pi)=\prod_{i=1}^m(1+|\mu_i(\pi)|).$$
By $\mu$ we will denote a complex $n$-tuple $(\mu_1,\dots,\mu_n)$. We define a quantity
$$c(s,\mu):=\prod_{i<j}\Gamma_\R\left(s+{\mu_i-\mu_j}\right).$$
We abbreviate $c(0,\mu)$ as $c(\mu)$.

\subsection{Whittaker--Plancherel formula} For general discussion on the Whittaker--Plancherel theorem we refer to \cite[Chapter $15$]{W}. Let $\hat{G}$ be the set of isomorphism classes of generic irreducible tempered unitary representations of $G$. Let $\hat{G}_0\subseteq\hat{G}$ be the isomorphism classes of spherical representations (by spherical representation we mean a representation which contains a right $K$-invariant vector). One can write down the general Whittaker--Plancherel formula for $G$ as follows. Let $F\in L^2(N\backslash G,\psi)$ continuous. Then,
\begin{equation}\label{genwhitplan}
    F(g)=\int_{\hat{G}}\sum_{W\in\B(\pi)}W(g)\langle F, W\rangle d\mu_p(\pi),
\end{equation}
where $d\mu_p$ is the Plancherel measure on $\hat{G}$, and $$\langle F, W\rangle:=\int_{N\backslash G} F(g)\overline{W(g)}dg.$$
Here $\B(\pi)$ is an orthonormal basis of $\pi$. The above sum does not depend on a choice of $\B(\pi)$.

We may choose a basis $$\B(\pi):=\cup_{\tau\in\hat{K}}\{W_\tau^i\mid 1\le i\le n_\tau\}$$ consisting of $K$-isotypic vectors. Here $\{W_\tau^i\}_{i=1}^{n_\tau}$ is an orthonormal basis of $\tau$-type. We also know that $n_\tau=1$ if $\tau$ is the trivial representation. We will now produce a rather simplified version of the Whittaker--Plancherel formula for spherical functions i.e. functions which are right $K$-invariant. We first note that, if $F\in L^2(N\backslash G,\psi)^K$ then,
$$\langle F, W\rangle=0,\quad\text{for all } W\in\B(\pi)\setminus\pi^K.$$
Therefore for spherical $F$ only the spherical representations will contribute to the right hand side of \eqref{genwhitplan}. Let $W_\pi\in \pi^K$ with $\|W_\pi\|=1$. Then \eqref{genwhitplan} reduces to
\begin{equation}\label{spherwhitplan}
    F(g)=\int_{\hat{G}_0}W_\pi(g)\langle F, W_\pi\rangle d\mu_p(\pi).
\end{equation}

\subsection{Local functional equation}\label{sec:local-funct-equat}
Local Rankin--Selberg zeta integral for a representation $\W(\Pi,{\tilde{\psi}})\ni \We$ of $\GL_{n+1}(\R)$ and $\W(\pi,\bar{\psi})\ni \V$ of $\GL_n(\R)$ is defined by
$$\int_{N\backslash G}\We\left[\begin{pmatrix}g&\\&1\end{pmatrix}\right] {W(g)}|\det(g)|^s dg.$$
If $\Pi$ is $\theta$-tempered and $\pi$ is tempered then the above integral is absolutely convergent if $\Re(s)>-1/2+\theta$.

Let $\omega_\pi$ be the central character of $\pi$. Then for any $W\in \W(\pi,\bar{\psi})$ the local functional equation is (see \cite{JS1, JS2})
\begin{multline}\label{primelfe}
    \int_{N\backslash G}{\We}\left[w\begin{pmatrix}g^{-t}&\\&1\end{pmatrix}\right]{W}(w'g^{-t})|\det(g)|^{-s}dg\\
    \qquad=\omega_\pi(-1)^n\gamma(1/2+s,\Pi\otimes\pi)\int_{N\backslash G}\We\left[\begin{pmatrix}g&\\&1\end{pmatrix}\right] {W(g)}|\det(g)|^s dg,
\end{multline}
where $w,w'$ are long Weyl elements of $\GL_{n+1}$ and $G$ respectively, and 
$$\gamma(s,\Pi\otimes\pi):=\epsilon(s,\Pi\otimes\pi)\frac{L(1-s,\widetilde{\Pi}\otimes\widetilde{\pi})}{L(s,\Pi\otimes\pi)},$$
and $\epsilon(s,.)$ is the epsilon factor attached to $\pi$. It is known that $\epsilon$-factors are entire in $s$ and have absolute value one, thus are constants $\epsilon(\pi)$.
Changing variable $g\mapsto w'g^{-t}w'$ we can also rewrite \eqref{primelfe} as
\begin{multline}\label{lfe}
    \int_{N\backslash G}{\We}\left[\begin{pmatrix}1&\\&g\end{pmatrix}w\right]{W}(gw')|\det(g)|^{s} dg\\
    \qquad=\omega_\pi(-1)^n\gamma(1/2+s,\Pi\otimes\pi)\int_{N\backslash G}\We\left[\begin{pmatrix}g\\&1\end{pmatrix}\right]{W(g)}|\det(g)|^s dg.
\end{multline}
Here in the LHS the matrix in the entry of $\We$ is in the $1+n$ block diagonal form.

\subsection{Spherical tempered dual}
One can parametrize $\hat{G}_0$ by 
$$\{\mu:=(\mu_1,\dots,\mu_n)\mid \mu_i\in i\R, 1\le i\le n\},$$
where a purely imaginary $n$-tuple $\mu$ corresponds with the induced representation 
$$\pi_{\mu}:=\mathrm{Ind}_B^G\chi,\quad\chi\left(na\right)=\prod_{i=1}^n|a_i|^{\mu_i}, n\in N,$$
where $a:=\diag(a_1,\dots,a_n)$,
because any tempered spherical representation of $G$ is of the above form. 
For later purpose, for $\mu\in\C^n$ we define the quantity
\begin{equation}\label{defn-d-mu}
d(\pi_\mu):=d(\mu):=1+\sum_{j=1}^n|\Im(\mu_j)|^2.
\end{equation}

\subsection{Conductors and gamma-factors}
Recall the definition of $\gamma$-factor from \S\ref{sec:local-funct-equat} and the definition of Conductor from \S\ref{langlands-parameters}.
\begin{lemma}\label{boundgammafactor}
Let $\Pi$ and $\pi$ be generic irreducible unitary representations of $\GL_{n+1}(\R)$ and $G$, respectively. Then
\begin{enumerate}
\item For $s\in\C$
  of bounded real part and a
  fixed positive distance
  away from any pole or zero of
  $\gamma(1/2-s,\pi)$,
  we have
    $$\gamma(1/2-s,\pi)\asymp C(\pi\otimes|\det|^{\Im(s)})^{\Re(s)}.$$
    \item $\frac{C(\Pi)^n}{C(\pi)^{n+1}}\le C(\Pi\otimes\pi)\le C(\Pi)^nC(\pi)^{n+1}$.
\end{enumerate}
\end{lemma}
\begin{proof}
$(1)$ is standard and follows from the Stirling approximation of the $L$-factors, for e.g., (see \cite{Br}). The second inequality of $(2)$ is obtained in the \cite[Appendix A]{H1}. The first inequality can be proved in the very same way as the other one. As in \cite[Appendix A]{H1} one can appeal to Langlands classification of the admissible dual and reduce to the case of representations $\Phi$ and $\phi$ of the Weil-Deligne group of $\R$. For instance, let both $\Phi$ and $\phi$ are one dimensional representations with Langlands parameters $(\mu,0)$ and $(\nu,0)$, respectively, then the parameter of $\Phi\otimes\phi$ can be given by $(\mu+\nu,0)$. Then the first inequality follows from
$$(1+|\mu|)\le (1+|\mu+\nu|)((1+|\nu|).$$
Rest of the cases follow similarly.
\end{proof}

For Brevity, we define $\Theta:\hat{G}\to \C$ by
$$\pi\mapsto\Theta(\pi,\Pi):=\omega_\pi^{-1}(C(\Pi))\gamma(1/2,\Pi\otimes\bar{\pi}),$$
where $\omega_\pi$ is the central character of $\pi$.
If $\pi$ is the spherical representation $\pi_\mu$ for some $\mu\in \C^n$ we, by abuse of notation, denote $\Theta(\pi_\mu,\Pi)$ by $\Theta(\mu,\Pi)$. We record that it follows from Lemma \ref{boundgammafactor}
\begin{equation}\label{boundG}
    \Theta(\mu+2M,\Pi)\ll_M \prod_{i=1}^n (1+|\mu_i|)^{O_M(1)}.
\end{equation}
for $M\in\Z^n_{\ge 0}$ fixed and $\mu\in \C^n$ with $0\le \Re(\mu)\ll 1$. We also note that if $\Pi$ is $\theta$-tempered for $0\le \theta<1/2$ then $\Theta(\mu,\Pi)$ is holomorphic for $\Re(\mu_i)\ge 0$.

\subsection{Explicit Plancherel measure} We describe the Plancherel measure explicitly in the case of the spherical Whittaker--Plancherel transform \eqref{spherwhitplan}. From (see \cite{GK1}) we get that if $\pi_\mu\in\hat{G_0}$ for some $\mu\in i\R^n$ then
\begin{equation}\label{planchereldensity}
    d\mu_p(\pi_\mu)=\left|\frac{c(1,\mu)}{c(0,\mu)}\right|^2d\mu_1\dots d\mu_n,
\end{equation}
where $d\mu_i$ are the Lebesgue measures on $i\R$ normalized by $2\pi i$.

\subsection{Differential operator and Sobolev norm} Let $\{X_i\}$ be a basis of $\g:=\mathrm{Lie}(G)$. We define, for each $M\ge 0$, a second order differential operator by
\begin{equation}\label{defdiffop}
\D_M:=M+1-\sum_{i=1}^{n^2}(X_i^2).
\end{equation}
We abbreviate $\D_0$ as $\D$. We define a Sobolev norm on the space of $\pi\in \hat{G}$ by
\begin{equation}\label{sobolev-norm}
    S_d(v):=\|\D^d v\|_\pi,
\end{equation}
A similar sort of Sobolev norm has been used in \cite{MV}.

\begin{lemma}\label{diffop}
Let $\D_M$ be the differential operator in \eqref{defdiffop}.
\begin{enumerate}
    \item $\D_M$ is self-adjoint and positive definite on unitary representations of $G$. Eigenvalues of $\D_M$ are at least $M+1$.
    \item If $C_G$ and $C_K$ denote the Casimir elements for the groups $G$ and $K$, respectively then,
    $$\D_M=M-(1+C_G)+2(1+C_K).$$
    \item $C_G$ acts on $\pi_\mu$ by the scalar $\lambda(\pi_\mu):=-T+\|\mu\|^2$, where $\|\mu\|^2:=\sum_{i=1}^n|\mu_i|^2$ and $T>0$ is an absolute constant depending only on $n$.
    \item Eigenvalue of the spherical vector in $\pi_\mu$ under $\D_M$ is of size $\asymp 1+\|\mu\|^2$.
\end{enumerate}
\end{lemma}
\begin{proof}
$(1)$ is standard and follows from the fact that $\sum_{i=1}^{n^2}X_i^2$ is self-adjoint and negative definite. This can be found in \cite{NS}.
To prove $(2)$ note that $\g=\mathfrak{p}+\mathfrak{k}$ where $\mathfrak{p}$, and $\mathfrak{k}$ are Lie algebras of $NA$ and $K$, respectively. We fix bases $\{X_i\}_i$ of $\mathfrak{k}$ and $\{Y_j\}_j$ of $\mathfrak{p}$. Thus from the definitions of the standard Cartan involution and Killing form \cite[Chapter VIII]{K} we get that
$$C_G=-\sum_{X_i\in\mathfrak{k}}X_i^2+\sum_{Y_i\in\mathfrak{p}}Y_i^2,\quad C_K=-\sum_{X_i\in\mathfrak{k}}X_i^2.$$
Thus from the definition of \eqref{diffop} we get that 
$$\D_M=M+1-\sum_{X_i\in\mathfrak{k}}X_i^2-\sum_{Y_i\in\mathfrak{p}}Y_i^2=M-(1+C_G)+2(1+C_K).$$
$(3)$ is standard in literature (see e.g. \cite[p.2]{BHM}). $(4)$ follows from $(3)$ as $C_K$ will act trivially on the spherical vector in $\pi_\mu$.
\end{proof}

\begin{lemma}\label{sobolev-norm-property}
Let $S_d$ be the Sobolev norm defined in \eqref{sobolev-norm}. Then for $d_1,d_2>0$ there exists $L:=L(d_1,d_2)>0$ such that 
$$\int_{\hat{G}}C(\pi)^{d_1}\sum_{W\in \B(\pi)}S_{d_2}(W)S_{-L}(W)d\mu_p(\pi)$$
is convergent. Here $\B(\pi)$ is an orthonormal basis of $\pi$ consisting of eigenvectors of $\D$..
\end{lemma}
\begin{proof}
Let $\B(\pi)=\{W_i\}_{i\in\N}$ with eigenvalues $\{\lambda_i\}_{i\in \N}$, correspondingly. From $(1)$ of Lemma \ref{diffop} we get that $\lambda_i\ge 1$. Thus,
$$\sum_{W\in\B(\pi)}S_{d_2}(W)S_{-L}(W)= \mathrm{Trace}\mid_\pi(\D^{d_2-L}).$$
There exists an element $P_{d_1}$ in the center of the universal enveloping algebra fo $G$ such that $$C(\pi)^{d_1}\ll \lambda_\pi(P_{d_1}),$$
where $\lambda_\pi(P)$ is the scalar by which $P$ acts on $\pi$. Thus the integral in the question can be bounded by 
$$\int_{\hat{G}}\mathrm{Trace}_\pi(P_{d_1}\D^{d_2-L})d\mu_p(\pi).$$
From \cite[\S$8.5$, Lemma $2$]{NV} we know that for large enough $A$ the operator $P_{d_1}\D^{-A}$ is bounded. Finally, the integral $\int_{\hat{G}}\mathrm{Trace}_\pi(\D^{-B})d\mu_p(\pi)$ is convergent for sufficiently large $B$. A proof of this result can be found in \cite[\S$A.4.2$, Lemma (ii)]{NV}. We conclude our proof by making $L$ large enough.
\end{proof}

\subsection{Spherical Whittaker functions} In this subsection we will work out some relevant analysis of the spherical Whittaker function on $G$. The general references for spherical Whittaker functions are \cite[Chapter $5$]{G}, \cite{St1, St2}, and Jacquet's work \cite{J3}. Let $\mu\in\C^n$. We call $\pi_\mu$ to be the spherical principal series representation with the Langlands parameters $\mu$. Let $W_\mu$ be the spherical vector in $\pi_\mu$ defined by the following normalization of the Jacquet's integral.
\begin{equation}\label{defnwhittaker}
    W_\mu(g):=c(1,\mu)d_n\int_{N}I_\mu(wng)\overline{\psi(n)}dn,\quad g\in G,\Re(\mu_i-\mu_{i+1})>0;
\end{equation}
where  
$$I_\mu(nak):=\delta^{1/2}(a)\prod_{i=1}^n a_i^{\mu_i},\quad n\in N, a\in A, k\in K.$$
Jacquet in \cite{J3} showed that $W_\mu$ has a analytic continuation to $\C^n$ and is invariant under action of the Weyl group on $\mu$. Here $d_n$ is an absolute constant such that when $\mu$ is purely imaginary,
\begin{equation}\label{L2whittaker}
    \|W_\mu\|^2=|c(1,\mu)|^2,
\end{equation}
by Stade's formula \cite[Theorem $1.1$]{St2}. We record two type of bounds of $W_\mu$ we will use at various stages of the proofs.

\begin{lemma}\label{rapiddecay}
Let $\mu$ be purely imaginary. Then for any $M\in \Z^{n-1}_{\ge 0}$
$$\frac{W_\mu(a)}{c(1,\mu)}\prec_{M} \delta^{1/2}(a)d(\mu)^{O_M(1)}\prod_{j=1}^{n-1}\min(1,a_{j+1}/a_j)^{M_j},$$
where $O_M(1)$ denotes a bounded quantity depending on $M$.
\end{lemma}
\begin{proof}
This result is already proved in a similar form in \cite[Theorem $1$]{BHM}. However, as we are happy with a polynomial dependency on $\mu$ we can infer from more general result in Lemma \ref{whittaker-bound}.
\end{proof}

\begin{lemma}\label{boundW}
Let $\mu\in \C^n$ such that $\Re(\mu_i)$ are non-negative, distinct, and small enough (say $<1/100$). Then for any $k\in \Z^n_{\ge 0}$
$$W_{\mu+k}(a)\prec c(1,-\sigma(\mu+k))\delta^{1/2}(a)\prod_{i=1}^n a_i^{\Re((\sigma(\mu+k))_i},$$
where $\sigma\in S_n$ such that $\Re(\mu_{\sigma(1)}+k_{\sigma(1)})\le\dots\le\Re(\mu_{\sigma(n)}+k_{\sigma(n)})$.
\end{lemma}
\begin{proof}
Let $\mu':=\mu+k$. Using the Weyl group invariance of $W_{\mu'}$ we may assume, without loss of generality, that $\Re(\mu'_1)\ge\dots\ge\Re(\mu'_n)$. Also note that the assumption on the real parts of $\mu$ forces the above ordering to be strict. In fact, there is an $\epsilon>0$ such that $\min\{\Re(\mu'_i-\mu'_j)\mid i<j\}\ge\epsilon$. Now we do a change of variable int the integral of \eqref{defnwhittaker} to see that
\begin{align*}
    W_{\mu'}(a)
    &=d_nc(1,\mu')\delta^{1/2}(a)\prod_{i=1}^na_i^{(w\mu')_i}\int_{N}I_{\mu'}(wn)\psi(-ana^{-1})dn\\
    &\prec |c(1,\mu')|\delta^{1/2}(a)\prod_{i=1}^sa_i^{\Re((w\mu')_i)}.
\end{align*}
The last integral is absolutely convergent, as $\min\{\Re(\mu'_i-\mu'_j)\mid i<j\}\ge\epsilon$ (see \cite{J3}). Moreover, we can bound the last integral
uniformly in $\mu'$. A proof of this is essentially done in the proof of the absolute convergence of the Jacquet's integral in \cite[Chapter $5.8$]{G}. From that proof it can be inductively seen that
$$\int_{N}|I_{\mu'}(wn)|dn\ll \int_{N}I_{{\epsilon}}(wn)dn\ll_\epsilon 1,$$
where $I_{{\epsilon}}$ is the spherical section in the principal series with real parameters $\nu$ such that $\min\{(\nu_i-\nu_j)\mid i<j\}\ge\epsilon$.
\end{proof}

We record a rapid decay estimate of the Whittaker transform of the test function $f$, as follows, which we will need later in the proof.
\begin{lemma}\label{boundinnerproduct}
Let $\mu\in\C^n$ with $\Re(\mu_i)\ge 0$ and $\sum_{i=1}^n\Re(\mu_i)^2\le R$ for all $i$ and for some $R\ge 0$. 
Let $p$ be a fixed sufficiently large natural number.
Then
for each fixed
$f \in C_c^\infty(N \backslash G, \psi)$,
$$\langle f, W_\mu\rangle:=\int_{N \backslash G}f(g)\overline{W_\mu(g)}dg\ll_{R,p}d(\mu)^{-p}|c(1,\Im(\mu))|.$$
\end{lemma}
\begin{proof}
Recall from Lemma \ref{diffop} that $\D_R=1+R-C_G+2C_K$. As $W_\mu$ is right $K$-invariant, $C_KW_\mu=0$. Thus 
$$\D_RW_\mu=(1+R+T-\sum_{j=1}^n\mu_j^2)W_\mu.$$
We check that
$$|1+R+T-\sum_{j=1}^n\mu_j^2|\ge T+\sum_{j=1}^n|\Im(\mu_j)|^2\asymp d(\mu).$$
Let $Z$ denote the center of $G$;
we identify it with $\mathbb{R}^\times$ in the usual way.
Integrating by parts with respect to $\D^p$, we obtain
\begin{align*}
    \langle f, W_\mu\rangle
    &=(1+R+T-\sum_{j=1}^n\bar{\mu}_j^2)^{-p}\int_{N\backslash G}\overline{W_\mu(h)}\D^p f(h) d h.
\end{align*}
We use coordinates $h:=z\diag(h',1)k$ for $N\backslash G$ where $h'$ lies in the diagonal subgroup $A_{n-1}$ of $\GL_{n-1}(\R)$, $z\in \R^\times$ and $k\in K$. Correspondingly, we write $dh=d^\times z\frac{dh'}{\delta(h')|\det(h')|}dk.$ We rewrite the last integral as
$$\int_{A_{n-1}}W_\mu\left[\begin{pmatrix}h'&\\&1\end{pmatrix}\right]\int_K\int_{Z}\D^pf\left[z\begin{pmatrix}h'&\\&1\end{pmatrix}k\right]|z|^{\sum\bar{\mu}_i}|\det(h')|^{-1} d^\times zdk\frac{dh'}{\delta(h')}.$$
We apply Cauchy--Schwarz on the $h'$ integral and use \eqref{L2whittaker} and that $f\in C_c^\infty$ to obtain that the last integral is
$$\ll_{p,R,f} d(\mu)^{-p}\|W_\mu\|^2\asymp d(\mu)^{-p}\prod_{i,j}\Gamma_\R(1+\mu_i+\overline{\mu_j}).$$
We use Stirling's estimate to obtain that
\begin{align*}
    &\prod_{i,j}\Gamma_\R(1+\mu_i+\overline{\mu_j})
    \asymp_R\prod_{i\neq j}\Gamma_\R(1+\Im(\mu_i)-\Im(\mu_j)+\Re(\mu_i)+\Re(\mu_j))\\
    &\asymp_R\prod_{i\neq j}|\Im(\mu_i)-\Im(\mu_j)|^{R'}\Gamma_\R(1+\Im(\mu_i)-\Im(\mu_j))
    \ll_{R,p}d(\mu)^{R'}|c(1,\Im(\mu))|^2,
\end{align*}
where $R'$ is a bounded constant depending on $R$ and $n$. Making $p$ sufficiently large we conclude the proof.
\end{proof}

\section{Reduction of the proof
of the main results}\label{sec:reduction}
We adopt the standard convention from analytic number
theory of writing $\epsilon$ (or $\tau$) for a small positive fixed quantity,
whose precise value we allow to change from one line to the
    next.

Let $\Omega \subseteq \GL_{n}(\mathbb{R})$
be the bounded neighborhood of the identity element and $\iota$ be as in Theorem \ref{existence-strong}. Recall $\psi$ from \eqref{add-char-of-n}. 
We first construct $\We\in \W(\Pi,\tilde{\psi})$
using the theory of the Kirillov model.
We denote by $C_c^\infty(N \backslash G, \psi)$
the space of smooth functions $f : G \rightarrow \mathbb{C}$ satisfying
$f(n g) = \psi(n) f(g)$ for all $n \in N$
and for which the support of $f$
has compact image
in $N \backslash G$.
We choose an element $f \in C_c^\infty(N \backslash G, \psi)$
with the following properties:
\begin{itemize}
    \item 
$f$ is right $K$-invariant
\item
$f(h)\geq \iota$
for all $h \in \Omega$
\item
$\int_{N \backslash G} |f|^2\,dg = 1$
\end{itemize}
We now define $V$
by requiring that
\begin{equation}\label{newvec}
    \We\left[\begin{pmatrix}g&\\&1\end{pmatrix}\right]:=f(g),
\end{equation}

We remark that the sphericality assumption of $f$ is not essential. We refer to the discussion in \S\ref{sphericality-unnecessary} for details.

\begin{prop}\label{mainprop}
Let $\Pi$ be as in Theorem \ref{existence-strong}.
For every $\delta>0$ and $h$ in a fixed bounded neighbourhood around the identity in $G$ there exists a $\tau>0$ such that 
$$\left|\We\left[\begin{pmatrix}h&\\c/C(\Pi)&1\end{pmatrix}\right]-\We\left[\begin{pmatrix}h&\\&1\end{pmatrix}\right]\right|<\delta,$$
for $c\in \R^n$ with $|c|<\tau$.
Here $\We$ is the vector chosen in \eqref{newvec}.
\end{prop}

We will now assume Proposition \ref{mainprop} and prove Theorems \ref{existence-weak} and \ref{existence-strong}. In the next two sections we will prove Proposition \ref{mainprop}. In the following Lemma we first prove a weaker version of Theorem \ref{existence-strong}.

\begin{lemma}\label{main-reduction}
For every $\delta>0$ and $h$ in a fixed bounded neighbourhood around the identity in $G$ there exists a $\tau>0$ such that 
\begin{itemize}
    \item the normalization $\|\We\|_{\W(\Pi,\tilde{\psi})}=1$,
    with the norm taken in the Kirillov model
    (\S\ref{sec:whittaker-kirillov}),
    \item the lower bound
    $\We\left[\begin{pmatrix}h&\\&1\end{pmatrix}\right] \geq \iota$ for all $h \in \Omega$,
    and
    \item 
    for all $g\in K_0(C(\Pi),\tau)$, 
 and for $h \in \Omega$,
\begin{equation*}
\left|\We\left[\begin{pmatrix}h&\\&1\end{pmatrix}g\right]-
\omega_{\Pi}(d_g)
\We\left[\begin{pmatrix}h&\\&1\end{pmatrix}\right]\right|<\delta.
\end{equation*}
\end{itemize}
Here $\We$ is the vector chosen in \eqref{newvec} and $\omega_\Pi$ and $d_g$ are as in Theorem \ref{existence-weak}.
\end{lemma}

\begin{proof}
We choose $V$ as in \eqref{newvec}.
We note that first two requirements of $V$ are automatically satisfied by the choice \eqref{newvec}. To prove the invariance property of $V$ we claim the following. 

For every $\delta>0$ and $h$ in a given fixed bounded set there exists $\tau>0$ such that for all $\begin{pmatrix}g'&\\c&1\end{pmatrix}\in K_0(C(\Pi),\tau)$,
$$\left|\We\left[\begin{pmatrix}hg'&\\c&1\end{pmatrix}\right]-\We\left[\begin{pmatrix}h&\\&1\end{pmatrix}\right]\right|<\delta.$$

This claim is sufficient. To see that, first we note the following Iwahori-type decomposition that for $d\neq 0$
$$\GL_{n+1}(\R)\ni\begin{pmatrix}A&b\\c&d\end{pmatrix}=d\begin{pmatrix}1_n&b/d\\&1\end{pmatrix}\begin{pmatrix}A/d-bc/d^2\\&1\end{pmatrix}\begin{pmatrix}1_n&\\c/d&1\end{pmatrix}.$$
Hence we can assume that any $g\in K_0(C(\Pi),\tau)$ is of the form 
$$g=d_g\begin{pmatrix}1_n&b\\&1\end{pmatrix}\begin{pmatrix}g'\\c&1\end{pmatrix},$$ with $g'\in G$ such that $\|g'-1\|\ll \tau$ and $|b|, |d-1|, C(\Pi)|c|<\tau$. Therefore,
$$\We\left[\begin{pmatrix}h&\\&1\end{pmatrix}g\right]=\omega_\Pi(d_g)\tilde{\psi}\left[\begin{pmatrix}1_n&hb\\&1\end{pmatrix}\right]\We\left[\begin{pmatrix}hg'\\c&1\end{pmatrix}\right].$$
Therefore using unitarity of $\omega_\Pi$ and $\tilde{\psi}$,
\begin{align*}
&\left\lvert
\We\left[\begin{pmatrix}h&\\&1\end{pmatrix}g\right]-\omega_\Pi(d_g)\We\left[\begin{pmatrix}h&\\&1\end{pmatrix}\right]
\right\rvert\\
&=
\left\lvert
(e(hb.e_n)-1)\We\left[\begin{pmatrix}hg'\\c&1\end{pmatrix}\right]
    +\We\left[\begin{pmatrix}hg'\\c&1\end{pmatrix}\right]-\We\left[\begin{pmatrix}h&\\&1\end{pmatrix}\right]
\right\rvert.
\end{align*}
Now using the claim above and the fact that $$\We\left[\begin{pmatrix}hg'\\c&1\end{pmatrix}\right]\ll 1,$$
which also follows from the same claim, we obtain the required invariance property by making $b$ small enough.

Now we turn to prove the claim. We note that from \eqref{newvec}, for $h$ in a compact set in $G$ there exists $\tau$ small enough with $\|g'-1\|<\tau$ such that $$\left|\We\left[\begin{pmatrix}hg'\\&1\end{pmatrix}\right]-\We\left[\begin{pmatrix}h&\\&1\end{pmatrix}\right]\right|<\delta.$$
Now applying a triangle inequality in the following
\begin{align*}
    &\We\left[\begin{pmatrix}hg'\\c&1\end{pmatrix}\right]-\We\left[\begin{pmatrix}h&\\&1\end{pmatrix}\right]\\
    &=\We\left[\begin{pmatrix}hg'\\c&1\end{pmatrix}\right]-\We\left[\begin{pmatrix}hg'\\&1\end{pmatrix}\right]+\We\left[\begin{pmatrix}hg'\\&1\end{pmatrix}\right]-\We\left[\begin{pmatrix}h&\\&1\end{pmatrix}\right],
\end{align*}
along with Proposition \ref{mainprop} we prove the claim.
\end{proof}

\begin{proof}[Proofs of Theorem \ref{existence-weak} and Theorem \ref{existence-strong} assuming Proposition \ref{mainprop}]
First of all we can use a similar technique as in Lemma \ref{main-reduction} and reduce to the case $g\in K_1(C(\Pi),\tau)$ using the unitarity of $\omega_\Pi$, as follows,
$$|\Pi(g)\We-\omega_\Pi(d_g)\We|= |\Pi(g/d_g)\We-\We|.$$

Set $\delta_0 := \min(\delta,1) \iota/2$ where $\delta$ and $\iota$ are as in the statement of Theorem \ref{existence-strong}.
Let $\We_0 \in \mathcal{W}(\Pi,\tilde{\psi})$
and $\tau_0 > 0$
be as in \eqref{newvec}, so that 
\[\We_0\left[\begin{pmatrix}h&\\&1\end{pmatrix}\right]\geq \iota\]
and for all $g \in K_1(C(\Pi), \tau_0), h \in \Omega$,
\[\left|\We_0\left[\begin{pmatrix}h&\\&1\end{pmatrix}g\right] -  \We_0\left[\begin{pmatrix}h&\\&1\end{pmatrix}\right]\right| \leq \delta_0.\]
Existence of such $\We_0$ follows from Lemma \ref{main-reduction} and the toy theorem \ref{toy-theorem}.

Let $\xi$ be the $L^1$-normalized
characteristic function of $K_1(C(\Pi),\tau_1)$.
Note that there exists $\tau_2>0$ such that for $g\in K_1(C(\Pi),\tau_2)$,
\[\|g \ast \xi - \omega_{\Pi}(d_g)\xi \|_{L^1} \leq \| g\ast \xi -\xi\|+|\omega_\pi(d_g)-1|\le \delta_0,\]
which follows from Lemma \ref{folner-lemma} and the toy theorem \ref{toy-theorem}.
Here $g \ast \xi(h) := \xi(g^{-1} h)$,
so that
$\Pi(g) \Pi(\xi) = \Pi(g \ast \xi)$.
Set $\We_1 := \Pi(\xi) \We_0$.
Since $\|\We_0\| = 1$,
we then have by the triangle inequality
that for $g \in K_1(C(\Pi), \tau_1)$,
\[\|\Pi(g) V_1 - \omega_\Pi(d_g)V_1\|
\leq \|g \ast \xi - \omega_\Pi(d_g)\xi\|_{L^1(G)} \leq \delta_0,\]
and for $\tau_1$ sufficiently small in terms of $\tau_0$
\begin{multline*}
    \left|\We_1\left[\begin{pmatrix}h&\\&1\end{pmatrix}g\right]-\We_1\left[\begin{pmatrix}h&\\&1\end{pmatrix}\right]\right|\\
    \le \int_{K_1(C(\Pi),\tau_1)} \xi(t) \left|\We_0\left[\begin{pmatrix}h&\\&1\end{pmatrix}gt\right] - \We_0\left[\begin{pmatrix}h&\\&1\end{pmatrix}t\right]\right|\, dg\leq 2\delta_0.
\end{multline*}
Also,
for $h \in \Omega$,
we
have \[\We_1\left[\begin{pmatrix}h&\\&1\end{pmatrix}\right] - \We_0\left[\begin{pmatrix}h&\\&1\end{pmatrix}\right] =
\int_{K_1(C(\Pi),\tau_0)} \xi(t) \left(\We_0\left[\begin{pmatrix}h&\\&1\end{pmatrix}t\right] - \We_0\left[\begin{pmatrix}h&\\&1\end{pmatrix}\right]\right)\, dt,\]
hence
\[\left|\We_1\left[\begin{pmatrix}h&\\&1\end{pmatrix}\right]-\We_0\left[\begin{pmatrix}h&\\&1\end{pmatrix}\right]\right| \leq \iota/2,\]
so in particular \[\We_1\left[\begin{pmatrix}h&\\&1\end{pmatrix}\right] \geq \iota - \iota/2 = \iota/2 > 0,\]
thus $\|\We_1\|\asymp 1$.
It follows that the vector
$V_1/\|V_1\|\in \W(\Pi,\tilde{\psi})$ and its image in $\Pi$ satisfy the required conclusions of Theorem \ref{existence-strong} and Theorem \ref{existence-weak}, respectively.
\end{proof}

\section{Proof of Proposition \ref{mainprop}}\label{sec:proof-mainprop}
To prove Proposition \ref{mainprop} we need an uniform bound of an Weyl element shifted newvector, as in Proposition \ref{heart}. Proving this proposition is actually the heart and the most difficult part of the article. 

\begin{prop}\label{heart}
Let $l\in\N$. Let $g=ak$ with $a:=\diag(a_1,\dots,a_n)\in A$, $k\in K$. Then
$$\D^l\We\left[\begin{pmatrix}C(\Pi)&\\&g\end{pmatrix}w\right]\prec_{l,N}\delta^{1/2}(a)\min(a_1^N,1),$$ where $w$ is the long Weyl element in $\GL_{n+1}(\R)$ and $\D$ is the differential operator defined in \eqref{defdiffop}.
\end{prop}
We will prove the above proposition in the next section. For now we will see how to derive Proposition \ref{mainprop} from Proposition \ref{heart}.

\begin{lemma}\label{brutaldiff}
Let $p\in \N$ and $\D$ as in \eqref{defdiffop}. Then
\begin{enumerate}
    \item For $\sigma\in \C$ small we have $\D^p|\det(g)|^\sigma\asymp_p|\det(g)|^\sigma$
    \item As a function of $h$ we have $\D^p(e(cw'h^{-1}e_1)-1)\ll_p \sum_{r=1}^{2p+1}|c|^r|h^{-1}e_1|^r$.
\end{enumerate}
\end{lemma}
\begin{proof}
Recall the definition of $\D$.
It is straightforward to check that 
$$\D^p|\det(g)|^\sigma=(1-n\sigma^2)^p|\det(g)|^\sigma,$$
which proves $(1)$.

Let $x:=2\pi i cw'$. So
$$e(cw'h^{-1}e_1)-1=\sum_{r=1}^\infty\frac{(xh^{-1}e_1)^r}{r!}.$$
Let $f_r:=\frac{(xh^{-1}e_1)^r}{r!}$. Then it is straightforward to check that\footnote{The computation will be intuitive after one does the similar computation in the $n=1$ case, in which case $\D$ becomes the simpler differential operator $1-(h\partial_h)^2$.}
$$\D f_r = f_r-(xh^{-1}e_1)f_{r-1}-|x|^2|h^{-1}e_1|^2f_{r-2}.$$
One can also easily check that
$$\D (|x|^2|h^{-1}e_1|^2)=-2n|x|^2|h^{-1}e_1|^2
.$$
Therefore inducting on $p$ and summing over $r$ we conclude $(2)$.
\end{proof}

\begin{lemma}\label{whittaker-bound}
Let $\pi$ be an irreducible generic unitary tempered representation of $G$ and $\V\in \W(\pi,\psi)$ be any smooth vector. Let $a=\diag(a_1,\dots,a_n)\in A$, and $k\in K$ and $L>0$.
For any $\eta$ small enough and $N$ large enough
$$\V(ak)\prec_{N} \delta^{1/2-\eta}(a)\prod_{i=1}^{n-1}\min(1,(a_{i+1}/a_{i})^{N})S_{p}(\V),$$
for $p$ depends on $N$ and the group $G$ only.
\end{lemma}

This lemma is essentially in \cite[Proposition $3.5$]{J1} but $\delta^{-\eta}(a)$ is replaced by $(1+\log\|a\|)^M$ which would be enough for our purpose. We give a proof of this lemma along the lines of \cite[Proposition $3.2.3$]{MV} which can be generalized for non-tempered $\pi$.

\begin{proof}
As $\V(ak)=\pi(k)\V(a)$ and $S_p(\pi(k)\V)\asymp_K S_p(\V)$ (see e.g. \cite[S1b]{MV}) it is enough to prove the Lemma for $k=1$. 

We will prove this inducting on $n$. Note that for $n=2$ this is proved in \cite[Proposition $3.2.3$]{MV}. We will generalize their idea of proof for general $n$. 

We note that there exists an $Y\in\g$ such that
$$d\pi(Y)\V(a) = (a_{n-1}/a_n)\V(a).$$
Let $W_1:=d\pi(Y^N)W\in\W(\pi,\psi)$. It is enough to show that
$$\V_1(a)\ll_{N,\eta} \delta^{1/2-\eta}(a)\prod_{i=1}^{n-2}\min(1,(a_{i+1}/a_i)^N)S_{p}(\V),$$
for some $p$ depending on $N$.

For a generic irreducible unitary tempered representation $\pi'$ of $\GL_{n-1}(\R)$ let $\B(\pi'):=\{\V'\}$ be an orthonormal basis of $\pi'$ consisting of eigenvectors of $\D'$ by diagonalizing it such that $\D'\V'=\lambda_{\V'}\V'$, where $\D'$ is the analogous differential operator on $\GL_{n-1}(\R)$ as in \eqref{defdiffop}. We obtain using Whittaker--Plancherel formula \eqref{genwhitplan} on $\GL(n-1)$
\begin{multline*}
    |\det(a)|^{-\sigma}\V_1(a)=\omega_\pi(a_n)\V_1(a/a_n)|\det(a)|^{-\sigma}\\
    =\omega_\pi(a_n)\int_{\widehat{\GL_{n-1}}}\sum_{\V'\in \B(\pi')}\V'(\tilde{a}/a_n)Z_{\V_1,\sigma}(\V')d\mu_p(\pi'),
\end{multline*}
where $\tilde{a}:=\diag(a_1,\dots,a_{n-1})$, and 
$$Z_{\V_1,\sigma}(\V'):=\int_{N_{n-1}\backslash\GL_{n-1}}\V_1\left[\begin{pmatrix}h&\\&1\end{pmatrix}\right]\overline{\V'(h)}|\det(h)|^{-\sigma}dh.$$
We integrate by parts $M$ times in the above integral with respect to $\D'$ to write
$$Z_{\V_1,\sigma}(\V')=\lambda_{\V'}^{-M}\int_{N_{n-1}\backslash\GL_{n-1}}{\D'}^M\left(\V_1\left[\begin{pmatrix}h&\\&1\end{pmatrix}\right]|\det(h)|^{-\sigma}\right)\overline{\V'(h)}dh.$$
We use Lemma \ref{brutaldiff} to evaluate $\D'|\det|^{-\sigma}$ and write the $h$-integral in the Iwasawa $a'k'$ coordinates. We appeal to \cite[Proposition $3.5$]{J1}, which states that that for any $\V$ in a tempered representation one has
$$\V(a)\prec_{N} \delta^{1/2}(a)\prod_{i=1}^{n-1}\min(1,(a_{i+1}/a_{i})^{N})S_{p'}(\V),$$
for some $p'$ depending only on $N$, and apply to $\V'$ and  ${\D'}^M\V_1$ which lie in tempered representations $\pi'$ and $\pi$, respectively.
Noting
$$\delta\left[\begin{pmatrix}a&\\&1\end{pmatrix}\right]=|\det(a)|\delta(a)$$
we conclude that
\begin{multline*}
Z_{\V_1,\sigma}(W')\ll S_{p}(\V_1)\lambda_{\V'}^{p'-M}\int_{A_{n=1}}(\|a'\|+\|{a'}^{-1}\|)^{\epsilon}\min(1,{a'}_{n-1}^{-N})\\
\prod_{i=1}^{n-2}\min(1,(a'_{i+1}/a'_{i})^N)|\det(a')|^{1/2-\sigma}da'.
\end{multline*}
The above integral is absolutely convergent for $\sigma<1/2$. To see that
note that 
$$\int_{\R_{>0}}{a'}_1^{1/2-\sigma\pm\epsilon}\min(1,(a'_1/a'_2)^{-N})d^\times a'_1\ll {a'}_2^{1/2-\sigma\pm\epsilon},$$
for $\epsilon$ sufficiently small. Proceeding like this for $i=1,2,\dots$ we end up in estimating
$$\int_{\R_{>0}}\min(1,{a'}_{n-1}^{-N}){a'}_{n-1}^{(n-1)(1/2-\sigma\pm\epsilon)}d^\times a'_{n-1}.$$
This is again convergent for $\epsilon$ small enough and thus we have
$$Z_{\V_1,\sigma}(\V')\ll S_{p}(\V_1)\lambda_{\V'}^{p'-M},$$
for some $p$ depending on $M$ and $N$.
Using induction hypothesis on $\V'$ we get
$$\V'(\tilde{a}/a_n)\ll_{\eta,N} \delta^{1/2-\eta}(\tilde{a})\prod_{i=1}^{n-2}(1,(a_{i+1}/a_i)^N)\lambda_{\V'}^q,$$
for some $q$ depending on $N$. 
We choose $\sigma=1/2-\eta$ for $\eta>0$ arbitrarily small and note that $$|\det(\tilde{a}/a_n)|^{1/2-\eta}\delta^{1/2-\eta}(\tilde{a})=\delta^{1/2-\eta}(a).$$ Finally, we apply Lemma \ref{sobolev-norm-property} to see that the sum over $\V'$ and the integral over $\pi'$ are convergent if $M$ is large enough, and we conclude the proof.
\end{proof}

\begin{proof}[Proof of Proposition \ref{mainprop} assuming Proposition \ref{heart}]
Recall that $\Pi$ is $\theta$-tempered as in Theorem \ref{existence-strong}. For every $\pi\in\hat{G}$ we choose an orthonormal basis $\B(\pi):=\{\V\}$ of $\pi$ consisting of eigenvectors of $\D$ by diagonalizing it. Applying the Whittaker--Plancherel formula \eqref{genwhitplan} of $G$ we get, for $0<\sigma<1/2-\theta$, that
\begin{align*}
    &|\det(g)|^{-\sigma} \We\left[\begin{pmatrix}g&\\c&1\end{pmatrix}\right]=|\det(g)|^{-\sigma}\Pi\left[\begin{pmatrix}1_n&\\c&1\end{pmatrix}\right]\We\left[\begin{pmatrix}g&\\&1\end{pmatrix}\right]\\
    &=\int_{\hat{G}}\sum_{\V\in\B(\pi)}\V(g)\int_{N\backslash G}\Pi\left[\begin{pmatrix}1_n&\\c&1\end{pmatrix}\right]\We\left[\begin{pmatrix}h&\\&1\end{pmatrix}\right]\overline{\V(h)}|\det(h)|^{-\sigma} dhd\mu_p(\pi).
\end{align*}
The inner integral can be checked to be absolutely convergent by Lemma \ref{whittaker-bound}.
We apply the $\GL(n+1)\times\GL(n)$ local functional equation \eqref{lfe} and the $N$-equivariance for $\We$ to obtain the inner integral above equals to
$$\omega_{\bar{\pi}}(-1)^n\gamma(1/2-\sigma,\Pi\otimes\bar{\pi})^{-1}\int_{N\backslash G}e(cw'h^{-1}e_1)\We\left[\begin{pmatrix}1&\\&h\end{pmatrix}w\right]\overline{\V(hw')}|\det(h)|^{-\sigma} dh.$$
Changing variable $h\mapsto C(\Pi)^{-1}h$ in the latter integral we conclude that
\begin{equation}\label{finalform}
\begin{split}
    &\We\left[\begin{pmatrix}g&\\c/C(\Pi)&1\end{pmatrix}\right]- \We\left[\begin{pmatrix}g&\\&1\end{pmatrix}\right]\\
    &=\int_{\hat{G}}\omega_{\bar{\pi}}((-1)^{n}C(\Pi)^{-1})C(\Pi)^{n\sigma}\gamma(1/2-\sigma,\Pi\otimes\bar{\pi})^{-1}\sum_{\V\in\B(\pi)}\V(g)|\det(g)|^{\sigma} \\
    &\hphantom{rubbish}\int_{N\backslash G}(e(cw'h^{-1}e_1)-1)\We\left[\begin{pmatrix}C(\Pi)&\\&h\end{pmatrix}w\right]\overline{\V(hw')}|\det(h)|^{-\sigma} dhd\mu_p(\pi).
\end{split}
\end{equation}
From Lemma \ref{boundgammafactor} we have
$$C(\Pi)^{n\sigma}\gamma(1/2-\sigma,\Pi\otimes\bar{\pi})^{-1}\ll C(\pi)^{(n+1)\sigma}.$$
In the last integral of \eqref{finalform} we integrate by parts by $\D$ as
$$\int_{N\backslash G}\D^L\left((e(cw'h^{-1}e_1)-1)\We\left[\begin{pmatrix}C(\Pi)&\\&h\end{pmatrix}w\right]|\det(h)|^{-\sigma}\right)\D^{-L}\overline{\V(hw')} dh.$$
We write $N\backslash G$ with $(a,k)$ coordinates. Using Lemma \ref{brutaldiff}, Proposition \ref{heart}, and Lemma \ref{whittaker-bound} we estimate the last integral by
$$\prec S_{p-L}(\V)\sum_{1\le r\ll L}|c|^r\int_{A} a_1^{-r}\prod_{i=1}^{n-1}\min(1,(a_{i+1}/a_i)^M)\min(1,a_1^N)\frac{da}{|\det(a)|^\sigma},$$
where $p$ depends on $M$ and $n$. Working as in the proof of Lemma \ref{whittaker-bound} we check that the above $A$-integral is absolutely convergent and hence \eqref{finalform} is bounded by
$$\ll_{p,L}|c|\int_{\hat{G}}C(\pi)^{(n+1)\sigma}\sum_{\V\in \B(\pi)}|\V(g)||\det(g)|^\sigma S_{p-L}(\V)d\mu_p(\pi).$$
As $g$ varies in a compact set $\Omega$, we use Lemma \ref{whittaker-bound} to bound 
$$\V(g)|\det(g)|^\sigma\ll_\Omega S_q(\V),$$
for all $\V\in \B(\pi)$, for some fixed $q$. Finally, making $L$ sufficiently large and appealing to Lemma \ref{sobolev-norm} we conclude.
\end{proof}

\section{Proof of Proposition \ref{heart}}\label{sec:proof-heart}
For $1\le s\le n$ we define a property $\pop(s)$, which stands for \emph{Partial Ordering at the Pivot $s$}, of the elements $a\in A$.
\begin{defn}\label{POP}
    We say an element $a\in A$ satisfies the property $\pop(s)$ for some $1\le s\le n$ if $$\max\{a_1,\dots,a_s\}\le \min\{1,\min\{a_{s+1},\dots,a_n\}\}.$$
\end{defn}
The proof of Proposition \ref{heart} will be proved in two different cases, depending whether $a$ satisfies $\pop(s)$ for some $s$ or not. When $a$ does not satisfy $\pop(s)$ for any $s$ the proof would be relatively easier which can be seen at the end of this section. Here we prove the required bound of $W_\mu$, which is the spherical vector in $\pi_\mu$ as defined in \eqref{defnwhittaker}, on the elements $a$ which do not satisfy $\pop(s)$ for any $s$. 

\begin{lemma}\label{notpop}
Let $a\in A$ does not satisfy $\pop(s)$ for any $1\le s\le n$ then 
$$\frac{W_\mu(a)}{c(1,\mu)}\prec_N \delta^{1/2}(a)\min(1,a_1^N)d(\mu)^{O_N(1)},$$
for $N$ large enough.
\end{lemma}
\begin{proof}
If $a_1\ge 1$ Lemma \ref{rapiddecay} with $M=0$ immediately implies this lemma. So we will assume that $a_1<1$. Note that, as we do not have $\pop(1)$ there exists $1<l'\le n$, such that $a_{l'}<a_1$. Also note that, as we do not have $\pop(n)$ there exists $1<r'\le n$ such that $a_{r'}>1$. Let
$$l:=\max\{l'\mid a_{l'}<a_1\},\quad r:=\min\{r'\mid a_{r'}>1\}.$$
If $r=n$ then we have $\pop(n-1)$, so $r<n$. If $l> r$ then from Lemma \ref{rapiddecay} we estimate that
$$\frac{W_\mu(a)}{c(1,\mu)}\prec_N \delta^{1/2}(a)d(\mu)^{O_N(1)}\prod_{j=r}^{l-1}(a_{j+1}/a_j)^N\le \delta^{1/2}(a)a_1^Nd(\mu)^{O_N(1)},$$
and we are done. Thus we may assume that $l<r$ (note that $l\neq r $ as $a_l<a_1<1<a_r$). Now we define
$$a_L:=\max\{a_1,\dots,a_l\},\quad a_R:=\min\{a_r,\dots,a_n\}.$$
Note that, $L<l$ clearly; also $R>r$, as if $R=r$ then we have $\pop(r-1)$. Also note that, if $a_L\ge a_R$ then
$$\frac{W_\mu(a)}{c(1,\mu)}\prec_N \delta^{1/2}(a)d(\mu)^{O_N(1)}\prod_{j=L}^{l-1}(a_{j+1}/a_j)^N\prod_{i=r}^{R-1}(a_{i+1}/a_i)^N\le \delta^{1/2}(a)a_1^N d(\mu)^{O_N(1)}.$$
Thus we may now assume that $a_L<a_R$. 
Hence, we have $1<l<r<n$ such that
\begin{equation}\label{stage1}
    \max\{a_1,\dots,a_l\}\le \min\{1,\min(a_r,\dots,a_n\}\}.
\end{equation}
Now we define 
$$l_1:=\max\{l'\mid a_{l'}\le a_L\},\quad r_1:=\min\{r'\mid a_{r'}\ge a_R\}.$$
Note that, $l_1\ge l$ clearly; in fact $l_1>l$, otherwise we will have $\pop(l)$. Similarly, $r_1\le r$ clearly; in fact $r_1<r$, otherwise  we will have $\pop(r-1)$. Also $l_1\neq r_1$ as $a_{l_1}\le a_L<a_R\le a_{r_1}$. But if $l_1>r_1$ then
\begin{align*}\frac{W_\mu(a)}{c(1,\mu)}
&\prec_N\delta^{1/2}(a)d(\mu)^{O_N(1)}\prod_{j=L}^{l-1}(a_{j+1}/a_j)^N\prod_{i=r_1}^{l_1-1}(a_{i+1}/a_i)^N\prod_{k=r}^{R-1}(a_{i+1}/a_i)^N
\\
&\le \delta^{1/2}(a)a_1^Nd(\mu)^{O_N(1)}.
\end{align*}
Thus we may assume that $l_1< r_1$, as $l_1=r_1$ would have $\pop(l_1-1)$. We now define 
$$a_{L_1}:=\max\{a_1,\dots,a_{l_1}\},\quad a_{R_1}:=\min\{a_{r_1},\dots,a_n\}.$$
Note that, if $a_{L_1}=a_{l_1}$ then we will have $\pop(l_1)$, so $L_1<l_1$. Similarly, if $a_{R_1}=a_{r_1}$ then we will have $\pop(r_1-1)$, so $R_1>r_1$. Also note that, if $a_{L_1}\ge a_{R_1}$ then
\begin{align*}
    \frac{W_\mu(a)}{c(1,\mu)}
    &\prec_N\delta^{1/2}(a)d(\mu)^{O_N(1)}\prod_{j=L}^{l-1}(a_{j+1}/a_j)^N\prod_{p=L_1}^{l_1-1}(a_{p+1}/a_p)^N\prod_{i=r_1}^{R_1-1}(a_{i+1}/a_i)^N\prod_{k=r}^{R-1}(a_{i+1}/a_i)^N\\
    &\le \delta^{1/2}(a)a_1^Nd(\mu)^{O_N(1)}.
\end{align*}
So we may assume that $a_{L_1}<a_{R_1}.$ Thus we have got nested pairs $1<l<l_1<r_1<r<n$ such that
\begin{equation}\label{stage2}
    \max\{a_1,\dots,a_{l_1}\}\le \min\{1,\min(a_{r_1},\dots,a_n\}\}.
\end{equation}
Proceeding in this way we will eventually, as there are only finitely, say $P$, many steps, we will eventually obtain $l_P=r_P$ or $l_P=r_P-1$ with similar  properties as in \eqref{stage1} or \eqref{stage2}. In either case, we will arrive at $\pop$, thus a contradiction.
\end{proof}

\subsection{A few auxiliary notations}
In the rest of this section we will prove the required decomposition of $W_\mu$ into partial $M$-Whittaker functions and prove the required bounds of them. Here by $N_r, A_r, K_r...$ etc., we will denote the maximal unipotent subgroup of upper triangular matrices, the positive diagonal subgroup, the maximal compact $\mathrm{O}(r)$... in $\GL_r(\R)$, respectively. For $\mu\in \C^r$ by $W_\mu$ (suppressing $r$) we will denote the spherical Whittaker function (with our chosen normalization as in \eqref{L2whittaker}) on $\GL_r(\R)$ with parameters $\mu$. For $a:=\diag(a_1,\dots, a_n)\in A$, by $a^r$ we will denote the element $\diag(a_1,\dots, a_r)\in A_r$. Let us denote $(\alpha)^s:=(\alpha_1,\dots,\alpha_s)$ for $s\le n$ and $\sum\alpha:=\alpha_1+\dots\alpha_n$, for any $\alpha\in\C^n$. We will abbreviate the condition $\Re(\alpha_i)\ge 0$ (resp. $>0$) for all $i$ by $\Re(\alpha)\ge 0$ (resp. $>0$). By a \textit{regular} $\alpha$ we mean that the coordinates $\alpha_j$ are distinct. By $S_r$ we will denote the symmetric group of $r$ letters, which is isomorphic to the Weyl group of $\GL(r)$.

We record that the residue of $\Gamma_\R(s)$ at $s=-2n$ for any $n\in \Z_{\ge 0}$ is $2\frac{(-\pi)^n}{n!}=\frac{2(-1)^n}{\Gamma_\R(2n+2)}.$ 
Let $\nu\in \C^r$ and $\nu'\in \C^{r'}$ with $r>r'$. Let 
$$\{\nu_i-\nu'_j\mid 1\le i\le r, 1\le j\le r'\}=A(\nu,\nu')\cup B(\nu,\nu'),$$
where $A(\nu,\nu')$ is the set of elements which are of form $2\Z_{\le 0}$ i.e. a possible pole of $\Gamma_\R$ and $B(\nu,\nu')$ is the compliment. We define
\begin{equation}
L(\nu,\nu'):=\prod_{a\in A(\nu,\nu')}\res_{s=a}\Gamma_\R(s)\prod_{b\in B(\nu,\nu')}\Gamma_\R(b).
\end{equation}
We will use these notations in the rest of the section.

\subsection{Integral representation of the spherical Whittaker function}
Note that
applying the Whittaker--Plancherel formula \eqref{genwhitplan} and the $\GL(n+1)\times\GL(n)$ local functional equation \eqref{lfe} we get that
\begin{align*}
    \We\left[\begin{pmatrix}C(\Pi)&\\&g\end{pmatrix}w\right]
    =\int_{\hat{G}}\omega_\pi(-1)^n\Theta(\pi,\Pi)\sum_{\V\in\B(\pi)}{\V(gw')}\langle f,\V\rangle d\mu_p(\pi),
\end{align*}
We choose $\B(\pi):=\{\V\}$ containing an ONB of $\pi$ of $K$-types. Recall that, $f$ is spherical (see \eqref{newvec}), so we conclude that for all $\V\in \W(\pi,\psi)$ with non-trivial $K$-type $\langle f, \V\rangle=0$.
Applying $\D^l$ we obtain that
\begin{equation}\label{reducespherical}
\D^l\We\left[\begin{pmatrix}C(\Pi)&\\&g\end{pmatrix}w\right]
=\int_{\hat{G}_0}\Theta(\pi,\Pi)\D^l{W_\pi(g)}\langle f, W_\pi\rangle d\mu_p(\pi),
\end{equation}
where $W_\pi$ is an $L^2$-normalized spherical vector in $\pi$.
Thus we can also conclude that $\We\left[\begin{pmatrix}C(\Pi)&\\&g\end{pmatrix}w\right]$ is spherical. Hence it is enough to prove Proposition \ref{heart} for $g=a\in A$.
Finally, choosing $W_{\pi_\mu}:=\frac{W_\mu}{\|W_\mu\|}$, using  \eqref{spherwhitplan}, \eqref{L2whittaker}, and \eqref{planchereldensity} along with the description of the tempered spherical dual of $G$ we rewrite \eqref{reducespherical} as
\begin{equation}\label{finalreducespherical}
\D_M^l\We\left[\begin{pmatrix}C(\Pi)&\\&a\end{pmatrix}w\right]
=\int_{(0)^n}\Theta(\mu,\Pi)\D^l_M{W_\mu(a)}\langle f,W_\mu\rangle\frac{d\mu}{|c(\mu)|^{2}}.
\end{equation}
 
We record following integral representation of the spherical Whittaker function which is a corollary of Stade's formula.
\begin{lemma}\label{splKLW}
Let $\nu\in \C^{r+1}$ with $\Re(\nu_i)\ge 0$. Then
$$W_{\nu}(a^r)=\kappa_ra_r^{\sum\nu}\int_{(0)^{r-1}}W_{\nu'}(a^{r-1}/a_r)\frac{L(\frac{1}{2},\pi_\nu\otimes{\pi_{-\nu'}})}{c(\nu')c(-\nu')}d\nu',$$
for some absolute constant $\kappa_r$ (depending only on $r$).
\end{lemma}
The unspecified constant appears because of our chosen normalization in \eqref{L2whittaker}. From now on we will denote any unspecified constant which only depends on $n,...$ by $\kappa_{n,...}$, i.e. $\kappa_{n,...}$ may vary from line to line.
\begin{proof}
Using \eqref{spherwhitplan} for $\GL_{r-1}(\R)$ and proceeding as before we can write 
\begin{align*}
    W_{\nu}(a^r)
    &=a_r^{\sum\nu}W_\nu\left[\begin{pmatrix}a^{r-1}/a_r&\\&1\end{pmatrix}\right]\\
    &=a_r^{\sum\nu}\int_{(0)^{r-1}}W_{\nu'}(a^{r-1}/a_r)\int_{N_{r-1}\backslash\GL_{r-1}}W_\nu\left[\begin{pmatrix}t&\\&1\end{pmatrix}\right]\overline{W_{\nu'}(t)}dt\frac{d\nu'}{|c(\nu')|^2}.
\end{align*}
We conclude the proof noting that the inner integral is a constant multiple of $L(\frac{1}{2},\pi_\nu\otimes\overline{\pi_{\nu'}})$ by Stade's formula \cite[Theorem $3.4$]{St1}.
\end{proof}
We abbreviate $\delta^{-1/2}\V$, for any $\GL_r$ Whittaker function $\V$, by $\V'$. Note that if $\Re(\nu)>0$ then shifting contour $\nu'\mapsto\nu'+1/2$ in the $\nu'$ integral in Lemma \ref{splKLW} (without crossing any pole) we obtain that
\begin{equation}\label{modifiedKLW}
    W'_\nu(a)=\kappa_r a_r^{\sum\nu}\int_{(0)^{r-1}}W'_{\nu'}(a^{r-1}/a_r)\frac{L(\nu,\nu')}{c(\nu')c(-\nu')}d\nu'.
\end{equation}
From now on we will work with $\V'$ instead of $\V$\footnote{This is because if we work with $\V$ the normalization by $\delta^{-1/2}$ will appear in every equation, somewhat unimportantly.}.


We fix $1\le s\le n$ from now on. Let $\Re(\mu)>0$ be small enough. In \eqref{modifiedKLW} we shift the contours of $\nu$ integrals to some positive quantity so that the integrand does not cross any polar hyperplanes. For instance, we may choose $\Re(\nu)>0$ such that $\max_j{\Re(\nu_j)}<\min_i\Re(\mu_i)$. We obtain
$$W'_\mu(a)=\kappa_na_n^{\sum\mu}\int W'_{\nu}(a^{n-1}/a_n)\frac{L(\mu,\nu)}{c(\nu)c(-\nu)}d\nu,$$
where the contours are the vertical lines with real parts as said above. In this section, we will mostly, not specify these type of contours explicitly. If the contours are unspecified then we will implicitly assume that the contours are vertical lines on the left of all possible poles and very close to the vertical lines with real parts being zero, as described above.
In the RHS we expand $W'_\nu$ using \eqref{modifiedKLW} exactly same as before and obtain
$$W'_\mu(a)=\kappa_na_n^{\sum\mu}\int\left(\frac{a_{n-1}}{a_n}\right)^{\sum\nu}\frac{L(\mu,\nu)}{c(\nu)c(-\nu)}\int W'_{\nu'}(a^{n-2}/a_{n-1})\frac{L(\nu,\nu')}{c(\nu')c(-\nu')}d\nu'd\nu.$$
Proceeding in this way we get an iterated integral representation of $\V'$ as following.
\begin{equation}\label{multipleKLW}
    W'_\mu(a)=\kappa_na_n^{\sum\mu}\int\left(\frac{a_{n-1}}{a_n}\right)^{\sum\nu}\frac{L(\mu,\nu)}{c(\nu)c(-\nu)}\dots\int W'_\tau(a^{s+1}/a_{s+2})\frac{L(\gamma,\tau)}{c(\tau)c(-\tau)}d\tau d\gamma\dots d\nu.
\end{equation}

\subsection{Decomposition of the spherical Whittaker function}
Now we start preparing for the decomposition of the Whittaker function. We define some power series which are analogues of the $I$-Bessel function on $\GL(2)$ (see the relevant discussion in \S\ref{sketch-n-ge-2}). Let $\tau\in \C^{s+1}$ with $\Re(\tau)>0$ small enough, and $k\in \Z^s_{\ge 0}$. We define 
\begin{equation}\label{definep}
    P_{k}(\tau):=\frac{L(\tau,(\tau)^s+2k)}{c(\tau)c(-\tau)c((\tau)^s+2k)c(-(\tau)^s-2k)},
\end{equation}
and 
\begin{equation}\label{definem}
    M_\tau(a^{s+1}):=\sum_{k\in \Z^s_{\ge 0}}P_k(\tau)W'_{(\tau)^s+2k}(a^s/a_{s+1}).
\end{equation}
In the next four lemmata we prove the decomposition of $\V'$ into $M$, inductively (see Lemma \ref{partialdecomposition} for the statement we aim for). For ease of the reader we describe the themes of these technical lemmata. In Lemma \ref{startdecomposition} we prove the base case of the induction. We start with the inner most integral in \eqref{multipleKLW} and shift all the contours to infinity. The integrand will cross finitely many families of infinitely many poles. We will collect the residues and construct the $M$-Whittaker functions (more precisely the function $M^1=M$, according to the definition in \eqref{definem} and \eqref{recursivem}) as power series. Then we replace the inner integral with such an $M$-Whittaker function. We do the similar contour shifting process in the second inner most integral of \eqref{multipleKLW}, and similarly collecting residues and making power series we obtain similar $M$-Whittaker functions (more precisely $M^2$, see definition in \eqref{recursivem}). Inductively, we define certain partial $M$-Whittaker functions in \eqref{recursivem}. Finally, we prove the inductive step of the decomposition in the proof of Lemma \ref{partialdecomposition}.

We note that a similar decomposition in the case of $s=1$ appeared in \cite{H, Bu}. In the both articles the authors proved the results by the method of differential equation, while we prove it by the spectral analysis and the zeta integrals.

\begin{lemma}\label{startdecomposition}
For $\tau\in\C^{s+1}$ with $\Re(\tau)>0$ and small enough. Then,
$$\frac{1}{c(\tau)c(-\tau)}W'_\tau(a^{s+1})=\kappa_sa_{s+1}^{\sum\tau}\sum_{\sigma\in S_{s+1}}M_{\sigma\tau}(a^{s+1}),$$ such that $M_\tau$ are entire in $\tau$.
\end{lemma}
\begin{proof}
We will prove the equality for regular $\tau$ so that the result will follow by analyticity. Note that using \eqref{modifiedKLW} we write
\begin{align*}
    W'_\tau(a^{s+1})
    &=\kappa_sa_{s+1}^{\sum\tau}\int_{(0)^{s}} W'_{z}(a^{s}/a_{s+1})\frac{L(\tau,z)}{c(z)c(-z)}dz.
\end{align*}
We want to shift the contour of $z_1$ to $\infty$. To justify that we first shift $z_1$ contour to $\Re(z_1)=2N+1$ collect residues at the poles and estimate the following shifted integral
$$\int_{(0)^{s-1}} \int_{(2N+1)}W'_{z_1,z'}(a^s/a_{s+1})\frac{L(\tau,z')L(\tau,z_1)}{c(z')c(-z')\prod_{i=2}^s\Gamma_\R(z_1-z_s)\Gamma_\R(z_s-z_1)}dz_1dz'.$$
We perturb the $z'$ contour a little bit so that we can apply Lemma \ref{boundW} to obtain (dropping the entries of $W_{z_1,z'}$
$$W'_{z_1,z'}\ll_a |c(1,-z')|\prod_{i=2}^s|\Gamma_\R(1+z_1-z_s)|.$$
We note that, using Stirling's approximation
$$\frac{c(1,-z')}{c(-z')}\prod_{i=2}^s\frac{\Gamma_\R(1+z_1-z_s)}{\Gamma_\R(z_1-z_s)}\ll \prod_{i=2}^s(1+|z_i|)^{O(1)}\prod_{i=2}^s|z_1-z_i|^{O(1)}.$$
On the other hand,
$$L(\tau,z_1)\ll (N!)^{-s-1}\prod_{i=1}^{s+1}\Gamma_\R(\tau_j-1/2-\Im(z_1)).$$
Writing $z_1$ with $\Re(z_1)=2N+1$ as $z_1+2N+1$ with $\Re(z_1)=0$ we obtain, by Stirling's estimate, that the integral is bounded by
$$\ll_\tau \frac{N^{O(1)}}{(N!)^{s+1}}\int E_\tau(z)\prod_{i=2}^s(1+|z_i|)^{O(1)}|z_1-z_i|^{O(1)}\prod_{k=1}^N|z_i-z_1-k|,$$
where $$E_\tau(z):=\exp\left[-\sum_{i,j}|\Im(\tau_i-z_j)|+\sum_{i\neq j}|\Im(z_i-z_j)|\right]\ll_\tau \exp\left[-\sum_i|\Im(z_i)|\right].$$
We estimate 
$$(1+|z_i|)^{O(1)}|z_1-z_i|^{O(1)}\prod_{k=1}^N|z_i-z_1-k|\ll z_1^{O(1)}z_i^{O(1)}\sum_{k_i\le N; r}k_1\dots k_r|z_1-z_i|^{N-r},$$
Upon integrating the above against $\exp\left[-\sum_i|\Im(z_i)|\right]$ we obtain that the integral is bounded by $\frac{N^{O(1)}}{(N!)^{s+1}}\prod_{i=1}^s (N+O(1))!\ll \frac{N^{O(1)}}{N!}$, which tends to zero as $N\to \infty$.

Now we shift the $z_1$ contour to infinity. We cross (simple) poles and gather the corresponding residues to obtain the following:
\begin{align*}
    &\int_{(0)^s} W'_z(a^s/a_{s+1})\frac{L(\tau,z)}{c(z)c(-z)}dz\\
    &=\int_{(0)^{s-1}}\frac{\prod_{i=1}^{s+1}\prod_{j=2}^{s}\Gamma_\R(\tau_i-z_j)}{|c(z_2,\dots,z_s)|^2}\int_{(0)}W'_z(a^s/a_{s+1})\frac{\prod_{i=1}^{s+1}\Gamma_\R(\tau_i-z_1)}{\prod_{j=2}^s\Gamma_\R(z_1-z_j)\Gamma_\R(z_j-z_1)}dz_1dz_2\dots dz_s\\
    &=\sum_{i=1}^{s+1}\int_{(0)^{s-1}}\frac{\prod_{i=1}^{s+1}\prod_{j=2}^{s}\Gamma_\R(\tau_i-z_j)}{|c(z_2,\dots,z_s)|^2}\sum_{k_1=0}^\infty \frac{W'_{\tau_i+2k_1,z_2,\dots,z_s}(a^s/a_{s+1})L(\tau,\tau_i+2k_1)}{\prod_{j=2}^s\Gamma_\R(\tau_i-z_j+2k_1)\Gamma_\R(z_j-\tau_i-2k_1)}dz_2\dots dz_s.
\end{align*}
We record the functional equation for $\Gamma_\R$ that $\forall m\in \Z$, 
$$\Gamma_\R(s)\Gamma_\R(-s)=(-1)^{m-1}\frac{2\pi}{s}\Gamma_\R(s+2m)\Gamma_\R(2-s-2m).$$
Using this we note that
$$\prod_{j=2}^{s}\frac{\Gamma_\R(\tau_i-z_j)}{\Gamma_\R(\tau_i-z_j+2k_1)\Gamma_\R(z_j-\tau_i-2k_1)}=\prod_{j=2}^{s}\frac{(-1)^{k_1}(\tau_i-z_j+2k_1)\Gamma_\R(\tau_i-z_j)}{2\pi\Gamma_\R(\tau_i-z_j)\Gamma_\R(z_j-\tau_i+2)}.$$
Hence in the $i$'th summand above, any $z_2$ only has family of poles at $\tau_j$ for $j\neq i$. Now we shift the $z_2$ contour to $\infty$ (upon similar justification as in the case of $z_1$). Proceeding in this way we obtain that for any $s$-tuple $\tau'$ consisting of distinct elements from $\{\tau_1,\dots,\tau_{s+1}\}$
$$\int_{(0)^s} W'_z(a^s/a_{s+1})\frac{L(\tau,z)}{c(z)c(-z)}dz=\sum_{\tau'}\sum_{k\in Z^{s}_{\ge 0}}C_{k}(\tau)W'_{\tau'+2k}(a^s/a_{s+1}),$$
where 
$$C_k(\tau,\tau'):=\frac{L(\tau,\tau'+2k)}{c(\tau'+2k)c(-\tau'-2k)}.$$
Noting that, if $\sigma\in S_{s+1}$ such that $(\sigma\tau)^s=\tau'$ then
$$P_{k}(\sigma\tau)=\frac{C_k(\tau,\tau')}{c(\tau)c(-\tau)}.$$
Thus we conclude the proof of the decomposition. To check that $M_\tau$ is holomorphic it is enough to check that $P_k(\tau)$ are holomorphic ($\V'$ is entire in its parameters, see \cite{J3}) and the series defining $M_\tau$ in \eqref{definem} is locally uniformly convergent. Here we only check that $P_k(\tau)$ are holomorphic. Later in Lemma \ref{boundofM} we will estimate $P_k$ which, along with Lemma \ref{boundW}, will imply locally uniform convergence. To check holomorphicity we first note
$$c(\tau)c(-\tau)=c((\tau)^s)c(-(\tau)^s)\prod_{j=1}^s\Gamma_\R(\tau_{s+1}-\tau_j)\Gamma_\R(\tau_j-\tau_{s+1}).$$
and thus $P_{k}(\tau)$ equals to
\begin{align*}
    &\prod_{j=1}^s\frac{(-\pi)^{k_j}}{k_j!}\prod_{i>j}\frac{\Gamma_\R(\tau_i-\tau_j-2k_j)}{\Gamma_\R(\tau_j-\tau_i)\Gamma_\R(\tau_i-\tau_j)}\prod_{i<j}\frac{\Gamma_\R(\tau_i-\tau_j-2k_j)}{\Gamma_\R(\tau_j+2k_j-\tau_i-2k_i)\Gamma_\R(\tau_i+2k_i-\tau_j-2k_j)}.
\end{align*}
We check that each factor in the last expression of $P_k(\tau)$ is holomorphic, thus we conclude.
\end{proof}

Now using Lemma \ref{startdecomposition} we can expand the inner most integral in \eqref{multipleKLW} and rewrite \eqref{multipleKLW} as following.
\begin{align*}
    W'_\mu(a)
    &=\kappa_{s,n}a_n^{\sum\mu}\int \left(\frac{a_{n-1}}{a_n}\right)^{\sum\nu}\frac{L(\mu,\nu)}{c(\nu)c(-\nu)}\dots\\
    &\dots\int_{(0)^{s+1}} \left(\frac{a_{s+1}}{a_{s+2}}\right)^{\sum\tau}L(\gamma,\tau)\sum_{\sigma\in S_{s+1}}M_{\sigma\tau}(a^{s+1})d\tau d\gamma\dots d\nu.
\end{align*}     
We interchange the innermost integral with the finite sum over $S_{s+1}$. We do a change of variable $\sigma\tau\mapsto\tau$ and rewrite as following.
\begin{align*}
    W'_\mu(a)
    &=\kappa_{s,n}a_n^{\sum\mu}\int \left(\frac{a_{n-1}}{a_n}\right)^{\sum\nu}\frac{L(\mu,\nu)}{c(\nu)c(-\nu)}\dots
    (s+1)!\int_{(0)}\left(\frac{a_{s+1}}{a_{s+2}}\right)^{\tau_{s+1}}L(\gamma,\tau_{s+1})\\
    &\times\int_{(0)^{s}}\left(\frac{a_{s+1}}{a_{s+2}}\right)^{\sum(\tau)^s}L(\gamma,(\tau)^s)M_{\tau}(a^{s+1})d(\tau)^sd\tau_{s+1} d\gamma\dots d\nu.
\end{align*}
Now we shift contours of the last integral to $\infty$. The family of the poles will occur at $(\tau)^s_j=\gamma_i+2\Z_{\ge 0}$ for all $1\le j \le s$ and some $1\le i\le s+1$. The residues will be of the form
$$\left(\frac{a_{s+1}}{a_{s+2}}\right)^{\sum\gamma'+2l}L(\gamma,\gamma'+2l)M_{\gamma'+2l,\tau_{s+1}},$$
for $l\in Z^s_{\ge 0}$ and $\gamma'\in \C^s$ with $\gamma'_j\in\{\gamma_1,\dots,\gamma_{s+2}\}$. However, some of these residues do not occur in the asymptotic expansion of $\V'$, for instance, the residues with parameter $\gamma'$ with $\gamma'_1=\gamma'_2$. In fact, the terms in the asymptotic expansion come from the residues where the coordinates of $\gamma'$ are distinct\footnote{We may assume that $\gamma$ is regular.}.
But, thanks to Lemma \ref{zeroresidues} where we show that the residues other than 
$$\left(\frac{a_{s+1}}{a_{s+2}}\right)^{\sum(\sigma\gamma)^s+2l}L(\gamma,(\sigma\gamma)^s+2l)M_{(\sigma\gamma)^s+2l,\tau_{s+1}},$$
for $\sigma\in S_{s+2}$, will vanish identically. In other words, the family of poles (of form $\gamma'+2\Z^s$) of the integrand only occur at $\gamma$ which has distinct coordinates.

We follow the same method of collecting residues and construct the power series. We continue this process in \eqref{multipleKLW} until the outer most integral. We recursively define the following.
$$M^{1}_\tau(a^{s+1}):=M_\tau(a^{s+1}),$$
for $r\ge 1$, and $\alpha\in \C^{s+r}$
\begin{multline}\label{recursivem}
    M^r_{\alpha}(a^{s+r}):=\int_{(0)^{r-1}}\left(\frac{a_{s+r-1}}{a_{s+r}}\right)^{\sum z}L(\alpha,z)\\
    \sum_{l\in\Z^{s}_{\ge 0}}\left(\frac{a_{s+r-1}}{a_{s+r}}\right)^{\sum(\alpha)^s+2l}\frac{L(\alpha,(\alpha)^s+2l)}{c(\alpha)c(-\alpha)}M^{r-1}_{(\alpha)^s+2l,z}(a^{s+r-1})dz.
\end{multline}
In each stage of contour shifting we need to prove that the integrand is holomorphic at the \textit{non-regular} points, as discussed above. We show this, inductively, in Lemma \ref{zeroresidues}, which is the base case, and in Lemma \ref{finalzeroresidues}, where we prove the inductive step. These lemmata can be thought as higher rank analogues of the fact that $I_n=I_{-n}$ for any natural number $n$, where $I$ denotes the classical $I$-Bessel function. 

\begin{lemma}\label{zeroresidues}
Let $\tau\in \C^{s+1}$ such that $\Re(\tau)\ge 0$ and $\tau_a\equiv\tau_b\mod 2\Z$ for $1\le a\neq b \le s$. Then $M_\tau$ is identically zero.
\end{lemma}
\begin{proof}
Let $1\le a< b\le s$. 
Expanding out the definition
\eqref{definep} of $P_k$
gives
\begin{align*}
    P_k(\tau)
    &=\frac{\Gamma_\R(\tau_a-\tau_b-2k_b)\Gamma_\R(\tau_b-\tau_a-2k_a)\res_{s=-2k_a}\Gamma_\R(s)\res_{s=-2k_b}\Gamma_\R(s)}{\Gamma_\R(\tau_a-\tau_b)\Gamma_\R(\tau_b-\tau_a)\Gamma_\R(\tau_a-\tau_b+2k_a-2k_b)\Gamma_\R(\tau_b-\tau_a+2k_b-2k_a)}\\
    &\times \frac{\prod_{j\neq a,b}L(\tau,\tau_j+2k_j)\prod_{i\neq a,b}\Gamma_\R(\tau_i-\tau_a-2k_a)\Gamma_\R(\tau_i-\tau_b-2k_b)}{\prod_{(i,j)\neq (a,b),(b,a)}\Gamma_\R(\tau_i-\tau_j)\Gamma_\R(\tau_i-\tau_j+2k_i-2k_j)}.
\end{align*}
Computing the residues we obtain that the above equals to
\begin{align*}
    &\frac{\pi^{-4}(\tau_a-\tau_b)(\tau_b-\tau_a+2k_b-2k_a)}{\Gamma_\R(2k_a+2)\Gamma_\R(2k_b+2)\Gamma_\R(\tau_a+2k_a-\tau_b+2)\Gamma_\R(\tau_b+2k_b-\tau_a+2)}\\
    &\times Q_k(\tau)\times\frac{\prod_{j\neq a,b}L(\tau,\tau_j+2k_j)}{\prod_{(i,j)\neq (a,b),(b,a)}\Gamma_\R(\tau_i-\tau_j)},
\end{align*}
where $$Q_k(\tau):=\frac{\prod_{i\neq a,b}\Gamma_\R(\tau_i-\tau_a-2k_a)\Gamma_\R(\tau_i-\tau_b-2k_b)}{\prod_{(i,j)\neq (a,b),(b,a)}\Gamma_\R(\tau_i-\tau_j+2k_i-2k_j)}.$$
Suppose, without loss of generality, that $\tau_a-\tau_b=2l\ge 0$. Note that $\Gamma_\R(\tau_b-\tau_a+2k_b+2)^{-1}=0$ for $k_b<l$. Now changing $k_b\mapsto k_b+l$ and replacing $\tau_b+2l=\tau_a$ only in the first two factors we obtain from \eqref{definem} that
\begin{align*}
    M_\tau
    =\sum_{k\in \Z^s_{\ge 0}}&W'_{\dots,\tau_a+2k_a,\dots,\tau_a+2k_b,\dots}\times Q_k(\dots,\tau_a,\dots,\tau_a,\dots)\times \frac{\prod_{j\neq a,b}L(\tau,\tau_j+2k_j)}{\prod_{(i,j)\neq (a,b),(b,a)}\Gamma_\R(\tau_i-\tau_j)}\\
    &\times\frac{\pi^{-4}(2l)(k_b-k_a)}{\Gamma_\R(2k_a+2)\Gamma_\R(2k_b+2l+2)\Gamma_\R(2k_a+2l+2)\Gamma_\R(2k_b+2)}.
\end{align*}
Let $\sigma_{ab}$ be the element in the Weyl group which transposes the $a$'th and $b$'th elements of $\tau$ and fixes everything else. Then doing a similar calculation we can check that 
\begin{align*}
    M_{\sigma_{ab}\tau}
    &=\sum_{k\in \Z^s_{\ge 0}}W'_{\dots,\tau_a+2k_a,\dots,\tau_a+2k_b,\dots}\times Q_k(\dots,\tau_a,\dots,\tau_a,\dots)\times \frac{\prod_{j\neq a,b}L(\tau,\tau_j+2k_j)}{\prod_{(i,j)\neq (a,b),(b,a)}\Gamma_\R(\tau_i-\tau_j)}\\
    &\times\frac{\pi^{-4}(-2l)(k_b-k_a)}{\Gamma_\R(2k_a+2l+2)\Gamma_\R(2k_b+2)\Gamma_\R(2k_a+2)\Gamma_\R(2k_b+2l+2)}\\
    &=-M_\tau.
\end{align*}
On the other hand it can be easily checked from \eqref{definep} that
$$P_{\sigma_{ab}k}(\sigma_{ab}\tau)=P_k(\tau).$$
Thus using the fact that Whittaker function is invariant under the Weyl group action on its parameters we can also conclude that 
$$M_\tau=M_{\sigma_{ab}\tau},$$ 
hence the conclusion.
\end{proof}

\begin{lemma}\label{finalzeroresidues}
Let $\alpha\in\C^{s+r}$ such that $\Re(\alpha)\ge 0$ and $\alpha_a\equiv\alpha_b\mod 2\Z$ for $1\le a\neq b \le s$. Then $M^r_\alpha$ is identically zero.
\end{lemma}
\begin{proof}
We prove by inducting on $r$. Note that the base case $r=1$ is proved in Lemma \ref{zeroresidues}. We assume the claim is true for $r\ge 1$. We consider the inner most sum of $M_\alpha^{r+1}$, which is
\begin{equation}\label{inner-most-sum-in-this-lemma}
    \sum_{l\in\Z^{s}_{\ge 0}}\left(\frac{a_{s+r}}{a_{s+r+1}}\right)^{\sum(\alpha)^s+2l}\frac{L(\alpha,(\alpha)^s+2l)}{c(\alpha)c(-\alpha)}M^{r}_{(\alpha)^s+2l,z}(a^{s+r-1}).
\end{equation}
We take a similar proof path as in Lemma \ref{zeroresidues}. Suppose that $\sigma_{ab}$ is the element in the Weyl group which only transposes $a$'th and $b$'th elements of $\alpha$. Clearly, $M^{r+1}_{\sigma_{ab}\alpha}=M^{r+1}_{\alpha}$, as before. We will show that when $\alpha_a-\alpha_b\in 2\Z$ then $M^{r+1}_{\sigma_{ab}\alpha}=-M^{r+1}_{\alpha}$ which will yield the claim.
We write the coefficient
\begin{align*}
    &\frac{L(\alpha,(\alpha)^s+2l)}{c(\alpha)c(-\alpha)}\\
    &=\frac{4\pi^{-2}\Gamma_\R(\alpha_a-\alpha_b-2l_b)\Gamma_\R(\alpha_b-\alpha_a-2l_a)(-1)^{l_a+l_b}}{\Gamma_\R(\alpha_a-\alpha_b)\Gamma_\R(\alpha_b-\alpha_a)\Gamma_\R(2l_a+2)\Gamma_\R(2l_b+2)} \frac{\prod_{i,j\notin \{a,b\}}L(\alpha_i,\alpha_j+2l_j)}{\prod_{(e\neq f)\neq(a,b),(b,a)}\Gamma_\R(\alpha_e-\alpha_f)}\\
    &\times\prod_{i\neq a,b}\Gamma_\R(\alpha_i-\alpha_a-2l_a)\Gamma_\R(\alpha_i-\alpha_a-2l_b)\prod_{j\neq a,b}\Gamma_\R(\alpha_a-\alpha_j-2l_j)\Gamma_\R(\alpha_b-\alpha_j-2l_j).
\end{align*}
Here $i$ and $j$ are varying over $1,\dots, r+s$ and $1,\dots, s$, respectively. Using the functional equation for $\Gamma_\R$ we obtain that the above is
\begin{align*}
    &=\frac{2\pi^{-3}(\alpha_b-\alpha_a)\Gamma_\R(\alpha_b-\alpha_a-2l_a)(-1)^{l_a}}{\Gamma_\R(\alpha_b-\alpha_a+2l_b+2)\Gamma_\R(2l_a+2)\Gamma_\R(2l_b+2)} \frac{\prod_{i,j\notin \{a,b\}}L(\alpha_i,\alpha_j+2l_j)}{\prod_{(e\neq f)\neq(a,b),(b,a)}\Gamma_\R(\alpha_e-\alpha_f)}\\
    &\times\prod_{i\neq a,b}\Gamma_\R(\alpha_i-\alpha_a-2l_a)\Gamma_\R(\alpha_i-\alpha_a-2l_b)\prod_{j\neq a,b}\Gamma_\R(\alpha_a-\alpha_j-2l_j)\Gamma_\R(\alpha_b-\alpha_j-2l_j).
\end{align*}
Without loss of generality, suppose that $a<b$ and $\alpha_a-\alpha_b=2t\ge 0$, and $t\in \Z$. As $M^r_\alpha=0$ by the induction hypothesis we note by the above computation that the summand vanishes for $l_b< t$. In the sum \eqref{inner-most-sum-in-this-lemma} we change variable $l_b\mapsto l_b+t$ and obtain that \eqref{inner-most-sum-in-this-lemma} equals to
\begin{align*}
    &\sum_{l\in\Z^{s}_{\ge 0}}\left(\frac{a_{s+r}}{a_{s+r+1}}\right)^{2t+\sum(\alpha)^s+2l} \times Q_l(\alpha)\times \frac{2\pi^{-3}(\alpha_b-\alpha_a)(-1)^{l_a}}{\Gamma_\R(2l_b+2)\Gamma_\R(2l_a+2)\Gamma_\R(2l_b+2t+2)}\\
    &\times\lim_{\alpha_a-\alpha_b\to 2t}\Gamma_\R(\alpha_b-\alpha_a-2l_a)M^{r}_{\dots,\alpha_a+2l_a,\dots,\alpha_b+2l_b+2t,\dots,z}(a^{s+r-1}),
\end{align*}
where 
\begin{align*}
    Q_l(\alpha):&=\frac{\prod_{i,j\notin \{a,b\}}L(\alpha_i,\alpha_j+2l_j)}{\prod_{(e\neq f)\neq(a,b),(b,a)}\Gamma_\R(\alpha_e-\alpha_f)}\\
    &\times\prod_{j\neq a,b}\Gamma_\R(\alpha_a-\alpha_j-2l_j)\Gamma_\R(\alpha_b-\alpha_j-2l_j)\prod_{i\neq a,b}\Gamma_\R(\alpha_i-\alpha_a-2l_a)\Gamma_\R(\alpha_i-\alpha_a-2l_b).
\end{align*}
We compute the last limit. Let $\beta:=\alpha_b-\alpha_a+2t$. Then the last limit equals to
\begin{align*}
    &\res_{s=-2t-2l_a}\Gamma_\R(s)\lim_{\beta\to 0}\beta^{-1}{M^{r}_{\dots,\alpha_a+2l_a,\dots,\alpha_a+\beta+2l_b,\dots,z}(a^{s+r-1})}\\
    &=\frac{2(-1)^{l_a+t}}{\pi\Gamma_\R(2l_a+2t+2)}\lim_{\beta\to 0}\beta^{-1}{M^{r}_{\dots,\alpha_a+2l_a,\dots,\alpha_a+\beta+2l_b,\dots,z}(a^{s+r-1})}.
\end{align*}
Thus we obtain that \eqref{inner-most-sum-in-this-lemma} equals to
\begin{align*}
    &\sum_{l\in\Z^{s}_{\ge 0}}\left(\frac{a_{s+r}}{a_{s+r+1}}\right)^{2t+\sum(\alpha)^s+2l} \times \frac{4\pi^{-4}(-1)^{t}Q_l(\alpha)}{\Gamma_\R(2l_b+2)\Gamma_\R(2l_a+2)\Gamma_\R(2l_b+2t+2)\Gamma_\R(2l_a+2t+2)}\\
    &\times (\alpha_b-\alpha_a)\lim_{\beta\to 0}\beta^{-1}M^{r}_{\dots,\alpha_a+2l_a,\dots,\alpha_a+2l_b+\beta,\dots,z}(a^{s+r-1}).
\end{align*}
Then doing a similar computation we obtain that the relevant sum as in \eqref{inner-most-sum-in-this-lemma} in the corresponding expression of $M_{\sigma_{ab}\alpha}$ equals to
\begin{align*}
    &\sum_{l\in\Z^{s}_{\ge 0}}\left(\frac{a_{s+r}}{a_{s+r+1}}\right)^{2t+\sum(\alpha)^s+2l} \times \frac{4\pi^{-4}(-1)^{t}Q_l(\alpha)}{\Gamma_\R(2l_b+2)\Gamma_\R(2l_a+2)\Gamma_\R(2l_b+2t+2)\Gamma_\R(2l_a+2t+2)}\\
    &\times(\alpha_a-\alpha_b)\lim_{\beta\to 0}\beta^{-1}M^{r}_{\dots,\alpha_a+2l_a-\beta,\dots,\alpha_a+2l_b,\dots,z}(a^{s+r-1}).
\end{align*}
Thus the proof will be complete if we can show that
$$\lim_{\beta\to 0}\beta^{-1}M^{r}_{\dots,\alpha_a+2l_a-\beta,\dots,\alpha_a+2l_b,\dots,z}(a^{s+r-1})=\lim_{\beta\to 0}\beta^{-1}M^{r}_{\dots,\alpha_a+2l_a,\dots,\alpha_a+2l_b+\beta,\dots,z}(a^{s+r-1}).$$
Note that both limits exist by the induction hypothesis ($M^r$ has a zero at a non-regular point).
Finally, the above equality follows from twisting the quantity in the limit by $|\det|^\beta$ and letting $\beta\to 0$.
\end{proof}

Finally, we prove the inductive step of the decomposition of $\V'$ in to $M$-Whittaker function, whose base case is proved in Lemma \ref{startdecomposition}.
\begin{lemma}\label{partialdecomposition}
Let $\mu\in E_n(\epsilon)$ for some $\epsilon>0$. Then there exists an absolute constant $\kappa_{s,n}$ such that,
$$\frac{1}{c(\mu)c(-\mu)}W'_\mu(a)=\kappa_{s,n}a_n^{\sum\mu}\sum_{\sigma\in S_n}M^{n-s}_{\sigma\mu}(a),$$
and $M_\mu$ is holomorphic in $\Re(\mu)\ge 0$.
\end{lemma}
\begin{proof}
From the definition \eqref{recursivem} and estimates in Lemma \ref{boundofM} holomorphicity of $M_\mu$ will follow. To prove the the equality in the statement we will induct on $r$. The base case $r=1$ is proved in Lemma \ref{startdecomposition}. At $r$'th intermediate stage the expression of $W'_\mu(a)$ looks like 
\begin{align*}
    &\frac{1}{c(\mu)c(-\mu)} W'_\mu(a)\\
    &=\kappa_{s,r,n}a_n^{\sum\mu}\int_{}\left(\frac{a_{n-1}}{a_n}\right)^{\sum\nu}\frac{L(\mu,\nu)}{c(\mu)c(-\mu)}\dots
    \int_{(0)^{s+r}}\left(\frac{a_{s+r}}{a_{s+r+1}}\right)^{\sum\theta}\frac{L(\eta,\theta)}{c(\eta)c(-\eta)}M_\theta^r(a^{s+r})d\theta\dots d\nu\\
    &=\kappa_{s,r,n}a_n^{\sum\mu}\int_{}\left(\frac{a_{n-1}}{a_n}\right)^{\sum\nu}\frac{L(\mu,\nu)}{c(\mu)c(-\mu)}\dots
    \int_{(0)^{r}}\left(\frac{a_{s+r}}{a_{s+r+1}}\right)^{\sum\theta'}{L(\eta,\theta')}\\
    &\times \int_{(0)^{s}}\left(\frac{a_{s+r}}{a_{s+r+1}}\right)^{\sum(\theta)^s}\frac{L(\eta,(\theta)^s)}{c(\eta)c(-\eta)}M_\theta^r(a^{s+r})d(\theta)^s d\theta'\dots d\nu,
\end{align*}
where $\theta':=(\theta_{s+1},\dots,\theta_{s+r})$. Now we shift contours to infinity in the last integral. Thus, employing Lemma \ref{finalzeroresidues}, collecting residues we obtain for some constant $d'=d_{r,s}$ that
\begin{align*}
    &\frac{1}{c(\mu)c(-\mu)} W'_\mu(a)\\
    &=\kappa_{s,r,n}a_n^{\sum\mu}\int_{}\left(\frac{a_{n-1}}{a_n}\right)^{\sum\nu}\frac{L(\mu,\nu)}{c(\mu)c(-\mu)}\dots
    \int_{(0)^{r}}\left(\frac{a_{s+r}}{a_{s+r+1}}\right)^{\sum\theta'}{L(\eta,\theta')}\\
    &\sum_{\sigma\in S_{s+r}}\sum_{l\in Z^s_{\ge 0}}\left(\frac{a_{s+r}}{a_{s+r+1}}\right)^{\sum(\sigma\eta)^s+2l}\frac{L(\eta,(\sigma\eta)^s+2l)}{c(\eta)c(-\eta)}M_{(\sigma\eta)^s+2l,\theta'}^r(a^{s+r})d\theta'\dots d\nu.
\end{align*}
Noting the symmetries of the variables inside the integrals and recalling \eqref{recursivem} we can conclude.
\end{proof}

Now we will estimate the function $M^{n-s}$ inductively. The following lemma will be used to prove the required bound in Proposition \ref{heart} for those $a\in A$ which are in the complementary case of what we considered in Lemma \ref{notpop}, i.e. $a$ satisfies $\pop(s)$ (see definition in \eqref{POP}). We loosely mention that $M^{n-s}$-Whittaker functions will be exponentially increasing in $a_i/a_j$ for $1\le i\le s$ and $s+1\le j\le n$ (see \S\ref{sketch-n-ge-2}). This is exactly where we will be using $\pop(s)$ to control the increment. 
\begin{lemma}\label{boundofM}
Let $1\le s\le n$ and $a$ satisfies $\pop(s)$. Also let, $\mu\in\C^n$ such that $\Re(\mu)>0$ small enough. We define $\mu':=(\mu_1+2N,\dots,\mu_s+2N,\mu_{s+1},\dots,\mu_n)$ for some fixed large integer $N$. Then
$$M_{\mu'}^{(n-s)}(a)\prec_{N,\epsilon}\frac{d(\mu)^{O(1)}}{c(\Im(\mu))}\frac{a_1^{2N}}{a_n^{2sN}}.$$
Here $O(1)$ in the exponent of $d(\mu)$ denotes a bounded constant depending on $N$ and $n$.
\end{lemma}
\begin{proof}
We prove by inducting on $r$ using the inductive definition in \eqref{recursivem}. Note that it is enough to prove the required bound of $M^r_\alpha$ for $\alpha:=(\mu_1+2N_1,\dots,\mu_s+2N_s,\alpha')$ where $N_i\ge N$ and $\alpha'\in \C^{r}$ with small positive real parts. For $\alpha\in \C^r$, by $\iota({\alpha})$ we will denote the reordering of the coordinates of $\alpha$ such that $\Re(\iota({\alpha})_1))\le\dots\le \Re(\iota({\alpha})_r)$. We will frequently use Stirling approximation, and also for $s\in\C$ having small real part and $k\ge 0$
$$|\Gamma(s-k)|\ll |\Gamma(s)|.$$

Let us first prove the claimed bound for $M^1$. From the definition \eqref{definep}, for $\tau\in \C^{s+1}$ with $\Re(\tau_j)\ge 2N$ for $1\le j\le s$, we estimate 
\begin{align*}
    &c(1,-\iota({(\tau)^s+2k}))P_{k}(\tau):
    =\frac{L(\tau,(\tau)^s+2k)c(1,-\iota({(\tau)^s+2k}))}{c(\tau)c(-\tau)c((\tau)^s+2k)c(-(\tau)^s-2k)}\\
    &\asymp \prod_{i=1}^s\frac{\pi^{k_i}}{k_i!}\frac{\Gamma_\R(\tau_{s+1}-\tau_i-2k_i)}{\Gamma_\R(\tau_{s+1}-\tau_i)\Gamma_\R(\tau_i-\tau_{s+1})}\frac{\prod_{1\le i\neq j\le s}\Gamma_\R(\tau_i-\tau_j-2k_j)}{\prod_{i\neq j}\Gamma_\R(\tau_j-\tau_i)c(\iota({(\tau)^s+2k}))}\frac{c(1,-\iota({(\tau)^s+2k}))}{c(-\iota({(\tau)^s+2k}))}.
\end{align*}
The last factor in the above quantity can be bounded by $(1+\|k\|)^{O(1)}d((\tau)^s)^{O(1)}$. The second factor is $\ll |\Gamma_\R(\tau_i-\tau_{s+1})|^{-1}$. In the third factor, suppose that $\Gamma_\R(\tau_p+2k_p-\tau_q-2k_q)$ appears in $c(\iota({(\tau)^s+2k}))$, for some $1\le p\neq q\le s$. Then we note that
$$\frac{\Gamma_\R(\tau_p-\tau_q-2k_q)\Gamma_\R(\tau_q-\tau_p-2k_p)}{\Gamma_\R(\tau_p-\tau_q)\Gamma_\R(\tau_q-\tau_p)\Gamma_\R(\tau_p+2k_p-\tau_q-2k_q)}\ll|\Gamma_\R(\tau_q-\tau_p)|^{-1}.$$
We can do similar estimate for each sub-factor in the third factor. If we define 
$$\tilde{c}(\tau):=\prod_{i=1}^{s}\Gamma_\R(\tau_i-\tau_{s+1})\prod_{1\le i\neq j\le s}\min\{|\Gamma_\R(\tau_i-\tau_j)|,|\Gamma_\R(\tau_j-\tau_i)|\},$$
then we have obtained that
$$c(1,-\iota({(\tau)^s+2k}))P_k(\tau)\ll |\tilde{c}(\tau)|^{-1}|\prod_{i=1}^s\frac{\pi^{k_i}}{k_i!}(1+\|k\|)^{O(1)}d(\tau)^{O(1)}.$$
Finally using \eqref{definem} and Lemma \ref{boundW} we estimate
\begin{align*}
    M_\tau(a^{s+1})
    &=\sum_{k\in \Z^s_{\ge 0}}P_k(\tau)W'_{(\tau)^s+2k}(a^s/a_{s+1})\\
    &\prec \frac{d(\tau)^{O(1)}}{\tilde{c}(\tau)}\frac{a_1^{2N}}{a_{s+1}^{\Re\sum(\tau)^s}}\sum_{k\in \Z^s_{\ge 0}}\frac{\pi^{k_i}}{k_i!}(1+\|k\|)^{O(1)}\\
    &\ll \frac{d(\tau)^{O(1)}}{\tilde{c}(\tau)}\frac{a_1^{2N}}{a_{s+1}^{\Re\sum(\tau)^s}},
\end{align*}
where in the first inequality we have employed $\pop(s)$ assumption.

Now we make an inductive hypothesis that for $r\ge 1$
$$M^{r}_\alpha(a^{s+r})\prec \frac{d(\alpha)^{O(1)}}{\tilde{c}(\alpha)}\frac{a_1^{2N}}{a_{r+s}^{\Re\sum(\alpha)^s}},$$
where motivated by the previous computation we define
$$\tilde{c}(\alpha):=\prod_{i=1}^{s}\prod_{j=1}^r\Gamma_\R(\alpha_i-\alpha_{s+j})\prod_{1\le i<j\le r}\Gamma_\R(\alpha_{s+i}-\alpha_{s+j})\prod_{1\le i\neq j\le s}\min\{|\Gamma_\R(\alpha_i-\alpha_j)|,|\Gamma_\R(\alpha_j-\alpha_i)|\}.$$
We start with \eqref{recursivem} and integrate term by term. We claim that
\begin{align*}
&\frac{L(\alpha,(\alpha)^s+2l)}{c(\alpha)c(-\alpha)}\int_{z}\left(\frac{a_{s+r}}{a_{s+r+1}}\right)^{\sum z}L(\alpha,z)\frac{d((\alpha)^s+2l)^{O(1)}}{\tilde{c}((\alpha)^s+2l,z)}d(z)^{O(1)}\\
&\prec \frac{a_1^{2N}}{a_{s+r}^{\Re\sum(\alpha)^s+2l}}|\tilde{c}(\alpha)|^{-1}\prod_{i=1}^s\frac{\pi^{l_i}}{l_i!}(1+\|l\|)^{O(1)}d((\alpha)^s)^{O(1)}d(\alpha_{s+1},\dots\alpha_{s+r+1})^{O(1)}.
\end{align*}
Employing the inductive hypothesis we can yield the proof as in the base case. 

To prove the claim we note that,
\begin{align*}
    &\frac{L(\alpha,(\alpha)^s+2l)L(\alpha,z)}{c(\alpha)c(-\alpha)\tilde{c}((\alpha)^s+2l,z)}\prod_{i=1}^{s}\prod_{j=1}^{r+1}\Gamma_\R(\alpha_i-\alpha_{s+j})\prod_{1\le i<j\le r+1}\Gamma_\R(\alpha_{s+i}-\alpha_{s+j})\\
    &\asymp\prod_{i=1}^s\frac{\pi^l_i}{l_i!}\frac{\prod_{i=1}^{r+1}\prod_{j=1}^r{\Gamma_\R(\alpha_{s+i}-z_j)}}{c(z)\prod_{1\le i<j \le r+1}\Gamma_\R(\alpha_{s+j}-\alpha_{s+i})}\prod_{i=1}^s\prod_{j=1}^{r}\frac{\Gamma_\R(\alpha_i-z_j)}{\Gamma_\R(\alpha_i+2l_i-z_j)}\\
    &\hphantom{rubbisrubbish}\times\frac{\prod_{1\le i\neq j\le s}\Gamma_\R(\alpha_i-\alpha_j-2l_j)}{c((\alpha)^s)\tilde{c}((\alpha)^s+2l)}\prod_{i=1}^{r+1}\prod_{j=1}^s\frac{\Gamma_\R(\alpha_{s+i}-\alpha_j-2l_j)}{\Gamma_\R(\alpha_{s+i}-\alpha_j)}.
\end{align*}
Note that, as in the $r=1$ case, the fifth and third factors of the last quantity are $\ll 1$, and the fourth factor is $\ll \tilde{c}((\alpha)^s)$. Thus it will be enough to prove that the integral 
$$\int_{(0)^r}\left|\frac{\prod_{i=1}^{r+1}\prod_{j=1}^r{\Gamma_\R(\alpha_{s+i}-z_j)}}{c(z)\prod_{1\le i<j \le r+1}\Gamma_\R(\alpha_{s+j}-\alpha_{s+i})}\right|d(z)^{O(1)}dz\ll d(\alpha_{s+1},\dots\alpha_{s+r+1})^{O(1)},$$
as this, along with
$$d((\alpha)^s+2l)^{O(1)}\ll d((\alpha)^s)^{O(1)}(1+\|l\|)^{O(1)},$$
proves the claim. To see the estimate of the last integral we follow the same path as in \cite[Proposition $1$]{B1}. We may assume that $0<\Re(\alpha_{s+j})<\epsilon$. We write $\alpha_{s+j}=\Re(\alpha_{s+j})+i\beta_j$ and $z_j=it_j$. We use Stirling approximation to obtain that the integral is bounded by
$$d(\beta)^{O(1)}\int_{\R^r}{\prod_{i,j}(1+|\beta_i-t_j|)^{-1/2+O(\epsilon)}}\prod_{i\neq j}|t_i-t_j|^{1/2}d(t)^{O(1)}E(t,\beta)dt,$$
where $E$ is the exponential factor given by
$$E(t,\beta):=\exp\left[-\frac{\pi}{4}\left(\sum_{i,j}|\beta_i-t_j|-\sum_{i< j}|t_i-t_j|-\sum_{i<
j}|\beta_i-\beta_j|\right)\right].$$
By fixing an order among $\beta_i$ it can be checked elementarily that (as in the proof of  \cite[Proposition $1$]{B1}) the quantity inside $\exp$ is always non-positive. Thus the essential supports of $t$ in this integral are bounded by polynomials in $\beta$. The integrand, other than the exponential factor, also being a polynomial in $t$ and $\beta$, the integral is bounded by some polynomial in $\beta$, thus is $d(\beta)^{O(1)}$.

Finally, we conclude the proof by noting that 
$$\tilde{c}(\mu)\gg c(\Im(\mu))d(\mu)^{O(1)},$$
where $O(1)$ in the exponent means a fixed non-negative exponent depending on $N$.
\end{proof}

We finally have all the ingredients we need to prove Proposition \ref{heart}. We will prove this by dividing the argument into two cases: whether $a$ has the property $\pop(s)$ for some $1\le s\le n$, or not.
\begin{proof}[Proof of Proposition \ref{heart}]
Main ingredient of the proof is to shift contours in the integral of \eqref{finalreducespherical}. We will divide the proof into two cases.

\textbf{Case I:} We assume that $a$ does not satisfy $\pop(s)$ for any $1\le s\le n$. Let $d_0(\mu)$ be the eigenvalue of $W_\mu$ under $\D$ (recall Definition \eqref{defdiffop}). We have $d_0(it)\ll d(it)$ which follows from \eqref{defn-d-mu} and Lemma \ref{diffop}. We see that, using \eqref{boundG} for $M=0$, Lemma \ref{boundinnerproduct} for $R=0$, and Lemma \ref{notpop} the RHS of \eqref{finalreducespherical} is bounded by
$$\D^l\We\left[\begin{pmatrix}C(\Pi)&\\&a\end{pmatrix}w\right]\prec_{N,p} \delta^{1/2}(a)\min(1,a_1^N)\int_{\R^n}{d(it)^{l-p}}\frac{|c(1,it)|^2}{|c(it)|^2}dt.$$
Using Stirling, $\frac{c(1,it)}{c(it)}\ll d(it)^{l'}$ for some absolute $l'$. Thus the integral in the RHS above is convergent if $p$ is sufficiently large. Hence proof of this case concludes.

\textbf{Case II:} We assume that $a$ satisfies $\pop(s)$ for some given $s$. We use Lemma \ref{partialdecomposition} in the RHS of \eqref{finalreducespherical}. We exchange the finite sum with integral over $\mu$ and change variable $w\mu\mapsto \mu$. We obtain that for some explicit constant $\kappa_n$
$$\D^l\We\left[\begin{pmatrix}C(\Pi)&\\&a\end{pmatrix}w\right]
=\kappa_n\delta^{1/2}(a)\int_{(0)^n}a_n^{\sum\mu}\Theta(\mu,\Pi)d_0(\mu)\langle f,W_\mu\rangle M^{n-s}_\mu(a){d\mu}.$$
Now we shift contour of $\mu_i\mapsto\mu_i+2N$ for $1\le i\le s$ for some natural number $N$. Note that the integrand does not cross any pole. Employing Lemma \ref{boundofM}, bound of $\Theta$ in \eqref{boundG}, and Lemma \ref{boundinnerproduct} with $R=nN+1$ we conclude that
$$\D^l\We\left[\begin{pmatrix}C(\Pi)&\\&a\end{pmatrix}w\right]\prec_{N,p} \delta^{1/2}(a)a_1^{2N}\int_{\R^n}\frac{d(it)^{O(1)}}{d(it)^p}\frac{|c(1,it)|}{|c(it)|}dt,$$
where $O(1)$ in the exponent of $d(it)$ depends at most on $l,N$. We argue as in the previous case and thus conclude.
\end{proof}

\subsection{Remarks on the sphericality assumption of the chosen newvector}\label{sphericality-unnecessary}
Here we take the opportunity to say a few words about the choice of the newvector in \eqref{newvec}. It is mostly motivated by the construction of the newvectors in \cite{JPSS1, M}. We note that in the non-archimedean case, the newvectors are spherical in the Kirillov model,
i.e., $\GL_{n}(\mathbb{Z}_p)$-invariant.
We analogously choose
our analytic newvectors
to be $\O(n)$-invariant,
but this feature of our construction
is not essential. The main purpose of this assumption is to make the presentation and proof of Proposition \ref{heart} a little simpler. In fact, any $f\in C^\infty_c(N\backslash G,\psi)$ would serve the purpose as \eqref{newvec} does. We briefly describe about the essential modifications needed to carry out
in the case when $f$ is not spherical.

One only needs to modify the proof of Proposition \ref{heart}, because the sphericality of $f$ has been used only in this proof. If we do not choose $f$ in \eqref{newvec} to be spherical then for the Whittaker--Plancherel expansion, we must use \eqref{genwhitplan} instead of \eqref{spherwhitplan}, hence \eqref{finalreducespherical} changes to (say, for $l=0$) 
\begin{equation*}
\int_{\hat{G}}\langle \lambda(g)f,j_\pi\rangle d\mu_p(\pi),
\end{equation*}
where $\lambda(g)f(h):=f(hg)$. Here $j_\pi$ is the relative character of $\pi$, also known as the (long Weyl) Bessel distribution attached to $\pi$, which is defined as a distribution on $G$ by
$$j_\pi(g):=\sum_{W\in \B(\pi)}W(g)\overline{W(1)}.$$
Here $\B$ is an orthonormal basis of
the Whittaker model of $\pi$. A reference on the Bessel distribution can be found in \cite{CPS, J2, BM}.
Recalling \eqref{relative-trace} we have that
$$J_\pi(f)=\int_{G} f(g) j_\pi(g) dg.$$
It can be proved, using spectral analysis, that $j_\pi$ satisfies the following recursion (compare with \eqref{modifiedKLW}):
\begin{equation}\label{recursion-j-bessel}
    j_\pi\left[\begin{pmatrix}1&\\&g\end{pmatrix}w\right]=\int_{\widehat{\GL}_{n-1}(\R)}\gamma(1/2,\pi\otimes \bar{\sigma})\omega_\sigma(-1)^{n-1}j_\sigma(gw')d\mu_p(\sigma),
\end{equation}
where $\omega_\sigma$ is the central character of $\sigma$ and $j_\sigma$ is the Bessel distribution attached to $\sigma$. A decomposition of $j_\pi$, analogous to the decomposition of the spherical Whittaker function, can be obtained using \eqref{recursion-j-bessel} (for $\GL(2)$ see, for instance, \cite[chapter $6$]{CPS}).

We should also point out that the Plancherel density in this case is not holomorphic, unlike in the spherical case. This is because the Plancherel density is a ratio of local $L$-factors which have poles. In the spherical case the numerator of the density gets cancelled with the $L^2$-norm of the spherical Whittaker function. So in the general case, while we shift contour we will likely to cross some polar hyperplanes coming from the Plancherel density. However, the residues will cancel with some part of the integral over $\pi$. For instance, in $\GL(2)$ the Plancherel integral over $\widehat{\GL(2)}$ can be decomposed as a sum of contour integrals over principal series and sum over discrete series. It can be checked that such residues from the integrals over the principal series will cancel out some summands in the sum over discrete series. This analysis will be detailed in an upcoming work by the first named author.

\section{Multiplicity One}\label{sec:proof-multiplicity-one}
Let $\We_\phi\in \Pi$ be a vector such that $\We_\phi$ in the Kirillov model is given by 
$$\We_\phi\left[\begin{pmatrix}g&\\&1\end{pmatrix}\right]:=\int_N \phi(ng)\overline{\psi(n)}dn,$$
where $\phi\in C_c^\infty(G)$, is sufficiently concentrated around the identity and $L^1$ normalized.

Recall the conventions on $\tau$ as in the beginning of \S\ref{sec:reduction}. Let $f^0_X$ be a
normalized majorant of $K_0(X,\tau)$. Let $f_1\in C_c^\infty(G)$ be a fixed $L^1$ normalized non-negative function supported around the identity. We also define
\begin{equation}\label{char-func-new-sub}
    f_X(g):=\int_G f^0_X\left[g\begin{pmatrix}u&\\&1\end{pmatrix}\right]f_1(u)du.
\end{equation}
Working as in Lemma \ref{convolution-normalized-majorant} implies 
we can see that
that $f_X$ is again a normalized majorant
of $K_0(X,\tau)$. Finally, for $h\in A\times K$ we define
\begin{equation*}
    f_X^h(g):=f_X\left[\begin{pmatrix}h^{-1}\\&1\end{pmatrix}g\begin{pmatrix}h\\&1\end{pmatrix}\right]
\end{equation*}
Here $A\times K$ is realized as a fundamental domain of $N\backslash G$ and we equip $A\times K$ with $dh$ which is induced from a $G$-invariant measure on $N\backslash G$.

\begin{prop}\label{estimating-projection-shrinking-support}
Let $U$ be a sufficiently small neighborhood of the identity element of $G$. Let $\phi \in C_c^\infty(U)$ with $\|\phi\|_{L^1(G)} \leq 1$. Let $\Pi$ be a generic irreducible unitary $\theta$-tempered representation of $\GL_{n+1}(\R)$. Then
$$\int_{A\times K}\Pi(f^h_X)\We_\phi(1)dh\ll_{\tau,\sigma} (C(\Pi)/X)^{-n\sigma},$$
for any $0<\sigma<1/2-\theta$ where the implied constant is independent of $\phi$.
\end{prop}

\begin{proof}
Let us abbreviate
$$\We_h(g):=\int_{G}f_1(u)\We_\phi\left[g\begin{pmatrix}u^{-1}h^{-1}&\\&1\end{pmatrix}\right]du.$$
Then we can write
$$\Pi(f^h_X)\We_\phi(1)=\int_G f_X^0(g)\We_h\left[\begin{pmatrix}h&\\&1\end{pmatrix}g\right]dg.$$
We use the Iwahori coordinates
$$g=z\begin{pmatrix}I&b\\&1\end{pmatrix}\begin{pmatrix}u&\\&1\end{pmatrix}\begin{pmatrix}I&\\c&1\end{pmatrix}, b,c^t\in\R^n,z\in\R^\times,u\in G,\quad dg=d^\times z\frac{du}{|\det(u)|}dbdc,$$
and do a change of variable in the above integral representation of $\Pi(f_X^h)\We_\phi(1)$ as
$$X^{-n}\int_{\R^{n}}\int_{\R^{n}}\int_{ G}\int_{\R^\times}f^0_X\left[z\begin{pmatrix}I&b\\&1\end{pmatrix}\begin{pmatrix}u&\\c/X&1\end{pmatrix}\right]\omega_\Pi(z) e(e_n^thb)\We_h\left[\begin{pmatrix}hu&\\c/X&1\end{pmatrix}\right]d^\times z\frac{du}{|\det(u)|}dbdc.$$
We use Whittaker--Plancherel expansion \eqref{genwhitplan} and the local functional equation \eqref{lfe}, as before, to write
\begin{multline*}
    \We_h\left[\begin{pmatrix}hu&\\c/X&1\end{pmatrix}\right]=\int_{\hat{G}}\gamma(1/2-\sigma,\Pi\otimes\bar{\pi})^{-1}C(\Pi)^{n\sigma}\omega_{\bar{\pi}}((-1)^{n}C(\Pi))\sum_{W\in\B(\pi)}W(hu)|\det(hu)|^{\sigma}\\\int_{N\backslash G}e\left(\frac{C(\Pi)}{X}cw't^{-1}e_1\right)\We_h\left[\begin{pmatrix}C(\Pi)&\\&t\end{pmatrix}w\right]\overline{W(tw')}|\det(t)|^{-\sigma}dt d\mu_p(\pi),
\end{multline*}
for some small enough $\sigma$. Thus can be written as (using absolute convergence to exchange integrals and summation)
\begin{multline}\label{main-equation}
    \Pi(f_X^h)\We_\phi(1)=\int_{\hat{G}}\gamma^*(\Pi,\pi)\sum_{W\in\B(\pi)}\int_{N\backslash G} \tilde{V}_{\Pi,W,h}(t)
    \int_{\R^{n}}\int_{\R^n}\int_{G}\int_{\R^\times}{f}^0\left[z\begin{pmatrix}I&b\\&1\end{pmatrix}\begin{pmatrix}u&\\c&1\end{pmatrix}\right]\\\omega_\Pi(z)e(e_n^thb)
    W(hu)|\det(hu)|^{\sigma}e\left(\frac{C(\pi)}{X}cw't^{-1}e_1\right)d^\times z\frac{du}{|\det(u)|}dbdcdtd\mu_p(\pi).
\end{multline}
Here
\begin{equation}\label{bound-gamma-star}
    \gamma^*(\Pi,\pi):=\gamma(1/2-\sigma,\Pi\otimes\bar{\pi})^{-1}C(\Pi)^{n\sigma}\omega_{\bar{\pi}}((-1)^{n}C(\Pi))\ll C(\pi')^{O(1)},
\end{equation}
which follows from Lemma \ref{boundgammafactor}, and 
$$\tilde{V}_{\Pi,W,h}(t):=\We_h\left[\begin{pmatrix}C(\Pi)&\\&tw'\end{pmatrix}w\right]\overline{W(t)}|\det(t)|^{-\sigma}.$$
Another application of Whittaker--Plancherel expansion \eqref{genwhitplan} and the local functional equation \eqref{lfe} allows us to write
\begin{multline*}
\We_h\left[\begin{pmatrix}C(\Pi)&\\&tw'\end{pmatrix}w\right]=\int_{\hat{G}}\omega_{\tilde{\pi}}((-1)^nC(\Pi)^{-1})\gamma(1/2,\Pi\otimes\tilde{\pi})\\\sum_{W'\in\B(\pi)}W'(t)\int_{N\backslash G}\We_h\left[\begin{pmatrix}t'&\\&1\end{pmatrix}\right]\overline{W'(t')}dt'.
\end{multline*}
The inner integral above can also be written as
$$\int_{G}\int_G f_1(t_1)\phi(t_2)W'(t_2ht_1)dt_1dt_2,$$
which follows from the definition of $\We_1$. We integrate by parts many times in the $t_1$ integral with respect to $\D$, then use Lemma \ref{whittaker-bound}
$$W'(t_2ht_1)\ll \delta^{1/2-\eta}(t_2ht_1) S_{d_0}(W'),$$
and finally, use that $\phi$ is $L^1$ normalized to deduce that the above integral is bounded by $S_{-d}(W')\delta^{1/2-\eta}(h)$ for any $d>0$ as $t_1\asymp t_2\asymp 1$. Again appealing to Lemma \ref{whittaker-bound} we get that
$$W'(t)\ll \delta^{1/2-\eta}(t)S_{d_0}(W'),$$
along with Lemma \ref{boundgammafactor} and Lemma \ref{sobolev-norm-property} yield that
$$\We_h\left[\begin{pmatrix}C(\Pi)&\\&tw'\end{pmatrix}w\right]\ll \delta^{1/2-\eta}(h)\delta^{1/2-\eta}(t),$$
for any small enough $\eta>0$.
Using this bound and again using Lemma \ref{whittaker-bound}, for $t=ak=\diag(a_1,\dots,a_n)k$ in the Iwasawa coordinates we obtain
\begin{equation}\label{tilde-V-bound}
    \tilde{V}_{\Pi,W,h}(t)\prec_{N,\eta}\delta^{1/2-\eta}(h)|\det(a)|^{-\sigma}\delta^{1-\eta}({a})\prod_{j=1}^{n-1}\min(1,(a_j/a_{j+1})^{-N})S_{d_0}(W),
\end{equation}
for some fixed $d_0$

Now in the $b$ and $c$ integral in \eqref{main-equation} we integrate by parts several times. Also we use the  $\D$ to integrate by parts in the $u$ integral several times and apply Lemma \eqref{whittaker-bound} to estimate $\pi(u)W(h)$ for $u\asymp 1$. We write $h=a'k'=(a_1',\dots,a_n')k'\in A\times K$. From the support condition of $f^0$ we obtain that
\begin{multline}\label{integral-f-bound}
    X^{-n}\int_{\R^{n}}\int_{\R^n}\int_{G}\int_{\R^\times}f^0_X\left[z\begin{pmatrix}I&b\\&1\end{pmatrix}\begin{pmatrix}u&\\c/X&1\end{pmatrix}\right]\omega_\Pi(z)\\e(e_n^thb)
    W(hu)|\det(hu)|^{\sigma}e\left(\frac{C(\Pi)}{X}cw't^{-1}e_1\right)d^\times z\frac{du}{|\det(u)|}dbdc\\
    \ll_{N,d,\eta} |\det(a')|^\sigma\delta^{1/2-\eta}(a')\prod_{i=1}^{n-1}\min(1,(a_i'/a_{i+1}')^{-N})\min(1,{a'}_n^{-N})\\
    \min\left(1,(C(\Pi)/a_1X)^{-N}\right)S_{-d}(W),
\end{multline}
for all large $N$ and $d$.
Writing the $t$-integral in the RHS of \eqref{main-equation} in the $ak$ coordinates, and using \eqref{integral-f-bound} and \eqref{tilde-V-bound} we bound it by
\begin{multline*}
    \ll_{N,\eta}|\det(a')|^\sigma\delta^{1-\eta}(a')\prod_{i=1}^{n-1}\min(1,(a_i'/a_{i+1}')^{-N})\min(1,{a'}_n^{-N})\\
    S_{-d}(W)\int_{A}\min(1,(C(\Pi)/a_1X)^{-N})\prod_{j=1}^{n-1}\min(1,(a_j/a_{j+1})^{-N})|\det({a})|^{-\sigma}\frac{d{a}}{\delta^{\eta}(a)}.
\end{multline*}
The last integral, as in the proof of Proposition \ref{mainprop}, is bounded by 
$(C(\Pi)/X)^{-n\sigma}$.
Making $d$ large enough, using \eqref{bound-gamma-star} and Lemma \ref{sobolev-norm-property}, as in the proof of Lemma \ref{whittaker-bound}, we conclude that \eqref{main-equation} is bounded by 
\begin{equation}\label{bound-j-pi-f}
    \Pi(f^h_X)\We_\phi(1)\ll_{N,\eta}|\det(a')|^\sigma\delta^{1-\eta}(a')\prod_{i=1}^{n-1}\min(1,(a_i'/a_{i+1}')^{-N})\min(1,{a'}_n^{-N})(C(\Pi)/X)^{-n\sigma},
\end{equation}
uniformly in $\phi$.
Thus we have
$$\int_{A\times K}\Pi(f_X^h)\We_\phi(1)dh\ll(C(\Pi)/X)^{-n\sigma}\int_{A} \prod_{i=1}^{n-1}\min(1,{a'}_n^{-N})\min(1,(a_i'/a_{i+1}')^{-N})|\det(a')|^{\sigma}\frac{da'}{\delta^{\eta}(a')}.$$
The last integral, again as in the proof of Lemma \ref{whittaker-bound}, is convergent.
We conclude noting that if $\Pi$ is $\theta$-tempered then we can take $0<\sigma<1/2-\theta$.
\end{proof}

\subsection{Proof of the theorems}

In this subsection we assume the the representations have trivial central character and work with $K_0(X,\tau)$ only. One can modify the statements and proofs for general central character and work with $K_1(X,\tau)$ as in the proof of Theorem \ref{existence-weak} and Theorem \ref{existence-strong}.

Let $\Pi^\infty$ be the subspace of smooth vectors in $\Pi$. We define $\Pi^{-\infty}$ to be the space of distributional vectors which can be identified with a subspace of the algebraic dual of $\tilde{\Pi}^\infty$. This identification is induced by the natural pairing between $\Pi^\infty\times \tilde{\Pi}^\infty$ which is given by the integration in the Kirillov model as in \S\ref{sec:whittaker-kirillov}. Recall that $C^\infty_c(N\backslash G,\psi)\subset\Pi^\infty$ through the Kirillov model. Therefore, there is a natural surjection from $\Pi^{-\infty}$ to the continuous dual of $C_c^\infty(N\backslash G,\psi)$.

Let $\delta\in\Pi^{-\infty}$ be the distributional vector characterized by
$$( \delta,\tilde{\We}) := \tilde{\We}(1),$$
for any $\tilde{\We}\in\tilde{\Pi}^\infty$, where we $(,)$ denotes the pairing between ${\Pi}^{-\infty}$ and $\tilde{\Pi}^\infty$.
That is $\delta$ is the Dirac delta-mass at the identity coset in $(N\backslash G,\psi)$.
If $\Pi$ is unitary then we can identify $\bar{\Pi}$ with $\tilde{\Pi}$, and equivalently, define $\delta\in\Pi^{-\infty}$ by
$$\langle \delta, \We\rangle :=\overline{\We(1)},$$
There is a natural continuous map, e.g. by \cite[\S3.6, Lemma 2]{NV},
$$C_c^\infty(\GL_{n+1}(\R))\times \Pi^{-\infty}\to\Pi^{\infty}:\quad (f,v)\mapsto \Pi(f)v.$$
Hence, we may evaluate $\pi(f)\delta$ at the identity.

By the spectral expansion in $\Pi$ we obtain that
$$\Pi(f)\delta(1)=\sum_{\We\in\B(\Pi)}\langle\Pi(f)\delta,\We\rangle \We(1),$$
where $\B(\Pi)$ is an orthonormal basis consisting of smooth vectors. The above sum is absolutely convergent and does not depend on a choice of $\B(\Pi)$, see \cite[Appendix 4]{BM}.
Now we check that
$$\langle\Pi(f)\delta,\We\rangle = \langle \delta, \bar{\Pi}(\check{f}) \We\rangle = \bar{\Pi}(\check{f})\overline{\We(1)},$$
where $\check{f}(g):=f(g^{-1})$. We change variable $g\mapsto g^{-1}$, exchange sum and integral in the spectral sum above, and use unitarity to construct an orthonormal basis $\{\Pi(g^{-1})\We\}_{\We\in\B(\Pi)}$ to obtain that
\begin{equation}\label{limiting-operation-bessel-distribution}
\Pi(f)\delta(1)=\sum_{\We\in\B(\Pi)}\Pi(f)\We(1)\overline{\We(1)}.
\end{equation}
Note that the RHS of \eqref{limiting-operation-bessel-distribution} is also denoted by $J_\Pi(f)$, see \eqref{relative-trace}.

We recall from \eqref{trace} that
$$\Tr(\Pi(f)):=\sum_{\We\in\B(\Pi)}\langle \Pi(f)\We,\We\rangle=\sum_{\We\in\B(\Pi)}\int_{N\backslash G}\int_{\GL_{n+1}(\R)}f(g)\We\left[\begin{pmatrix}h\\&1\end{pmatrix}g\right]\overline{\We\left[\begin{pmatrix}h\\&1\end{pmatrix}\right]}dg dh.$$
We may $\B(\Pi)$ to be an orthonormal basis of $\Pi$ consisting of eigenfunctions of $\D$.
The above sums and integrals become absolutely convergent after we integrate by parts in the $g$-integral sufficiently many times by $\D$, see proof of Lemma \ref{sobolev-norm-property}.
We move the sum over $\B(\Pi)$ past the integral $N\backslash G$, change the basis $\B(\Pi)$ with $\Pi\left[\begin{pmatrix}h^{-1}\\&1\end{pmatrix}\right]\We$, and change variable $g\mapsto \begin{pmatrix}h^{-1}\\&1\end{pmatrix}g\begin{pmatrix}h\\&1\end{pmatrix}$ to obtain
$$\Tr(\Pi(f))=\int_{N\backslash G}\sum_{\We\in\B(\Pi)}\Pi(f^h)W(1)\overline{W(1)} dh$$
where, as before,
$$f^h(g):=f\left[\begin{pmatrix}h^{-1}\\&1\end{pmatrix}g\begin{pmatrix}h\\&1\end{pmatrix}\right].$$
Thus using \eqref{limiting-operation-bessel-distribution} we can write
\begin{equation}\label{trace-limit-operation}
    \Tr(\Pi(f))=\int_{N\backslash G}\Pi(f^h)\delta(1)dh.
\end{equation}
The above can also be interpreted as a limit of truncated integrals as
$$\Tr(\Pi(f))=\lim_{R\to\infty}\Pi\left(\int_{N\backslash G} f_R(h)f^hdh\right)\delta(1),$$
where $f_R$ is a smoothened characteristic function of the $R$-radius ball around the identity in $G$.

\begin{proof}[Proof of Theorem \ref{trace-estimates}]
Given any 
normalized majorant
$F_X$ of $K_0(X,\tau)$, one can approximate $F_X$ by $f_X$ defined in \eqref{char-func-new-sub} using e.g. Lemma \ref{folner-lemma}. Thus it is enough to prove Theorem \ref{trace-estimates} for $f_X$ replacing $F_X$.

We choose $\{\phi_j\}_j\in C^\infty_c(G)$ to be a sequence as in the statement of Proposition \ref{trace-estimates}. Then $\We_{\phi_j}$ is a sequence in the continuous dual of $C_c^\infty(N\backslash G,\psi)$ tending to $\delta$ mass in $C_c^\infty(N\backslash G,\psi)$ and their supports in $N\backslash G$ are shrinking to the identity coset as $j\to\infty$. By continuity of the association $(f,v)\mapsto \pi(f)v$ and the evaluation map $\We\mapsto \We(1)$ we have 
$$\lim_{j\to\infty}\Pi(f_X)\We_{\phi_j}(1)=J_\Pi(f_X),$$
from \eqref{limiting-operation-bessel-distribution},
and
$$\lim_{j\to \infty}\int_{A\times K}\Pi(f_X^h)\We_\phi(1)dh = \Tr(\Pi(f_X)).$$
Applying \eqref{bound-j-pi-f} for $h=1$ and Proposition \ref{estimating-projection-shrinking-support} we conclude the proof. 
\end{proof}

\begin{proof}[Proof of Theorem \ref{nonexistence}]
Let $\delta,X,\tau$, and $\Pi\ni v$ be as in the statement of Theorem \ref{nonexistence}. In particular, we have
$$\|\Pi(g)v-v\|_\Pi<\delta,\quad\forall g\in K_0(X,\tau).$$
Let $F_X$ be a
normalized majorant
of
$K_0(X,\tau)$.
Then
$$\Tr(\Pi(F_X))\ge \langle \Pi(F_X)v,v\rangle_\Pi=1+\langle \Pi(F_X)v-v,v\rangle_\Pi\ge 1-\delta,$$
where the last inequality is obtained by Cauchy--Schwarz. Then Theorem \ref{trace-estimates} implies that
$$C(\Pi)\ll_\tau X(1-\delta)^{1/n\sigma},$$
which completes the proof
\end{proof}

\begin{proof}[Proof of Theorem \ref{uniqueness}]
Fix $0<\delta<1$. Let $\tau$ and $\Pi$ be as in the statement of the Theorem \ref{uniqueness}. Let $X\asymp C(\Pi)$ and $F_X$ be 
a normalized majorant
of $K_0(X,\tau)$.
We define $\mathcal{V}_\delta$ by
$$\mathcal{V}_\delta:=\{v\in\Pi;\|v\|_\Pi=1\mid\|\Pi(g)v-v\|_\Pi<\delta, \forall g\in K_0(X,\tau)\}.$$
Let $\B(\Pi)$ be any orthonormal basis of $\Pi$. Then Theorem \ref{trace-estimates} implies that
$$\sum_{v\in\B(\Pi)\cap \mathcal{V}_\delta}\langle \Pi(F_X)v,v\rangle_\Pi\le \Tr(\Pi(F_X)) \ll_\tau 1.$$
But for each $v\in\B(\Pi)$ we have, as in the proof of Theorem \ref{nonexistence},
$$\langle \Pi(F_X)v,v\rangle_\Pi\ge 1-\delta.$$
Hence, cardinality $\mathcal{V}_\delta\cap \B(\Pi)$ must be $\ll_\tau (1-\delta)^{-1}$, which concludes the proof.
\end{proof}

\section{Proof of Theorem \ref{application}}\label{sec:proof-application}
Our proof uses a pre-Kuznetsov type formula on $\PGL_{n}(\R)$.
Let $f_X$ be a 
normalized majorant of $K_0(X,\tau)$ (see \eqref{defn-congruence-subgroup}) and sufficiently concentrated near $1$ (i.e., $\tau$ is sufficiently small). We recall that $f_X$ 
is assumed to take non-negative values. Let $F_X$ be the self-convolution of $f_X$ defined by 
$$F_X(g):=\int_{\PGL_{n}(\R)} f_X(h){f_X(gh)}dh.$$
Recall the conventions on $\tau$ as in the beginning of \S\ref{sec:reduction}. 
Lemma \ref{convolution-normalized-majorant} implies that $F_X$ is also 
a
normalized majorant of $K_0(X,\tau)$.
\begin{lemma}\label{convolution-normalized-majorant}
Suppose that $\tau_0 > 0$
is sufficiently small.
Let $F^1_X$ and $F^2_X$ be normalized majorants of $K_*(X,\tau_0)$. Then there is a $\tau_1$ such that the convolution $F^1_X*F^2_X$ is also a normalized majorant of $K_*(X,\tau_1)$.
\end{lemma}

\begin{proof}
Nonnegativity and property $(2)$ in Definition \ref{defn:normalized-majorant} of the convolution are obvious. Note that there exist $\tau_2,\tau_3>0$ so that for all $g\in K_*(X,\tau_0)$ we have
$$K_*(X,\tau_2)\subset gK_0(X,\tau_0)\subset K_*(X,\tau_3),$$
which can be seen from direct matrix multiplication.

We claim that for every $\tau_4>0$ there exists a sufficiently small $\tau_5>0$ such that if $h\in K_*(X,\tau_5)$ then $g\mapsto F_X^1(hg)$ satisfies $(1)$ and $(2)$ in Definition \ref{defn:normalized-majorant} for $K_*(X,\tau_4)$. Thus there exists $\tau_1>0$ such that $F_X^1*F_X^2$ satisfies $(1)$ in Definition \ref{defn:normalized-majorant} for $K_*(X,\tau_1)$

Finally we have,
$$\int_{\GL_n(\R)}\partial_a^\alpha\partial_b^\beta\partial_c^\gamma\partial_d^\delta F^1_X\left[d\begin{pmatrix}a&b\\&1\end{pmatrix}\begin{pmatrix}I_{n}&\\c&1\end{pmatrix}h\right]F_X^2(h)dh\ll X^{2A+|\gamma|+*\delta}\vol(K_*(X,\tau_1)).$$
This proves property $(3)$ of Definition \ref{defn:normalized-majorant}.
\end{proof}

We also check that $\mathrm{Vol}(K_0(X,\tau))\asymp X^{n-1}$ which implies that $\int_{N}F_X(x)\overline{\tilde{\psi}(x)}dx\ll X^{n-1}$.

In this section by $\pi$ we will denote an automorphic representation for $\PGL_n(\Z)$. We obtain a spectral decomposition of
\begin{equation}\label{spectral-decomposition}
    \sum_{\gamma\in \PGL_{n}(\Z)}F_X(x_1^{-1}\gamma x_2)=\int_{
    \pi
    }\sum_{\varphi\in \B(\pi)}\pi(F_X)\varphi(x_1)\overline{\varphi(x_2)}d\mu_{\mathrm{aut}}(\pi),
\end{equation}
where $\B(\pi)$ is an orthonormal basis of $\pi$
and
we integrate over $\pi$
in the automorphic spectrum of $\PGL_{n}(\Z)\backslash\PGL_{n}(\R)$
with respect to
the automorphic Plancherel measure
$d\mu_{\mathrm{aut}}$ 
(see \cite[Chapter $11.6$]{G} for details).

Let $\Gamma_\infty:=N\cap\PGL_{n}(\Z)$. It is easy to see that for a sufficiently small neighborhood
$\mathcal{U} \subset \PGL_{n}(\R)$ around the identity,
\begin{equation}\label{using-support}
    N\mathcal{U} N\cap \PGL_{n}(\Z)=\Gamma_\infty.
\end{equation}

Recall $\psi$, which is an additive character of $N$, defined in \eqref{add-char-of-n}). Let $W_\varphi$ be the Whittaker function of $\varphi$, i.e.
$$\int_{\Gamma_\infty\backslash N}\varphi(xg)\overline{{\psi}(x)}dx=W_\varphi(g).$$
We note, by Schur's lemma, that there exists a positive constant $\ell(\pi)$ such that
\begin{equation}\label{defn-c-pi}
    \|\varphi\|^2_{\pi}=\ell(\pi)\|W_\varphi\|^2_{\W(\pi,\psi)}.
\end{equation}
We note that when $\pi$ is cuspidal then $\ell(\pi)\asymp L(1,\pi,\Ad)$ where the underlying constant in $\asymp$ is absolute (coming from the residue of a maximal Eisenstein series at $1$, see \cite[p. $617$]{B3}).

We recall the Bessel distribution on a generic representation $\pi$ from  
\eqref{relative-trace}:
\begin{equation}\label{diff-basis}
    J_\pi(F_X):=
    \sum_{\V\in \B(\pi)}\pi(F_X)\V(1)\overline{\V(1)}
    =
    \sum_{\V\in \B(\pi)}|\pi(f_X)\V(1)|^2.
\end{equation}
For non-generic $\pi$ we define $J_\pi(F_X)$ to be identically zero.
We fix a basis $\B(\pi)$ containing a Whittaker newvector $\We$, i.e. $\We$ as in \eqref{newvec}.

We now take 
$x_i\in \Gamma_\infty\backslash N$
and multiply both sides of \eqref{spectral-decomposition} by $\psi(x_2)\overline{\psi(x_1)}$ and integrate with respect to $x_i$. 
Thus we automatically get rid of the non-generic part of the spectrum. Using \eqref{using-support}, and \eqref{diff-basis} we rewrite $\eqref{spectral-decomposition}$ as
\begin{align*}
    \int_{\pi}J_\pi(F_X)\ell(\pi)^{-1}d\mu_{\mathrm{aut}}(\pi)
    &=\sum_{\gamma\in \Gamma_\infty}\int_{\Gamma_\infty\backslash N}\int_{\Gamma_\infty\backslash N}F_X(x_1^{-1}\gamma x_2)\psi(x_2-x_1)dx_1dx_2\\
    &=\int_{\Gamma_\infty\backslash N}\int_{N}F_X(x_1^{-1}x_2)\psi(x_2-x_1)dx_1dx_2\\
    &=\int_{N}F_X(x)\overline{\psi(x)}\ll X^{n-1}.
\end{align*}

Note that $J_\pi(F_X)$ is non-negative on the automorphic spectrum.
We first drop everything other than the cuspidal spectrum from the LHS of the above. It is known that (see \cite{MS}) that cuspidal automorphic representations are $\theta$-tempered for $0\le \theta\le \frac{1}{2}-\frac{1}{1+n^2}$. From Theorem \ref{existence-strong} we obtain that if $\pi$ is cuspidal with $C(\pi)<X$ then
$$\pi(f_X)\V(1)\gg 1.$$
Hence $J_\pi(F_X)\gg 1$ for cuspidal $\pi$ with $C(\pi)<X$, and we conclude.

\end{document}